\documentclass[twoside,11pt,reqno]{amsart}
\usepackage[margin=1.1in]{geometry}
\usepackage{amsmath} 
\usepackage{amssymb} 
\usepackage{amsopn}
\usepackage{amsthm} 
\usepackage{amsfonts} 
\usepackage{mathrsfs}
\usepackage{stmaryrd}
\usepackage{colonequals}
\usepackage{tikz-cd}
\usepackage{enumitem}
\usepackage{slashbox}
\usepackage{float}

\usepackage[french,english]{babel}

\usepackage{cite}

\usepackage{verbatim}

\usepackage[pdfpagelabels,pdftex]{hyperref}
\usepackage{setspace}
\usepackage{amssymb}
\usepackage{euscript}
\usepackage{cite}
\usepackage[all]{xy}
\usepackage{tabularx}

\theoremstyle{plain}
\usepackage{amsfonts}
\usepackage{graphicx}
\usepackage{amsmath}
\hypersetup{
  colorlinks   = true, %Colours links instead of ugly boxes
  urlcolor     = blue, %Colour for external hyperlinks
  linkcolor    = blue, %Colour of internal links
  citecolor   = green, %Colour of citations
  pdftitle={},
  pdfauthor={},
  pdfsubject={},
  pdfkeywords={},
  breaklinks=true,
  bookmarksopen=true,
  bookmarksnumbered=true,
  pdfpagemode=UseOutlines,
  plainpages=false}

\makeatletter
\newcommand{\colim@}[2]{%
  \vtop{\m@th\ialign{##\cr
    \hfil$#1\operator@font lim$\hfil\cr
    \noalign{\nointerlineskip\kern1.5\ex@}#2\cr
    \noalign{\nointerlineskip\kern-\ex@}\cr}}%
}
\newcommand{\colim}{%
  \mathop{\mathpalette\colim@{\rightarrowfill@\textstyle}}\nmlimits@
}
\newcommand{\prolim@}[2]{%
  \vtop{\m@th\ialign{##\cr
    \hfil$#1\operator@font lim$\hfil\cr
    \noalign{\nointerlineskip\kern1.5\ex@}#2\cr
    \noalign{\nointerlineskip\kern-\ex@}\cr}}%
}
\newcommand{\prolim}{%
  \mathop{\mathpalette\colim@{\leftarrowfill@\textstyle}}\nmlimits@
}
\makeatother

\newcommand{\bA}{{\mathbb A}}

\newcommand{\bF}{{\mathbb F}}
\newcommand{\bG}{{\mathbb G}}

\newcommand{\bP}{{\mathbb P}}
\newcommand{\bQ}{{\mathbb Q}}
\newcommand{\bR}{{\mathbb R}}

\newcommand{\bW}{{\mathbb W}}

\newcommand{\bZ}{{\mathbb Z}}

\newcommand{\cF}{{\mathscr F}}
\newcommand{\cG}{{\mathscr G}}

\newcommand{\cM}{{\mathscr M}}
\newcommand{\cN}{{\mathscr N}}

\newcommand{\caG}{{\mathcal G}}

\newcommand{\caO}{{\mathcal O}}

\newcommand{\caS}{{\mathcal S}}

\newcommand{\caU}{{\mathcal U}}

\newcommand{\caW}{{\mathcal W}}
\newcommand{\caX}{{\mathcal X}}

\newcommand{\fP}{{\mathfrak P}}

\newcommand{\fX}{{\mathfrak X}}

\newcommand{\fa}{{\mathfrak a}}
\newcommand{\fb}{{\mathfrak b}}

\newcommand{\ff}{{\mathfrak f}}

\newcommand{\fm}{{\mathfrak m}}

\newcommand{\fp}{{\mathfrak p}}
\newcommand{\fq}{{\mathfrak q}}

\DeclareMathOperator{\id}{id}
\newcommand{\ev}{{\rm ev}}

\DeclareMathOperator{\im}{im}

\DeclareMathOperator{\GL}{GL}

\DeclareMathOperator{\Perf}{Perf}
\newcommand{\Perff}{\Perf_{\overline{\bF}_q}}

\DeclareMathOperator{\Spec}{Spec}

\newcommand{\obF}{\overline{\bF}_q}

\usepackage{mathtools}
\DeclarePairedDelimiterX{\norm}[1]{\lVert}{\rVert}{#1}

\DeclareMathOperator{\Spa}{Spa}

\newcommand{\cd}{{\rm cd}}

\DeclareMathOperator{\Gal}{Gal}

\newcommand{\sep}{{\rm sep}}

\newcommand{\nr}{{\rm nr}}

\DeclareMathOperator{\nil}{{\rm nil}}

\newcommand\reallywidehat[1]{%
\savestack{\tmpbox}{\stretchto{%
  \scaleto{%
    \scalerel*[\widthof{\ensuremath{#1}}]{\kern-.6pt\bigwedge\kern-.6pt}%
    {\rule[-\textheight/2]{1ex}{\textheight}}%WIDTH-LIMITED BIG WEDGE
  }{\textheight}% 
}{0.5ex}}%
\stackon[1pt]{#1}{\tmpbox}%
}

\makeatletter
\newtheorem*{rep@theorem}{\rep@title}
\newcommand{\newreptheorem}[2]{%
\newenvironment{rep#1}[1]{%
 \def\rep@title{#2 \ref{##1}}%
 \begin{rep@theorem}}%
 {\end{rep@theorem}}}
\makeatother

\newtheorem{thm}{Theorem}[section]

\newtheorem{prop}[thm]{Proposition}

\newtheorem{cor}[thm]{Corollary}
\newtheorem*{introCor}{Corollary}

\newtheorem{lm}[thm]{Lemma}

\newtheorem{conj}[thm]{Conjecture}
\newtheorem*{conj*}{Conjecture}

\newreptheorem{thm}{Theorem}
\newreptheorem{prop}{Proposition}
\newreptheorem{cor}{Corollary}
\newtheorem*{thma}{Theorem A}
\newtheorem*{thmb}{Theorem B}
\newtheorem*{thmc}{Theorem C}

\theoremstyle{definition}
\newtheorem{Def}[thm]{Definition}

\newtheorem{rem}[thm]{Remark}

\newtheorem{ex}[thm]{Example}

\newenvironment{pro*}[1][Proof]{{\it{#1:}} }{}

\newcommand\rar{ \rightarrow }
\newcommand\tar{ \twoheadrightarrow }
\newcommand\har{ \hookrightarrow }
\newcommand\Rar{ \Rightarrow }
\newcommand\LRar{ \Leftrightarrow }
\newcommand\longrar{\longrightarrow}

\newcommand\ord{\mathop{\rm ord}}

\newcommand\charac{\mathop{ \rm char}}
\newcommand\dirlim{\mathop{\underrightarrow{\lim} }}
\newcommand{\sm}{{\,\smallsetminus\,}}

\DeclareMathOperator\supp{supp}

\newcommand{\diag}{{\rm diag}}
\DeclareMathOperator\Lang{Lang}

\makeatletter
\newcommand*{\doublerightarrow}[2]{\mathrel{
  \settowidth{\@tempdima}{$\scriptstyle#1$}
  \settowidth{\@tempdimb}{$\scriptstyle#2$}
  \ifdim\@tempdimb>\@tempdima \@tempdima=\@tempdimb\fi
  \mathop{\vcenter{
    \offinterlineskip\ialign{\hbox to\dimexpr\@tempdima+1em{##}\cr
    \rightarrowfill\cr\noalign{\kern.5ex}
    \rightarrowfill\cr}}}\limits^{\!#1}_{\!#2}}}
\newcommand*{\triplerightarrow}[1]{\mathrel{
  \settowidth{\@tempdima}{$\scriptstyle#1$}
  \mathop{\vcenter{
    \offinterlineskip\ialign{\hbox to\dimexpr\@tempdima+1em{##}\cr
    \rightarrowfill\cr\noalign{\kern.5ex}
    \rightarrowfill\cr\noalign{\kern.5ex}
    \rightarrowfill\cr}}}\limits^{\!#1}}}
\makeatother

\newcounter{absatzcounter}[section]
\numberwithin{equation}{section}

\numberwithin{equation}{section}

\begin{document}

\title{Arc-descent for the perfect loop functor and $p$-adic Deligne--Lusztig spaces}
\author{Alexander B. Ivanov}
\address{Mathematisches Institut \\ Universit\"at Bonn \\ Endenicher Allee 60 \\ 53115 Bonn, Germany}
\email{ivanov@math.uni-bonn.de}
\maketitle

\makeatletter
\newenvironment{abstracts}{%
  \ifx\maketitle\relax
    \ClassWarning{\@classname}{Abstract should precede
      \protect\maketitle\space in AMS document classes; reported}%
  \fi
  \global\setbox\abstractbox=\vtop \bgroup
    \normalfont\Small
    \list{}{\labelwidth\z@
      \leftmargin3pc \rightmargin\leftmargin
      \listparindent\normalparindent \itemindent\z@
      \parsep\z@ \@plus\p@
      \let\fullwidthdisplay\relax
      \itemsep\medskipamount
    }%
}{%
  \endlist\egroup
  \ifx\@setabstract\relax \@setabstracta \fi
}

\newcommand{\abstractin}[1]{%
  \otherlanguage{#1}%
  \item[\hskip\labelsep\scshape\abstractname.]%
}
\makeatother

\begin{abstracts}
\abstractin{english}
We prove that the perfect loop functor $LX$ of a quasi-projective scheme $X$ over a local non-archimedean field $k$ satisfies arc-descent, strengthening a result of Drinfeld. Then we prove that for an unramified reductive group $G$, the map $LG \rar L(G/B)$ is a $v$-surjection. This gives a mixed characteristic version (for $v$-topology) of an equal characteristic result (in \'etale topology) of Bouthier--\v{C}esnavi\v{c}ius.

In the second part of the article, we use the above results to introduce a well-behaved notion of Deligne--Lusztig spaces $X_w(b)$ attached to unramified $p$-adic reductive groups. We show that in various cases these sheaves are ind-representable, thus partially solving a question of Boyarchenko. Finally, we show that the natural covering spaces $\dot X_{\dot w}(b)$ are pro-\'etale torsors over clopen subsets of $X_w(b)$, and analyze some examples.
\end{abstracts}

\tableofcontents
\vspace{-6ex}
\section{Introduction}\label{sec:intro}

This paper has two parts. In the first we analyze the behavior of vector bundles over the fraction field of the Witt vectors in the arc- and $v$-topologies on perfect schemes, and deduce consequences for the (perfect) loop space of a scheme over a non-archimedean local field. In the second part we use the above to introduce a well-behaved notion of Deligne--Lusztig spaces attached to $p$-adic reductive groups and to establish various properties of these spaces.

\subsection*{Loop spaces and arc-topology} For a $\bQ_p$-scheme\footnote{For simplicity of notation we work with $\bQ_p$ in this introduction, whereas in the main body of the article everything is done for arbitrary non-archimedean local field.} $X$, the perfect loop space of $X$ is the set-valued functor $LX \colon \Perf \rar {\rm Sets}$ on the category of perfect $\bF_p$-algebras, which sends $R$ to $LX(R) = X(W(R)[p^{-1}])$, where $W(R)$ is the ring of $p$-typical Witt-vectors of $R$. A result of Drinfeld \cite{Drinfeld_18} implies that $LX$ is a sheaf for the fpqc-topology on $\Perf$. For quasi-projective $X$ we prove the following strengthening.

\begin{thma}[Theorem \ref{thm:LX_is_vsheaf}] \hypertarget{thm:A}
Let $X/\bQ_p$ be quasi-projective. Then $LX$ is an arc-sheaf.
\end{thma}

The arc-topology on schemes was studied in \cite{BhattM_18} and independently in \cite{Rydh_submersions_II}. A map of qcqs schemes $S' \rar S$ is an arc-cover if any immediate specialization in $S$ lifts to $S'$. Restricted to perfect $\bF_p$-schemes, the arc-topology agrees with the canonical topology on this category \cite[Thm.~5.16]{BhattM_18}. We prove Theorem \hyperlink{thm:A}{A} for $\bP^n_{\bQ_p}$ by showing, using perfectoid techniques from \cite{ScholzeW_20}, that vector bundles on $W(R)[p^{-1}]$ form an arc-stack in $R$. Then we deduce it for quasi-projective schemes by analyzing the effect of open and closed immersions.

Locally in the arc-topology (and even in the $v$-topology) every qcqs scheme admits a covering by the spectrum of a product of valuation rings. We show, extending some results of Kedlaya \cite{Kedlaya_16_ringAinf}, that all vector bundles over such affine schemes are trivial.

\begin{thmb}[Theorem \ref{lm:no_torsors}]
Let $A \in \Perf$ be such that any connected component of $\Spec A$ is the spectrum of a valuation ring. Then all finite locally free $W(A)[p^{-1}]$-modules of constant rank are free. 
\end{thmb}

In the equal characteristic case, a related result (for more general $A$, but only for modules of rank $1$) was shown for the \'etale topology by Bouthier--\v{C}esnavi\v{c}ius \cite[Cor.~3.1.5]{BouthierC_19}. Theorem \hyperlink{thm:B}{B} has the following consequence for reductive groups over $\bQ_p$.

\begin{introCor}[Corollary \ref{thm:arc_exactness_and_splitting}]
Let $G/\bQ_p$ be an unramified reductive group, and let $B \subseteq G$ be a $\bQ_p$-rational Borel subgroup. Then $LG \rar L(G/B)$ is surjective in the $v$-topology.
\end{introCor}

Applying Theorem \hyperlink{thm:A}{A} above, Ansch\"utz generalized this corollary recently to the case of arbitrary parabolic subgroups of $G$\cite[Cor.~11.5]{Anschuetz_20}.

\subsection*{$p$-adic Deligne--Lusztig spaces} 
Classical Deligne--Lusztig theory \cite{DeligneL_76} studies families of varieties attached to a reductive group $G$ over a finite field $\bF_q$. Their $\ell$-adic cohomology contains essentially complete information about the representation theory of the finite Chevalley group $G(\bF_q)$. Forty years ago Lusztig suggested the existence of a similar theory over $p$-adic fields \cite{Lusztig_79}. Since then, several related constructions were studied by many people, see in particular \cite{Lusztig_04, Stasinski_09, Boyarchenko_12, BoyarchenkoW_16, Ivanov_15_ADLV_GL2_unram, ChenS_17, Chan_siDL, CI_ADLV, CI_loopGLn}. Nevertheless, until now a satisfactory formalism of ``$p$-adic Deligne--Lusztig varieties'' did not exist; let alone a suitably general definition, it was not even clear in which category they should live. 

Our definition works as follows. Let $G_0$ be an unramified reductive group over $\bQ_p$, and denote by $G$ its base change to $\breve \bQ_p$, the completion of a maximal unramified extension of $\bQ_p$. Fix a maximal torus and a Borel subgroup $T \subseteq B \subseteq G$, both defined over $\bQ_p$. Let $W$ be the Weyl group of $T$ in $G$ and for $w \in W$, let $\caO(w) \subseteq (G/B)^2$ denote the $G$-orbit corresponding to $w$ by the Bruhat decomposition. For $w \in W$, $b \in G(\breve \bQ_p)$ define (Definition \ref{def:Xwb_and_covers}) the \emph{$p$-adic Deligne--Lusztig space} $X_w(b)$ by the Cartesian diagram of functors on $\Perf_{\overline{\bF}_p}$,
\[
\xymatrix{
X_w(b) \ar[r] \ar[d] & L\caO(w)\ar[d] \\ L(G/B) \ar[r]^-{(\id, b\sigma)} \ar[r] & L(G/B) \times L(G/B)
}
\]
where the lower horizontal arrow is the graph of the geometric Frobenius morphism\footnote{Note that $L(G/B)$ lives in characteristic $p$, hence admits a geometric Frobenius. In contrast, the Frobenius action on $(G/B)(\breve \bQ_p)$ surely does not extend to a scheme morphism over $\breve \bQ_p$. This is a crucial difference with the classical case, where $G/B$ itself lives over $\overline\bF_p$ and admits a geometric Frobenius.} of $L(G/B)$ composed with left multiplication by $b$. Similarly, for a lift $\dot w \in G(\breve \bQ_p)$ of $w$, one has a functor $\dot X_{\dot w}(b)$ equipped with a map $\dot X_{\dot w}(b) \rar X_w(b)$.
By Theorem \hyperlink{thm:A}{A}, $X_w(b)$, $\dot X_{\dot w}(b)$ are arc-sheaves. 

Let $G_b$ be the functorial $\sigma$-centralizer of $b$, cf. \S\ref{sec:bsigma_fixed_points}. It is a $\bQ_p$-group, isomorphic to an inner form of a Levi subgroup of $G$. The locally profinite group $G_b(\bQ_p)$ acts (continuously) on $X_w(b)$. Similarly, there is an outer unramified form $T_w$ of $T$, such that $G_b(\bQ_p) \times T_w(\bQ_p)$ acts on $\dot X_{\dot w}(b)$.

In \cite{CI_loopGLn,CI_DrinfeldStrat} the spaces $\dot X_{\dot w}(b)$ for $G = \GL_n$, $w$ Coxeter and $b$ basic (i.e., $G_b$ is an inner form of $G$) were studied in detail, and it was shown that their $\ell$-adic cohomology partially realizes local Langlands and Jacquet--Langlands correspondences. This gives the hope that the spaces $X_w(b)$ for general $G$ allow an elegant geometric construction of a big portion of smooth $G(\bQ_p)$-representations in nicely organized families, and shed new light on local Langlands and Jacquet--Langlands correspondences. 

\subsection*{Properties of $X_w(b)$} In the classical theory, if $w\in W$ is contained in a parabolic subgroup $B \subseteq P \subseteq G$, the Deligne--Lusztig variety $X_w$ admits a certain disjoint union decomposition indexed over $G(\bF_q)/P(\bF_q)$ \cite[\S3]{Lusztig_76_Fin}. In Theorem \ref{cor:disjoint_dec_of_Xwb_complete} we prove an analogous decomposition for the $p$-adic spaces $X_w(b)$. This is more complicated in various respects: we have to make use of Theorem \hyperlink{thm:A}{A},  the additional parameter $b$ appears and there are some non-vanishing Galois cohomology groups. As a consequence, we deduce a sufficient criterion for emptyness of $X_w(b)$ (note that classical Deligne--Lusztig varieties are always non-empty).

\begin{introCor}[Corollary \ref{cor:emptyness_Xwb_support}] If $b$ is not $\sigma$-conjugate to any element of the smallest $\bQ_p$-rational parabolic subgroup $P \subseteq G$ containing $w$, then $X_w(b) = \varnothing$. 
\end{introCor}

For a $\sigma$-conjugacy class $C \subseteq W$, let $C_{\rm min}$ denote the set of elements of minimal length. Combining Theorem \ref{cor:disjoint_dec_of_Xwb_complete} with Frobenius-cyclic shifts, we deduce the following. 

\begin{introCor}[Corollary \ref{cor:isomorphism_class_depends_only_on_Cmin}]
Let $b \in G(\breve \bQ_p)$ and let $C$ be a $\sigma$-conjugacy class in $W$. All $X_w(b)$ for $w$ varying through $C_{\rm min}$ are mutually $\underline{G_b(\bQ_p)}$-equivariantly isomorphic.
\end{introCor}

Note that for classical Deligne--Lusztig varieties, one only has universal homeomorphisms. As we work over $\Perf$, all Frobenii are invertible, which is the reason why we get isomorphisms.

\smallskip

Next we partially answer a question of Boyarchenko, whether $p$-adic Deligne--Lusztig spaces exist as ind-schemes \cite[Problem 1]{Boyarchenko_12}. 

\begin{thmc}[Corollary \ref{thm:ind_rep_min_length}] \hypertarget{thm:C}
Let $w \in W$ be of minimal length in its $\sigma$-conjugacy class. Then for all $b \in G(\breve \bQ_p)$ and all lifts $\dot w$ of $w$, $X_w(b)$, $\dot X_{\dot w}(b)$ are ind-representable.
\end{thmc}

In particular, this theorem shows that $X_w(b)$ are reasonable geometric objects (in contrast to the ambient space $L(G/B)$, which is not truly of geometric nature, cf. Remark \ref{rem:LPn}). In the proof of Theorem \hyperlink{thmc}{C}, we closely follow the strategy of Bonnaf\'e--Rouquier \cite{BonnafeR_08}, who gave a new proof of a theorem due to Orlik--Rapoport \cite[\S5]{OrlikR_08} and He \cite[Thm.~1.3]{He_07_aff}, stating that a classical Deligne--Lusztig variety $X_w$ is affine if $w\in W$ has minimal length in its $\sigma$-conjugacy class. 
The proof shows ind-representability of $X_w(b)$ also for other $w \in W$ (Theorem \ref{thm:ind_rep_via_Braid}), including the longest elements of all parabolic subgroups of $W$.

Interestingly, it turns out that $X_w(b)$ is quite often not representable by a scheme. Namely, we show in Theorem \ref{thm:non_rep} that if $C \subseteq W$ is a $\sigma$-conjugacy class, such that $X_w(b) \neq \varnothing$ for $w\in C_{\rm min}$, then $X_w(b)$ is not representable by a scheme for all $w\in C \sm C_{\rm min}$.

\smallskip

In the classical theory, the maps $\dot X_{\dot w} \rar X_w$ are finite \'etale $T_w(\bF_q)$-torsors. By contrast, in our situation $\dot X_{\dot w}(b) \rar X_w(b)$ is in general not surjective. We will show in \S\ref{sec:Torsors} that there is a disjoint decomposition $X_w(b) = \coprod_{\bar w} X_w(b)_{\bar w}$, such that for certain $\bar w$ attached to the lift $\dot w$, $\dot X_{\dot w}(b) \rar X_w(b)_{\bar w}$ is a pro-\'etale $T_w(\bQ_p)$-torsor (and all $\dot X_{\dot w}(b)$ lying over the same $X_w(b)_{\bar w}$ are equivariantly isomorphic). In fact, by Theorem \hyperlink{thm:B}{B} it will at least be a $v$-torsor; to deduce that it is even a pro-\'etale torsor we will need a descent result of Gabber. At least in particular cases the discrepancy with the classical theory can be explained by the difference between rational and stable conjugacy classes of maximal  tori in $p$-adic groups \cite[Cor.~4.7]{Ivanov_pDL_Cox_1}.

\smallskip

Finally, we state the following conjecture (spelled out by Chan and the author), which is a variant of what Lusztig conjectured in \cite[p.~171]{Lusztig_79} concerning the $p$-adic Deligne--Lusztig sets attached to anisotropic tori defined there.

\begin{conj}
If $w$ is Coxeter, then $X_w(b)$ is representable by a perfect scheme. 
\end{conj}

Evidence is provided by \cite[Prop.~2.6]{CI_loopGLn} and the examples below. Many cases of this conjecture are shown in \cite{Ivanov_pDL_Cox_1}. We note that the proofs in all these cases in fact lead to a concrete description of $X_w(b)$ in terms of more accessible objects. More optimistically, one might conjecture that $X_w(b)$ is a scheme, whenever $w$ is of minimal length in its $\sigma$-conjugacy class. 

Finally, let us note that Lusztig's set $X = \{g \in G(\breve \bQ_p) \colon g^{-1} \dot w\sigma(g)\dot w^{-1} \in U(\breve \bQ_p)\} / U(\breve \bQ_p) \cap {}^{w^{-1}}U(\breve \bQ_p)$ from \cite[p.~171]{Lusztig_79} (here $U \subseteq B$ is the unipotent radical and $\dot w \in G(\breve \bQ_p)$ is a lift of $w \in W$; thus $F$ from \emph{loc.~cit.} is ${\rm Ad}(\dot w)\circ\sigma$ in our notation) should not be expected to carry a natural structure of a scheme over $\overline \bF_p$ if $w$ is $\sigma$-conjugate to a Coxeter element, but not itself Coxeter. Indeed, by Theorems \ref{thm:non_rep} and \ref{thm:ind_rep_via_Braid} one should rather expect an ind-(perfect scheme).

\subsection*{Case $G_0 = \GL_2$} 

The spaces $X_w(b)$ and $\dot X_{\dot w}(b)$, despite of containing interesting representation-theoretic information, are of rather explicit nature, and can often be described by explicit equations. We discuss various examples in \S\ref{sec:examples}. Let us describe $X_w(b)$ for $G_0 = \GL_2$ here. All of these spaces are schemes.

\begin{table}[H]
\begin{tabular}{|l||*{3}{c|}}\hline \label{tab:1}
\backslashbox{$w$}{$b$} &{$p^{(c,c)}$} ($c \in \bZ$) &{$\left(\begin{smallmatrix} 0 & p^c \\ p^{c+1} & 0 \end{smallmatrix}\right)$} & {$p^{(c, d)}$ ($c>d$)} \\\hline\hline
$1$ & $\underline{\mathbb{P}^1(\bQ_p)}$ &  $\varnothing$ & $\{0,\infty \}$ \\\hline
 $w_0$  & $\coprod\limits_{G_0(\bQ_p)/ZG_0(\bZ_p)} L^+\Omega_{\breve\bZ_p}^1$ & $L^+\bA^1_{\breve\bZ_p} \sqcup L^+\bA^1_{\breve\bZ_p}$ & $L\bG_m$ \\\hline
\end{tabular}
\vspace{1ex}
\caption{$X_w(b)$ for $\GL_2$}
\end{table}

Let $w_0 \in W$ be the longest element, $L^+$ the positive loop functor, $Z \subseteq G_0(\bQ_p)$ the center, and let $\Omega^1_{\caO_{\breve k}} = \Spec \caO_k[T]_{T-T^p}$. Then all essentially different possibilities for $X_w(b)$ for $\GL_2$ are listed in Table \ref{tab:1}. Here $\underline{\mathbb{P}^1(\bQ_p)}$ is the totally disconnected $\overline \bF_p$-scheme, whose underlying topological space is $\bP^1(\bQ_p)$ equipped with $p$-adic topology.

\subsection*{Acknowledgements} The author wants to thank Peter Scholze for numerous very helpful advices concerning this article. In particular, Definition \ref{def:Xwb_and_covers} was suggested by him. The author wants to thank Charlotte Chan, with whom he initially started to work on Lusztig's conjecture. Also he wants to thank Johannes Ansch\"utz for enlightening discussions, David Rydh and Christian Kaiser for helpful remarks, and an anonymous referee for useful suggestions. The author was supported by the DFG via the Leibniz Prize of Peter Scholze.

\section{Notation and preliminaries}

\subsection{Notation}\label{sec:gen_not} Fix a prime number $p$ and denote by $\Perf$ the category of perfect rings of characteristic $p$. For $R \in \Perf$, denote by $\Perf_R$ the category of perfect $R$-algebras.

\subsubsection{Setup}\label{sec:setup_general} Fix a field $\kappa \in \Perf$. For $R \in \Perf_\kappa$ denote by $W(R)$ the ($p$-typical) Witt-vectors of $R$. We work simultaneously in two cases. Therefore we let $\caO_{k_0}$ be either $W(\kappa)$ or $\kappa[\![t]\!]$. In the first resp. second case we say that we work in \emph{mixed} resp. \emph{equal characteristic case}. We also set $k_0 = {\rm Frac}(\caO_{k_0})$, i.e., $k_0$ is either $W(\kappa)[1/p]$ or $\kappa(\!(t)\!)$. 

We fix a finite totally ramified extension $k$ of $k_0$, and we denote by $\varpi$ a uniformizer of $k$, and by $\caO_k$ the integers of $k$. We will indicate in which case we are by writing $\charac k = 0$ resp. $\charac k = p$ in the mixed resp. equal characteristic case. For $R \in \Perf_\kappa$, there is an essentially unique $\varpi$-adically complete and separated $\caO_k$-algebra $\bW(R)$, in which $\varpi$ is not a zero-divisor and which satisfies $\bW(R)/ \varpi\bW(R) = R$. More explicitly 
\[
\bW(R) := \begin{cases}  W(R) \otimes_{W(\kappa)} \caO_k &\text{if $\charac k = 0$} \\ R[\![\varpi]\!] &\text{if $\charac k = p$,}\end{cases}
\]
i.e., in the first case $\bW(R)$ are the ramified Witt vectors, details on which can be found for example in \cite[1.2]{FarguesFontaine_book}. In particular, $\bW(\kappa)[1/\varpi] = k$. If $\bar{\kappa}$ is an algebraic closure of $\kappa$, then we put $\caO_{\breve k} = \bW(\bar\kappa)$ and $\breve k = \bW(\bar\kappa)[1/\varpi]$. This is the $\varpi$-adic completion of a maximal unramified extension of $k$. 

We have a multiplicative map $[\cdot] \colon R \rar \bW(R)$, which is the Teichm\"uller lift if $\charac k = 0$, and the natural embedding otherwise. Slightly abusing terminology, we call $[\cdot]$ the Teichm\"uller lift in both cases. It is canonical and, in particular, independent of the choice of the uniformizer $\varpi$ and functorial in $R$. Moreover, every element of $\bW(R)$ can uniquely be written as a convergent sum $\sum_{i = 0}^\infty [a_i]\varpi^i$ with $a_i \in R$ (if $\charac k = 0$, this uses that $R$ is perfect).

For $R \in \Perf$ we denote by ${\rm Sch}_R$ the category of perfect quasi-compact and quasi-separated (= qcqs) schemes over $R$. For generalities on perfect schemes we refer to \cite{Zhu_17,BhattS_17}.  The functor $\bW(\cdot)$ extends to all of ${\rm Sch}_\kappa$. It takes values in $\varpi$-adic formal schemes.

By a presheaf on $\Perf_R$ we mean a contravariant set-valued functor on $\Perf_R$. If $F$ is a presheaf on $\Perf_R$, and $R' \in \Perf_R$, we sometimes write $F(\Spec R')$ for $F(R')$. Using Yoneda's lemma we regard ${\rm Sch}_R$ as a full subcategory of all presheaves on $\Perf_R$.

\subsubsection{Setup over a finite field}\label{sec:setup_over_finite_field} Our main application concerns the case when $\kappa = \bF_q$ is a finite field with $q$ elements. Then $k$ is a local non-archimedean field and ${\rm Aut}_{\text{cont}}(\breve k/k) \cong \Gal(\obF/\bF_q)$ is topologically generated by the Frobenius automorphism, which we denote by $\sigma$, and which induces the automorphism $x \mapsto x^q$ of $\obF$.

Any $R \in \Perf_{\bF_q}$ possesses the $\bF_q$-linear Frobenius automorphism $x \mapsto x^q$. For any presheaf $\cF_0$ on $\Perf_{\bF_q}$ this induces an automorphism $\sigma_{\cF_0} \colon \cF_0 \rar \cF_0$. Let $\cF = \cF_0 \times_{\Spec \bF_q} \Spec \obF$ be the corresponding presheaf on $\Perf_{\obF}$. We have the \emph{geometric Frobenius} automorphism $\sigma_{\cF} := \sigma_{\cF_0} \times \id$ of $\cF$. 
If $\cF$ is clear from the context, we also write $\sigma$ for $\sigma_{\cF}$.

\subsubsection{Ind-schemes}
Let $R \in \Perf$. An \emph{ind-(perfect scheme)} over $R$ is a functor on $\Perf_R$, which is isomorphic to an inductive limit of perfect schemes $(X_\alpha)_{\alpha \in \bZ_{\geq 0}}$, such that all transition maps $X_\alpha \rar X_{\alpha + 1}$ are closed immersions\footnote{Sometimes in the literature these ind-schemes are called \emph{strict}, whereas the term ``ind-scheme'' is reserved for those $\colim_\alpha X_\alpha$, with the assumption on the transition maps dropped.}. Any perfect scheme is in particular a scheme, and the same holds for ind-(perfect schemes). Therefore we will simply speak of schemes resp. ind-schemes instead of perfect schemes resp. ind-(perfect schemes). Nevertheless, the reader should keep in mind that throughout the article we work only with perfect objects.

\subsubsection{Further notation and conventions}\label{sec:further_not_conv}

For a field $F$ we denote by $F^{\rm sep}$ its separable closure. For a scheme $X$ we denote by $|X|$ its underlying topological space. We abbreviate ``quasi-compact and quasi-separated'' by qcqs.

All occurring locally profinite sets will be second countable, so by ``locally profinite'' we will always mean ``locally profinite and second countable''. Such a set can be written as a countable disjoint union of profinite sets. Indeed, $T$ second-countable + locally compact + Hausdorff $\Rar$ $T$ paracompact, whereas $T$ paracompact + locally compact + totally disconnected $\Rar$ $T = \coprod_{i \in I} T_i$ with $T_i$ compact.

\subsection{$v$- and arc-topologies}\label{sec:vsheaves} We will make use of the $v$-topology on $\Perf$, see \cite[\S2]{BhattS_17} and \cite[\S2]{Rydh_10}. Recall \cite[Def.~2.1]{BhattS_17} that a morphism of qcqs schemes $f\colon X \rar Y$ is a \emph{$v$-cover}, or universally subtrusive, if for any map $\Spec V \rar Y$, with $V$ a valuation ring, there is an extension $V \har W$ of valuation rings and a commutative diagram
\[
\xymatrix{
 \Spec(W) \ar[r] \ar@{->>}[d] & X \ar[d]^f \\
\Spec(V) \ar[r] & Y
}
\]
The $v$-topology on $\Perf$ is the topology induced by $v$-covers on objects in $\Perf$ (regarded as affine schemes). We note that the $v$-topology on $\Perf$ is subcanonical \cite[Thm.~4.1]{BhattS_17}.

We will also use the even stronger arc-topology from \cite{BhattM_18}. Recall that a morphism in $\Perf$ is an \emph{arc-cover} if the above condition holds for all $V$ of rank $\leq 1$, and one can choose $W$ to be of rank $\leq 1$. The arc-topology on $\Perf$ is subcanonical and, moreover, a morphism in $\Perf$ is an arc-cover if and only if it is an universally effective epimorphism \cite[Thm.~5.16]{BhattM_18}. In particular, any arc-sheaf on $\Perf$ extends uniquely to an arc-sheaf on ${\rm Sch}_{\bF_p}$.

\begin{lm}\label{lm:check_surj_vsheaves}
Let $f \colon \cF \rar \cG$ be a morphism of $v$-sheaves on $\Perf_\kappa$, and assume that $\cF$ is qcqs and $\cG$ is quasi-separated. The following are equivalent:
\begin{itemize}
\item[(i)] $f$ is surjective (resp. an isomorphism).
\item[(ii)] For each valuation ring $V \in \Perf_\kappa$ with algebraically closed fraction field, $f(V) \colon \cF(V) \rar \cG(V)$ is surjective (resp. bijective).
\end{itemize}
\end{lm}
\begin{proof}
(i) clearly implies (ii). Now assume the surjectivity part of (ii). To check that $f$ is surjective, it suffices to do so after any base change $\Spec A \rightarrow \cG$ to a representable sheaf, i.e., we may assume that $\cG = Y$ for some $Y \in \Perf_\kappa$ and (as $\cG$ was assumed to be quasi-separated) that $\cF$ is still quasi-compact. 
As $\cF$ is quasi-compact, there is some affine $X \in \Perf_\kappa$ and a surjective map of $v$-sheaves $X \rar \cF$, which by composition with $f$ gives a map of $v$-sheaves $g \colon X \rar Y$ such that still, for any valuation ring $V$ with algebraically closed fraction field, $g(V)$ is surjective. This is a $v$-cover, so it is surjective map of $v$-sheaves. Hence also $f$ is surjective. 

Now assume bijectivity in (ii). We already know that $f$ is surjective, and it remains to prove injectivity. As above we can assume that $\cG = Y \in \Perf_\kappa$ and $\cF$ qcqs. The diagonal of $\cF$ factors through the injective map $g \colon \cF \rar \cF \times_Y \cF$. But by assumption, $g(V)$ is bijective for any valuation ring $V$. Also $\cF \times_Y \cF$ is qcqs, so by the above part of the proof, $g$ is an isomorphism, which implies that $f$ is injective. 
\end{proof}

Every scheme has a $v$-cover of the following very particular shape.

\begin{lm}[\cite{BhattS_17}, Lemma 6.2]\label{lm:standard_vcover} Let $X$ be a qcqs scheme. Then there is a $v$-cover $\Spec A \rar X$ such that 
\begin{itemize}
\item[(1)] Each connected component of $\Spec A$ is the spectrum of a valuation ring.
\item[(2)] The subset $(\Spec A)^c$ of closed points in $\Spec A$ is closed.
\end{itemize}
\end{lm}
Moreover, the composition $(\Spec A)^c \har \Spec A \rar \pi_0(\Spec A)$ is then a homeomorphism  of profinite sets (recall that for any qcqs scheme $X$, $\pi_0(X)$ is profinite \cite[Tag 0906]{StacksProject}).

\section{Loop functors}

We fix the setup of \S\ref{sec:setup_general}. The loop functor applied to a $k$-scheme $X$ produces a set-valued functor $LX$ on $\Perf_\kappa$. In this section we review and prove some facts about this construction.

\subsection{Definitions}\label{sec:loop_spaces}

Let $X$ be a scheme over $k$. As in \cite{PappasR_08, Zhu_17}, we have the loop space $LX$ of $X$, which is the functor on $\Perf_\kappa$, 
\[
R \mapsto LX(R) = X(\bW(R)[1/\varpi]).
\]

\begin{prop}[\S1.a of \cite{PappasR_08} and Proposition 1.1 of \cite{Zhu_17}] \label{prop:loop_of_affine_is_indproaffpfp} 
Let $X$ be an affine scheme of finite type over $\kappa$. Then $LX$ is representable by an ind-scheme.
\end{prop}

The association $X \mapsto LX$ is functorial. Also, $L(\cdot)$ sends closed immersions of affine schemes of finite type over $k$ to closed immersions of ind-schemes \cite[Lm.~1.2]{Zhu_17}.

\begin{rem}\label{rem:LPn}
For $n\geq 1$, the functor $L\bP^n$ does not seem to be a reasonable geometric object. Indeed, if $L^+$ denotes the functor of positive loops (cf. \S\ref{sec:positive_loops}), we have the perfect  \emph{scheme} $L^+\bP^n$ and a natural inclusion $L^+\bP^n \rar L\bP^n$, which is not an isomorphism. But the valuative criterion for properness implies that $L\bP^n(\ff) = L^+\bP^n(\ff)$ for any algebraically closed field $\ff/\bF_p$, i.e., $L^+\bP^n \rar L\bP^n$ induces a bijection on underlying topological spaces. 
\end{rem}

\begin{lm}\label{lm:L_commutes_limits}
The functor $X \mapsto LX$ commutes with arbitrary limits. 
\end{lm}
\begin{proof} 
This follows from the definitions.
\end{proof}

Now assume the setup of \S\ref{sec:setup_over_finite_field}. Let $X_0$ is an $k$-scheme and put $X = X_0 \times_k \breve k$. By Lemma \ref{lm:L_commutes_limits} we have $LX = LX_0 \times_{\Spec \bF_q} \Spec \obF$. In particular, the presheaf $LX$ carries the geometric Frobenius automorphism $\sigma = \sigma_{LX} \colon LX \rightarrow LX$.

\subsection{Graph morphism}

We again work in the setup of \S\ref{sec:setup_general}. Let $X$ be a separated $k$-scheme. Let $R \in \Perf_\kappa$ and $f_1,f_2 \in LX(R)$. Then $f_1,f_2$ correspond to morphisms 
\[
\tilde f_1, \tilde f_2 \colon \Spec \bW(R)[1/\varpi] \rar X.
\] 
As $X$ is separated, the equalizer 
\[ 
Z := {\rm Eq}\left(\Spec \bW(R)[1/\varpi] \doublerightarrow{\tilde f_1}{\tilde f_2} X \right)
\]
is a closed subscheme of $\Spec \bW(R)[1/\varpi]$. Regarding $\Spec R$ as a presheaf on $\Perf_\kappa$, we consider the subfunctor $F = F_{f_1,f_2}$ of $\Spec R$, such that $(\alpha \colon R \rightarrow R') \in (\Spec R)(R')$ lies in $F(R')$ if and only if the map $\tilde\alpha \colon \Spec \bW(R')[1/\varpi] \rar \Spec \bW(R)[1/\varpi]$ induced by $\alpha$ factors through $Z$.

\begin{lm}\label{lm:representability_coincidence_set}
In the above situation $F$ is representable by a closed subscheme of $\Spec R$.
\end{lm}

\begin{proof}
Let $\overline{Z}$ be the closure of $Z$ in $\Spec \bW(R)$. As $\alpha$ already induces a map $\tilde\alpha^+ \colon \Spec \bW(R') \rar \Spec \bW(R)$, we have for a given $R' \in \Perff$, 
\[
F(R') =  \{ \alpha \colon R \rightarrow R' \colon \text{$\tilde\alpha^+$ factors through $\overline{Z}$} \}
\]
For $n\geq 0$, let $\bW_n(R) = \bW(R)/\varpi^n\bW(R)$, and consider the closed subscheme $\overline Z_n = \overline{Z} \times_{\Spec \bW(R)} \Spec \bW_n(R)$ of $\overline{Z}$.
Let 
$F_n$ be the subfunctor of $\Spec R$, defined by
\[
F_n(R') = \{\alpha \colon R \rightarrow R' \colon \text{corresponding map $\tilde\alpha_n \colon \Spec\bW_n(R') \rightarrow \Spec \bW_n(R)$ factors through $\overline{Z}_n$} \}
\]
As $\prolim_n F_n = F$, we are reduced to show that $F_n$ is represented by a closed subscheme of $\Spec R$. 

Let $\fa \subseteq \bW_n(R)$ be the ideal of $\bW_n(R)$ defining $\overline Z_n$. Any element $a \in \bW_n(R)$ can be written in a unique way as a sum $a = \sum_{i=0}^{n-1}[a_i]\varpi^i$ with $a_i \in R$. Let $\fb \subseteq R$ be the ideal generated by all coefficients $a_i$ when $a = \sum_{i=0}^{n-1}[a_i]\varpi^i$ varies through $\fa$ and $i$ varies through $\{0,1, \dots,n-1\}$. By functoriality of the Teichm\"uller lift, it is clear that the map $\tilde\alpha_n \colon \bW_n(R) \rar \bW_n(R')$ induced by $\alpha$ is given by $\sum_{j=0}^{n-1} [x_i]\varpi^i \mapsto \sum_{j=0}^{n-1} [\alpha(x_i)]\varpi^i$. From this it follows that $\alpha(\fb) = 0 \LRar \tilde\alpha_n(\fa) = 0$. Thus $F_n$ is represented by $\Spec R/\fb$.
\end{proof}

\begin{lm}\label{lm:representability_graph}
Let $X$ be a separated $k$-scheme and let $\beta$ be an endomorphism of $LX$. Then the graph morphism $(id,\beta) \colon LX \rar LX \times LX$ of $\beta$ is representable by closed immersions. In particular, $LX$ is separated.
\end{lm}
\begin{proof}
Let $R \in \Perf_\kappa$ and let $f_1,f_2 \colon \Spec R \rar LX \times LX$ be an $R$-valued point. We have to show that $G := \Spec R \times_{LX \times LX} LX$ is representable by a closed subscheme of $\Spec R$. In fact, $G$ is a subfunctor of $\Spec R$ and $(\alpha \colon R \rightarrow R') \in (\Spec R)(R')$ lies in $G(R')$ if and only if there exists a (necessarily unique) $\gamma \colon \Spec R' \rightarrow LX$, such that $(\id,\beta) \circ \gamma = (f_1,f_2) \circ \alpha \colon \Spec R' \rar LX \times LX$. Thus (as $\tilde f \circ \tilde\alpha = \widetilde{f\circ\alpha}$ with $\tilde\alpha$ as in the text before Lemma \ref{lm:representability_coincidence_set}), 
\[
G(R') = \{\alpha \colon \Spec R' \rar \Spec R \colon \beta f_1 \alpha = f_2 \alpha \} = F_{\beta f_1,f_2}(R'),
\]
which is representable by Lemma \ref{lm:representability_coincidence_set}.
\end{proof}

\subsection{Positive loops}\label{sec:positive_loops}
Let $\caX$ be an $\caO_k$-scheme. Then the space of positive loops $L^+\caX$ is the functor on $\Perf_\kappa$,
\[
R \mapsto L^+\caX(R) = \caX(\bW(R)).
\]
We also have truncated versions of this. For $r \geq 1$, let $L^+_r\caX$ be the functor on $\Perf_\kappa$, sending $R \mapsto \caX(\bW(R)/\varpi^r\bW(R))$. Suppose that $\caX$ is affine and of finite type over $\caO_k$. Then $L^+\caX$ and $L^+_r\caX$ are representable by schemes, and the latter is perfectly finite presented over $\kappa$. Moreover, if $\caX_\eta$ denotes the generic fiber of $\caX$, then the natural map $L^+\caX \rar L\caX_\eta$ is a closed immersion.

\section{Schemes attached to (locally) profinite sets}\label{sec:schemes_for_profinite_sets}

Let $\kappa$ be a field. For any topological space $T$ we may consider the functor on qcqs $\kappa$-schemes,
\begin{equation}\label{eq:functor_cont_functions_top_space}
\underline T = \underline T_\kappa \colon S \mapsto {\rm Cont}(|S|, T)
\end{equation} 
(we omit $\kappa$ from notation, whenever it is clear from the context). If $T$ is compact Hausdorff, $\underline{T}$ is represented by the affine scheme $\Spec {\rm Cont}(T, \kappa)$, where we write ${\rm Cont}(T,\kappa)$ for the ring of continuous functions $T \rar \kappa$, where $\kappa$ is equipped with the discrete topology. We only will need this for $T$ profinite, so let's recall the proof in that case. We can write $T = \prolim_n T_n$ as an inverse limit of discrete finite sets. Then each $\underline{T}_n$ is represented by the affine scheme $\Spec {\rm Cont}(T_n, \kappa)$, and $\underline T = \prolim_n \underline T_n$ is an inverse limit of affine schemes, hence \cite[Tag 01YW]{StacksProject} itself an affine scheme, the spectrum of $\dirlim_n{\rm Cont}(T_n, \kappa) =  {\rm Cont}(T, \kappa)$. 

We will need a topological version of the above construction. Let $\caO$ be any ring and $0 \neq \varpi \in \caO$ a non-zero divisor contained in the Jacobson radical of $\caO$. Equip the ring $k := \caO[\varpi^{-1}]$ with the $\varpi$-adic topology. Recall from \cite[5.4.15-16]{GabberR_03} (applied to $R = \caO$, $t = \varpi$, $I = \caO$), that there is a natural way to topologize the sets $X(k)$ of $k$-points of all affine $k$-schemes $X$ of finite type, compatible with immersions, and such that for $X = \bA_k^n$ we get $k^n$ with its $\varpi$-adic topology. By \cite[5.4.19]{GabberR_03}, this construction naturally globalizes to all $k$-schemes $X$ locally of finite type, which satisfy the condition
\begin{equation}\label{eq:cond_on_X_topologize}
X(k) = \bigcup_{U \subseteq X} U(k)
\end{equation}
where $U$ ranges over all open affine $k$-subschemes.\footnote{Condition \eqref{eq:cond_on_X_topologize} is satisfied when $X$ is either an open subscheme of a projective $k$-scheme \cite[Lm.~5.4.17]{GabberR_03} or ind-quasi-affine and locally of finite type \cite[Lm.~2.2.6]{BouthierC_19}. We will not make use of these results.}

For a topological space $T$ we may consider the $k$-scheme 
\[
\underline{T}_{k,\varpi} := \Spec {\rm Cont}_{\varpi}(T, k),
\]
the spectrum of the ring of continuous functions $T \rar k$. 
We do not claim that $\underline T_{k,\varpi}$ represents some functor similar as in \eqref{eq:functor_cont_functions_top_space}. Instead it has the following useful property.

\begin{lm}\label{lm:maps_from_underlineTtau}
Let $T$ be a profinite set and $X$ a $k$-scheme satisfying condition \eqref{eq:cond_on_X_topologize}. There is a bijection, functorial in $T$ and $X$, ${\rm Hom}_k(\underline T_{k,\varpi}, X) = {\rm Cont}_{\varpi}(T, X(k))$, where on the right side $\varpi$ indicates that $X(k)$ is endowed with the $\varpi$-adic topology.
\end{lm}
\begin{proof}
For any disjoint covering by finitely many clopen subsets $T = \bigcup_i T_i$, we have $\underline T_{k,\varpi} = \coprod_i \underline T_{i,k,\varpi}$.
Thus, as $T$ is profinite, the problem is local on $T$ and on $X$ and we may assume that $X$ is affine. We then may assume that $X = \bA_k^n$. Assume we are given a $k$-morphism $\underline T_{k,\varpi} \rar X$. For any $t \in T$, there is a corresponding maximal ideal of ${\rm Cont}_{\varpi}(T,k)$, the kernel of evaluation map at $t$, and the quotient of ${\rm Cont}_{\varpi}(T,k)$ modulo this ideal is isomorphic to $k$. Thus $T$ can be identified with a subset of $|\underline T_{k,\varpi}|$. Let $\alpha \colon \underline T_{k,\varpi} \rar \bA_k^n = \Spec k[x_1,\dots,x_n]$ be a $k$-morphism. Then on underlying topological spaces $\alpha$ maps the subset $T$ of $|\underline T_{k,\varpi}|$ into the subset $\mathbb{A}_k^n(k) \subseteq |\mathbb{A}_k^n|$. Denote the resulting map by $\beta \colon T \rar \bA^n(k) = k^n$. Let $\alpha^\# \colon k[x_1,\dots,x_n] \rar {\rm Cont}_\varpi(T,k)$ be the homomorphism corresponding to $\alpha$. For $1\leq i \leq n$ let $\lambda_i = \alpha^\#(x_i) \in {\rm Cont}_\varpi(T,k)$. The point $t \in T$ is mapped by $\beta$ to the point $(\lambda_1(t), \dots, \lambda_n(t)) \in k^n$. Thus, as $\lambda_i$ is continuous, also $\beta$ is. This defines a map in one direction in the proposition.

Conversely, start with a $\varpi$-adically continuous map $\beta \colon T \rar k^n$. For $f \in k[x_1,\dots,x_n]$ define $\alpha^\#(f) := f \circ \beta \colon T \rar k^n \rar k$. As $\beta$ and $f$ are $\varpi$-adically continuous, also $\alpha^\#(f)$ is. We obtain the $k$-morphism $\alpha \colon \underline{T}_{k,\varpi} \rar \bA_k^n$ of schemes attached to $\alpha^\#$. These two constructions are mutually inverse. Functoriality is clear.
\end{proof}

Assume now the setup of \S\ref{sec:setup_general}. Taking $\caO = \caO_k$, the above considerations apply to the field $k$ equipped with $\varpi$-adic topology. Note that condition \eqref{eq:cond_on_X_topologize} holds for any $k$-scheme $X$ locally of finite type. One verifies directly that for $\kappa \in \Perf$ the ring ${\rm Cont}(T, \kappa)$ is a perfect $\kappa$-algebra.

\begin{lm}\label{lm:Witt_of_continuous_maps} Let $T$ be a profinite set. There are natural isomorphisms $\bW({\rm Cont}(T, \kappa)) \cong {\rm Cont}_{\varpi}(T, \caO_k)$ and $\bW({\rm Cont}(T, \kappa))[1/\varpi] \cong {\rm Cont}_{\varpi}(T, k)$.
\end{lm}

\begin{proof}
For $R \in \Perf_\kappa$, $\bW(R)$ is characterised as the unique $\bW(\kappa)$-algebra $\tilde R$, which is complete and separated with respect to the $\varpi$-adic topology, in which $\varpi$ is not a zero-divisor, and which satisfies $\tilde R / \varpi\tilde R = R$. When $R = {\rm Cont}(T,\kappa)$, $\tilde R = {\rm Cont}_{\varpi}(T, \caO_k)$ satisfies these properties. This proves the first claim. The second claim follows from the first. 
\end{proof}

\begin{cor}\label{cor:morphisms_from_const_sheaf_into_loop_group}
Let $T$ be a profinite set and let $X$ be a $k$-scheme, which is locally of finite type. Then ${\rm Hom}(\underline{T}_{\kappa}, LX) = {\rm Cont}_{\varpi}(T, X(k))$. 
\end{cor}
\begin{proof}
This formally follows from the definitions and Lemmas \ref{lm:maps_from_underlineTtau} and \ref{lm:Witt_of_continuous_maps}.
\end{proof}

Here we consider homomorphisms in the category of presheaves on $\Perf_{\kappa}$ (or any full subcategory of sheaves, where these presheaves belong to). Let now $T$ be a locally profinite set (recall the convention from \S\ref{sec:further_not_conv}).

\begin{cor}\label{cor:morphisms_from_cont_locprofin_into_loop_group}
Let $T$ be a locally profinite set, and let $X$ be a $k$-scheme, which is locally of finite type. Then ${\rm Hom}(\underline{T}_{\kappa}, LX) = {\rm Cont}_\varpi(T, X(k))$
\end{cor}
\begin{proof}
Write $T = \coprod_i T_i$ with $T_i$ profinite. Then $\underline{T}_{\kappa} = \coprod_i \underline{T}_{i,\kappa}$. We are done by Corollary \ref{cor:morphisms_from_const_sheaf_into_loop_group}.
\end{proof}

\subsection{A quasi-compactness lemma}\label{sec:qc_lemma}
We work in the setup of \S\ref{sec:setup_over_finite_field}. We prove a lemma needed in \S\ref{sec:bsigma_fixed_points}.

\begin{lm}\label{lm:qc_loc_profin}
Let $X$ be an affine $k$-scheme of finite type. Then $X(k)$, equipped with its $\varpi$-adic topology, is locally profinite, and the map $\underline{X(k)} \rar LX$ induced by Lemma \ref{cor:morphisms_from_cont_locprofin_into_loop_group}  is quasi-compact. 
\end{lm}
\begin{proof}
Fix a closed immersion $X \har \bA^n$ over $k$. It identifies $X(k)$ with a closed subset of the locally profinite set $k^n$, so the first claim follows. As $LX \rar L\bA^n$ is an (even closed) immersion, it is enough to show that the composition $\underline{X(k)} \rar LX \rar L\bA^n$ is quasi-compact. Therefore, it is enough to show that $\underline{X(k)} \rar \underline{k^n}$ and $\underline{k^n} \rar L\bA^n$ are quasi-compact. The first is a closed immersion, and so we are reduced to the case $X = \bA^n$. Exhaust $X$ by $X = \dirlim_i X_i$, with $X_i = \varpi^{-i} L^+\bA_{\caO_k}^n$ for $i \geq 0$. We have $\underline{X(k)} \times_{LX} X_i = \underline{T_i}$, where $T_i = \varpi^{-i}\caO_k^n \subseteq k^n = X(k)$ is profinite. If now $Y \rar LX$ is a map from an affine scheme in $\Perf_{\obF}$, then it factors through $Y \rar X_i$ for $i \gg 0$ and $\underline T \times_{LX} Y = (\underline T \times_{LX} X_i) \times_{X_i} Y = \underline{T_i} \times_{X_i} Y$, which is quasi-compact. 
\end{proof}

\section{Arc-descent for the loop functor}\label{sec:arc_descent}

Here we work in the setup of \S\ref{sec:setup_general}. We will prove the following result.

\begin{thm}\label{thm:LX_is_vsheaf}
Let $X$ be a quasi-projective scheme over $k$. Then $LX$ is an arc-sheaf on $\Perf_\kappa$. 
\end{thm}

The proof strategy is: 1) show that if $\Spec R' \rar \Spec R$ is an arc-cover, then $\Spa \bW(R') \rar \Spa \bW(R)$ is an arc-cover, that is, surjective on rank one valuations (Proposition \ref{prop:surjectivity_on_valuations}); 2) Using part 1) and perfectoid techniques from \cite{ScholzeW_20}, show arc-descent for vector bundles on $\bW(R)[1/\varpi]$ (Proposition \ref{prop:descent_vector_bundles_arc}). This gives Theorem \ref{thm:LX_is_vsheaf} for $X = \bP^n$; 3) again exploiting part 1), show that if Theorem \ref{thm:LX_is_vsheaf} holds for $X$, then it holds for any locally closed subscheme of $X$.

\subsection{Witt-vectors and arc-covers}\label{sec:witt_vectors_and_arc_covers}

First we study the effect of the functors $R \mapsto \bW(R)$ resp. $R \mapsto \bW(R)[1/\varpi]$ on arc-covers. 

\subsubsection{Two lemmas}

\begin{lm}\label{lm:Witt_reduced}
Let $R$ in $\Perf_\kappa$. The rings $\bW(R)$, $\bW(R)[1/\varpi]$ are reduced. Moreover, $\bW(\cdot)$ and $\bW(\cdot)[1/\varpi]$ preserve injections of rings.
\end{lm}
\begin{proof}
Both claims are immediate in the equal characteristic case. In the other case, the first claim is easy (see e.g. \cite[Lm.~3.7]{Shimomoto_14}), and the second claim follows from the existence and uniqueness of the Witt presentation and the functoriality of Teichm\"uller lifts.
\end{proof}

\begin{lm}\label{lm:dominance_arc_Witt}
Let $R \rar R'$ be an arc-cover in $\Perf_\kappa$. Then $\Spec \bW(R') \rar \Spec \bW(R)$ and $\Spec \bW(R')[1/\varpi] \rar \Spec \bW(R)[1/\varpi]$ are dominant. 
\end{lm}
\begin{proof}
Arc-covers are surjective on spectra, thus dominant, and hence $\ker(R \rar R') \subseteq \nil(R)$. As $R$ is perfect, it is reduced, and thus $R \rar R'$ is injective. By Lemma \ref{lm:Witt_reduced} the same holds after applying $\bW$ (resp. applying $\bW$ and inverting $\varpi$).
\end{proof}

\subsubsection{Continuous valuations} 

Our next goal will be to prove that if $R \rar R'$ is an arc-cover in $\Perf_\kappa$, then the image of $\Spec \bW(R')[1/\varpi] \rar \Spec \bW(R)[1/\varpi]$ contains all closed points of the target. Therefore we use the adic spectrum, which we first recall. 

Let $A$ be a ring. Recall (for example from \cite[2.3]{ScholzeW_20}) that a \emph{valuation}
on $A$ is a map $|\cdot| \colon A\rar \Gamma\cup\{0\}$ into a totally ordered abelian group $\Gamma$, such that $|0| = 0$, $|1| = 1$, $|xy| = |x|\cdot|y|$, $|x+y| \leq \max(|x|,|y|)$ (where by convention $0<\gamma$ and $\gamma 0 = 0$ for all $\gamma \in \Gamma$). Two valuations $|\cdot|$, $|\cdot|'$ on $A$ are \emph{equivalent} if $|a|\leq |b| \LRar |a| \leq |b|$ for all $a,b \in A$. A valuation is \emph{of rank $\leq 1$}, if it is equivalent to a valuation with value group $\Gamma = \bR_{> 0}^\times$. The \emph{support} of a valuation $|\cdot|$ is the prime ideal $\supp |\cdot| = \{x \in R \colon |x| = 0\}$. If $A$ is a topological ring, then a valuation $|\cdot|$ is said to be \emph{continuous}, if $\{x \in A \colon |x| < \gamma \} \subseteq A$ is open for each $\gamma \in \bR_{> 0}$. If $A^+$ is a subring of a topological ring $A$, then the \emph{adic spectrum} $\Spa(A,A^+)$ of $(A,A^+)$ is the set of equivalence classes of continuous valuations on $A$, such that $|a| \leq 1$ for all $a \in A^+$. We consider the subset $\Spa_{\leq 1}(A,A^+)$ of $\Spa(A,A^+)$ of equivalence classes of continuous valuations of rank $\leq 1$.

If $R \in \Perf_\kappa$, we always equip $R$ with the discrete topology, take $R^+ = R$, and write $\Spa_{\leq 1}(R)$ for $\Spa_{\leq 1}(R,R)$. For $R \in \Perf_\kappa$ we always equip $\bW(R)[1/\varpi]$ with the $\varpi$-adic topology, with respect to which it is separated and complete, and we write $\Spa_{\leq 1} \bW(R)$ for $\Spa_{\leq 1}(\bW(R), \bW(R))$. Note that $R$ and $\bW(R)[1/\varpi]$ are uniform Huber rings, and the latter is also Tate. Moreover, $(R,R)$ and $(\bW(R)[1/\varpi], \bW(R))$ are Huber pairs.

\subsubsection{Witt vectors and the adic spectrum} 

\begin{prop}\label{prop:surjectivity_on_valuations}
Let $R \rar R'$ be an arc-cover in $\Perf$. Then $\Spa_{\leq 1} \bW(R') \rar \Spa_{\leq 1} \bW(R)$ is surjective.
\end{prop}
\begin{proof}
In the same way as in \cite[Lm.~4.4]{Kedlaya_13} there is a map $\mu \colon \Spa_{\leq 1}\bW(R) \rar \Spa_{\leq 1}(R)$, which is defined by precomposition with the Teichm\"uller lift, i.e., it sends a valuation $|\cdot|$ of $\bW(R)$ to the valuation $\widetilde{|\cdot|} := \mu(|\cdot|) \colon R \stackrel{[\cdot]}{\rar} \bW(R) \rar \bR_{\geq 0}$. 

\begin{lm}\label{lm:valuation_realized_via_Kedlayas_map}
Let $|\cdot| \in \Spa_{\leq 1}(\bW(R))$ and let $\widetilde{|\cdot|} = \mu(|\cdot|) \in \Spa_{\leq 1}(R)$ have support $\fp$. Then $\widetilde{|\cdot|}$ corresponds to a homomorphism $R \tar R/\fp \har \caO_K$ where $K = {\rm Frac}(R/\fp)$ is a non-archimedean field in $\Perf$. By functoriality of $\bW$ we obtain a $\varpi$-adically continuous homomorphism $\bW(R) \rar \bW(\caO_K)$. Then $|\cdot|$ lies in the image of $\Spa_{\leq 1} \bW(\caO_K) \rar \Spa_{\leq 1}\bW(R)$.
\end{lm}
\begin{proof}
We have $\ker(\bW(R) \rar \bW(\caO_K)) = \{\sum_{n\geq 0} [x_n]\varpi^n \colon x_n \in \fp \}$. As $|\cdot|$ is $\varpi$-adically continuous, for each $x = \sum_{n\geq 0} [x_n]\varpi^n \in \ker(\bW(R) \rar \bW(\caO_K))$ we have
\[
| x | = \left| \sum_{n\geq 0} [x_n]\varpi^n \right| = \lim_{N \rar +\infty} \left|\sum_{n\geq 0}^N [x_n]\varpi^n\right| \leq \lim_{n \rar \infty}\max\nolimits_{n=1}^N |[x_n]\varpi^n| = \lim_{n \rar \infty}\max\nolimits_{n=1}^N \widetilde{|x_n|}|\varpi|^n  = 0,
\]
as all $x_n \in \fp$. Thus the support of $|\cdot|$ contains $\ker(\bW(R) \rar \bW(\caO_K))$, i.e., $|\cdot|$ is induced from a valuation (again denoted $|\cdot|$) of $\bW(R/\fp) \cong \bW(R)/ \ker(\bW(R) \rar \bW(\caO_K))$. It remains to show that $|\cdot|$ on $\bW(R/\fp)$ lifts to a valuation in $\Spa_{\leq 1}\bW(\caO_K)$. Inside $\bW(\caO_K)$ we have the subring of elements with bounded denominator:
\[
\bW(\caO_K)' = \{[r]^{-1} \sum_{n=0}^{\infty} \left[x_i\right] \varpi^n \colon r \in R/\fp \text{ and for all $n \geq 0$: } x_n \in R/\fp \text{ and } x_n/r \in \caO_K \}.
\]
(that this is indeed a subring follows from the multiplicativity of the Teichm\"uller lift). Then $(\bW(\caO_K)', \bW(\caO_K)')$ is an affinoid ring with $\varpi$-adic completion equal to $(\bW(\caO_K), \bW(\caO_K))$. Therefore, $\Spa_{\leq 1}\bW(\caO_K)' = \Spa_{\leq 1}\bW(\caO_K)$ and it is enough to lift $|\cdot|$ to $\bW(\caO_K)'$. But here we can (and must) define the valuation $|\cdot|'$ by $\left|[r]^{-1}\sum_{n=0}^{\infty} \left[x_i\right] \varpi^n \right|' := \widetilde{|r|}^{-1}\left|\sum_{n=0}^{\infty} \left[x_i\right] \varpi^n \right|$. Clearly, this is independent of the choice of the presentation as a fraction. Moreover, it is a valuation of rank $1$ extending $|\cdot|$ on $\bW(R/\fp)$, and it remains to check that $|\cdot|'$ is $\varpi$-adically continuous and bounded by $1$. Let $x = [r]^{-1}\sum_{n=0}^{\infty} \left[x_n\right] \varpi^n \in \bW(\caO_K)'$. A computation (similar to the above) using the $\varpi$-adic continuity of $|\cdot|$, along with the fact that $r^{-1}x_n \in \caO_K$ for each $n$, so that $\widetilde{|x_n|} \leq \widetilde{|r|}$, gives $|\sum_{n\geq 0}[x_n]\varpi^n| \leq \widetilde{|r|}$. This in turn gives $|x|' \leq 1$. Finally, $\varpi$-adic continuity of $|\cdot|'$ follows from this and $|\varpi|' = |\varpi| < 1$.
\end{proof}

We continue with the proof of Proposition \ref{prop:surjectivity_on_valuations}. Fix a valuation  $|\cdot| \in \Spa_{\leq 1}\bW(R)$. The attached valuation $\mu(|\cdot|)$ of $R$ corresponds to a homomorphism $R \rar \caO_K$ into the integers of a non-archimedean field in $\Perf$. This gives the frontal commutative square in the diagram \eqref{eq:cube_diagram_valuations}. As $R\rar R'$ is an arc-cover, $\Spa_{\leq 1} R' \rar \Spa_{\leq 1}R$ is surjective (in fact, these statements are equivalent). This means that we can find a non-archimedean field extension $L$ of $K$ with integers $\caO_L$, and a valuation of $R'$ corresponding to a homomorphism $R' \rar \caO_L$, such that the right side of the cube in diagram \eqref{eq:cube_diagram_valuations} is commutative. Then using the map $\mu$ and the functoriality of the involved constructions we can extend these two commutative squares to the full commutative diagram,
\begin{equation}\label{eq:cube_diagram_valuations}
    \begin{tikzcd}[row sep=1.5em, column sep = 1.5em]
    \Spa_{\leq 1} \bW(\caO_L) \arrow[rr] \arrow[dr] \arrow[dd] &&
    \Spa_{\leq 1} \caO_L \arrow[dd] \arrow[dr] \\
    & \Spa_{\leq 1} \bW(\caO_K) \arrow[rr] \arrow[dd] &&
    \Spa_{\leq 1} \caO_K \arrow[dd] \\
    \Spa_{\leq 1} \bW(R') \arrow[rr,] \arrow[dr] && \Spa_{\leq 1} R' \arrow[dr] \\
    & \Spa_{\leq 1} \bW(R) \arrow[rr] && \Spa_{\leq 1} R
    \end{tikzcd}
\end{equation}
where each horizontal arrow is the map $\mu$ for the corresponding ring. Applying Lemma \ref{lm:valuation_realized_via_Kedlayas_map} to the frontal square, we are reduced to the case that $R,R'$ are valuation rings and $R\rar R'$ is a injective local homomorphism. Then $R \rar R'$ is faithfully flat. Hence $\bW(R) \rar \bW(R')$ is $p$-completely faithfully flat. A point in $\Spa_{\leq 1}(\bW(R))$ corresponds to a homomorphism $\bW(R) \rar \caO_K$ to the integers of some non-archimedean field $K$. From the $p$-complete flatness of $\bW(R) \rar \bW(R')$ it follows that $\bW(R') \otimes_{\bW(R)} \caO_K$ has a non-trivial generic fiber. Hence it admits a morphism into the integers $\caO_M$ of some non-archimedean field $M$, and the corresponding composition $\bW(R') \rar \bW(R') \otimes_{\bW(R)} \caO_K \rar \caO_M$ gives a rank one continuous valuation on $\bW(R')$ bounded by $1$, which lifts the original valuation of $\bW(R)$.\end{proof}

\begin{cor}\label{cor:arc_cover_surj_on_closed_primes}
Let $\alpha \colon R \rar R'$ be an arc-cover in $\Perf_\kappa$. The maximal ideals of $\bW(R)[1/\varpi]$ lie in the image of $\Spec \bW(R')[1/\varpi] \rar \Spec \bW(R)[1/\varpi]$.
\end{cor}
\begin{proof}
Let $\fP$ be a maximal ideal of $\bW(R)[1/\varpi]$ and let $\fp = \fP \cap \bW(R)$ be the corresponding prime ideal of $\bW(R)$. The ring $\bW(R)[1/\varpi]$ is equipped with the $\varpi$-adic topology and $\bW(R)$ is an open subring. As $\fP$ is a maximal ideal, it is closed, hence $\fp$ is closed in $\bW(R)$. It follows that $\varpi$ is neither zero nor a unit in the domain $A := \bW(R)/\fp$. Moreover, $A$ has the field of fractions $K := \bW(R)[1/\varpi]/\fP = A[1/\varpi]$. Let $\fq \subseteq A$ be a prime ideal containing $\varpi$. Then there exists a valuation subring $V$ of $K$ with maximal ideal $\fm_V$ dominating the localization $A_\fq$ of $A$, i.e., $A \subseteq A_\fq \subseteq V \subseteq K$ such that $\fm_V \cap A = \fq$. Denote the corresponding valuation on $\bW(R)$ by $|\cdot|_V$.We have $|x|_V \leq 1$ for all $x \in \bW(R)$, $0 < |\varpi|_V < 1$, and the support of $|\cdot|_V$ is $\fp$. Let $\overline\varpi$ be the image of $\varpi$ in $A \subseteq V$. The ideal $\fq_{0,V} = \sqrt{\overline \varpi V}$ of $V$ is the minimal prime ideal containing $\overline\varpi V$, and $\fp_{0,V} = \bigcap_{n\geq 0} \overline\varpi^n V$ is the maximal prime ideal contained in $\overline\varpi V$. The resulting specialization $\fp_{0,V} \rightsquigarrow \fq_{0,V}$ is an immediate one, hence the corresponding valuation ring $(V/\fp_{0,V})_{\fq_{0,V}}$ is of rank $1$ and the image of $\varpi$ is a pseudo-uniformizer (cf. \cite[Rem.~2.2]{BhattM_18}). Let $|\cdot|$ denote the corresponding valuation of $\bW(R)$. It is continuous, hence in $\Spa_{\leq 1}(\bW(R))$, hence by Proposition \ref{prop:surjectivity_on_valuations} can be lifted to a valuation $|\cdot|'$ of $\bW(R')$, whose support, a prime ideal of $\bW(R')$, maps to $\fp_0 := \supp_{|\cdot|}$ under $\Spec \bW(R') \rar \Spec \bW(R)$. It thus remains to show that $\fp_0 = \fp$. But $\fp_0$ is the preimage of $\fp_{0,V}$ in $\bW(R)$, i.e.,
\[
\fp_0 = \{x \in \bW(R) \colon |x|_V \leq |\varpi|_V^n \text{ for all $n > 0$} \}.
\]
As $\varpi \not\in \fp_0$ and $\fp_0 \supseteq \fp$, we have $\fp_0 = \fp$ by maximality of $\fP$. 
\end{proof}

\subsection{Arc-descent for vector bundles over $\bW(R)[1/\varpi]$}

Let ${\rm Perfd}$ denote the category of all perfectoid spaces. Generalizing the $v$-topology \cite[Def.~17.1.1]{ScholzeW_20}, we may define the arc-topology on ${\rm Perfd}$.

\begin{Def}\label{def:arc_top_on_perfd}
We say that a family of morphisms $\{f_i \colon X_i \rar Y\}_{i \in I}$ in ${\rm Perfd}$ is an \emph{arc-cover} if for all quasi-compact open subsets $V \subseteq Y$, there exists a finite subset $I_V \subseteq I$ and a quasi-compact open $U_i \subseteq X_i$ for all $i \in I_V$ such that any rank-$1$-point of $V$ comes from a rank-$1$-point of some of the $U_i$'s. We call the topology on ${\rm Perfd}$ generated by arc-covers the \emph{arc-topology}.
\end{Def}

This topology is stronger than the $v$-topology. Nevertheless, several results from \cite{Scholze_ECD,ScholzeW_20} formulated for the $v$-topology continue to hold for the arc-topology with essentially the same proofs. 
For example we have the following arc-version of \cite[Thm.~8.7, Prop.~8.8]{Scholze_ECD}.

\begin{lm}
\label{lm:structure_presheaf_sheaf}
The pre-sheaf $X \mapsto \caO_X(X)$ is a sheaf for the arc-topology on ${\rm Perfd}$. Moreover, for an affinoid perfectoid $X$, $H_{\rm arc}^i(X,\caO_X) = 0$ for $i > 0$ and $H_{\rm arc}^i(X,\caO_X^+)$ is almost zero for all $i > 0$.   
\end{lm}
\begin{proof} The proof goes along the lines of \cite[Thm.~8.7, Prop.~8.8]{Scholze_ECD}. To show the first statement, we first note that $\caO_X(X)$ injects into $\prod_{x \in |X|} K(x)$. Moreover, it is enough to only consider the rank-1 points of $X$, as any point has a unique rank-1 generalization. This implies that $\caO_X$ is separated. By the same arguments as in \cite[Thm.~8.7]{Scholze_ECD} we can reduce to the situation that $X$ is totally disconnected affinoid perfectoid, $Y = \Spa(S,S^+)\rar \Spa(R,R^+) = X$ is a map of affinoid perfectoid spaces, in which it suffices to show that if $\varpi \in R$ is a pseudo-uniformizer, then 
\begin{equation}\label{eq:long_almost_exact}
0 \rar R^+/\varpi \rar S^+/\varpi \rar S^+/\varpi \otimes_{R^+/\varpi} S^+/\varpi \rar \dots
\end{equation}
is almost exact (in fact, we need exactness at $S^+/\varpi$ only).
This can be done locally on $X$, so we can replace $X$ by any of its connected components, i.e., we may assume that $X = \Spa(K,K^+)$ for some perfectoid field $K$. But $K^\circ/K^+$ is almost zero, so that we may replace $K^+$ by $K^\circ$ (and $Y$ by $Y \times_{\Spa(K,K^+)} \Spa(K,K^\circ)$), i.e., we may assume $X = \Spa(K,K^\circ)$. In that situation $X$ consists of a unique rank-$1$ point, so that $|Y| \rar |X|$ is surjective by assumption. By \cite[Prop.~7.23]{Scholze_ECD} $K^\circ/\varpi \rar S^+/\varpi$ is then faithfully flat and we are done with the first claim.

The second claim follows from the almost exactness of \eqref{eq:long_almost_exact} by exactly the same argument as in the proof of \cite[Prop.~8.8]{Scholze_ECD}.
\end{proof}

We have the following version of \cite[Lm.~17.1.8]{ScholzeW_20}.

\begin{lm}\label{lm:perfd_descent}
The fibered category sending any $X \in {\rm Perfd}$ to the category of locally finite free $\caO_X$-modules is a stack for the arc-topology on ${\rm Perfd}$. 
\end{lm}
\begin{proof} The proof goes along the lines of \cite[Lm.~17.1.8]{ScholzeW_20}. Let $\widetilde X = \Spa(\widetilde{R},\widetilde{R}^+) \rar X = \Spa(R,R^+)$ be a morphism of perfectoid affinoids, which is an arc-cover. By \cite[Thm.~2.7.7]{KedlayaL_15} it is sufficient to show that the base change functor from the finite projective $R$-modules to finite projective $\widetilde{R}$-modules equipped with a descent datum is an equivalence of categories. Full faithfullness follows from Lemma \ref{lm:structure_presheaf_sheaf}. As by \cite[Thm.~2.7.7]{KedlayaL_15} vector bundles can be glued over open covers, essential surjectivity can be checked locally.

Now, literally the same argument as in \cite[Lm.~17.1.8]{ScholzeW_20} works and shows the claim in the case that $R$ is a perfectoid field. The argument of \cite[Lm.~17.1.8]{ScholzeW_20} to deduce the general case from the above goes through also here, as $\check{H}^1_{\rm arc}(\widetilde{X}/X, M_r(\caO_X^+/\varpi))$ is almost zero by Lemma \ref{lm:structure_presheaf_sheaf}.
\end{proof}

As a consequence we deduce the following version of \cite[Prop.~19.5.3]{ScholzeW_20}.

\begin{prop}\label{prop:descent_vector_bundles_arc}
The fibered category sending any $R \in \Perf$ to the category of locally finite free $\bW(R)[1/\varpi]$-modules is a stack for the arc-topology.
\end{prop}
\begin{proof}
This follows from Lemma \ref{lm:perfd_descent} in the same way as \cite[Prop.~19.5.3]{ScholzeW_20} follows from \cite[Lm.~17.1.8]{ScholzeW_20}\footnote{We emphasize that the role of $p$ from \cite{ScholzeW_20} is played here by $\varpi$, whereas the $\varpi$ from \cite{ScholzeW_20} has no analogue here: in fact, in contrast to \cite{ScholzeW_20}, where $R$ is a perfectoid ring in characteristic $p$ with pseudo-uniformizer $\varpi$, we simply work with a perfect ring $R$ in characteristic $p$.}. We explain the argument in the mixed characteristic case; the other case is similar. Let $R \rar \widetilde{R}$ be an arc-cover in $\Perf$. Let $A^+ = \bW(R)$, $A = A^+[1/\varpi]$ and let $\widetilde{A}^+ = \bW(\widetilde{R})$, $\widetilde{A} = \widetilde{A}^+[1/\varpi]$. Let $U = \Spa(A,A^+)$ and $\widetilde{U} = \Spa(\widetilde{A},\widetilde{A}^+)$. We have to show descent for vector bundles along $\widetilde{U} \rar U$. 

Note that $U$ is sousperfectoid. Indeed, let $\bZ_p[p^{1/p^\infty}]_p^{\wedge}$ denote the $p$-adic completion of $\bZ_p[p^{1/p^{\infty}}]$. Consider $A^{\prime +} = W(R) \widehat{\otimes}_{\bZ_p} \bZ_p[p^{1/p^\infty}]_p^{\wedge}$ and let $A' = A^{\prime +}[1/\varpi]$. Then 
\[ 
U' = U \times_{\Spa \bZ_p} \Spa \bZ_p[p^{1/p^\infty}]_p^{\wedge} = \Spa(A',A^{\prime+})
\] 
is an affinoid perfectoid space. 
Moreover, $\widetilde{U}' = \widetilde{U} \times_{\Spa \bZ_p} \Spa \bZ_p[p^{1/p^\infty}]_p^{\wedge} = \widetilde{U} \times_U U'$ is also affinoid perfectoid and by Proposition \ref{prop:surjectivity_on_valuations} an arc-cover of $U'$. By Lemma \ref{lm:perfd_descent} vector bundles descend along $\widetilde{U}' \rar U'$. Now the last paragraph of the proof of \cite[Prop.~19.5.3]{ScholzeW_20} applies literally.
\end{proof}

\subsection{Proof of Theorem \ref{thm:LX_is_vsheaf}}\label{sec:loop_are_arc_v}

First assume that $X = \bP_k^n$. Then the fact that $LX$ is an arc-sheaf is a consequence of descent of vector bundles, i.~e., Proposition \ref{prop:descent_vector_bundles_arc}. The general case follows from this special case and Lemma \ref{lm:LX_arc_reduction_to_PPn}.

\begin{lm}\label{lm:LX_arc_reduction_to_PPn}
Let $\iota \colon Y \rar X$ be an immersion of $k$-schemes. If $LX$ is an arc-sheaf, then $LY$ also is.
\end{lm}

\begin{proof}
Let $R \rar R'$ be an arc-cover in $\Perf$. Write $W = \Spec \bW(R)[1/\varpi]$, $W' = \Spec \bW(R')[1/\varpi]$ and let $f \colon W' \rar W$ be the corresponding morphism. We must show that $Y(W) = {\rm Eq}(Y(W') \rightrightarrows Y(W' \times_W W'))$, assuming the same holds for $X$. As $Y \rar X$ is an immersion, we have $Y(W) \subseteq X(W)$ and similarly for $W', W' \times_W W'$. The lemma thus reduces to show that whenever we have a commutative diagram
\[
\xymatrix{
W' \ar[r]^f \ar[d] & W \ar[d]^\beta \\
Y \ar@{^(->}[r]^{\iota} & X 
}
\]
of $k$-schemes, there exists a (necessarily unique) $k$-morphism $W \rar Y$ making the diagram commute. But $W$ is reduced by Lemma \ref{lm:Witt_reduced} and $\iota$ is an immersion, thus it suffices to check that $\beta(W) \subseteq Y$ set-theoretically.
By factoring $\iota$, we may assume  that it is either closed or open. Assume first $\iota$ is closed. By Proposition \ref{lm:dominance_arc_Witt}, $W$ contains a dense subset $D$, which maps into $Y$, and then we are done as $\beta(W) = \beta(\overline D) = \overline{\beta(D)} \subseteq \overline{Y} = Y$. Assume now $\iota$ is open. 
By Corollary \ref{cor:arc_cover_surj_on_closed_primes} all closed points of $W$ are mapped into $Y$. As $W$ is affine, any point $w \in W$ specializes to some closed point $w'$. Then $\beta(w)$ specializes to $\beta(w') \in Y$. As $Y$ is open in $X$, it is stable under generalization, hence $\beta(w) \in Y$.
\end{proof}

\section{Extension of vector bundles and loop spaces of Grassmannians}
We work in the setup of \S\ref{sec:setup_general}. Let $A \in \Perf_\kappa$ be a ring such that each connected component of $\Spec A$ is the spectrum of a valuation ring. The set $T = \pi_0(\Spec A)$ of connected components is profinite, and the natural morphism $\Spec A \rar \underline{T}_{\kappa}$ corresponds to an inclusion ${\rm Cont}(T,\kappa) \har A$. Taking Witt-vectors, inverting $\varpi$ and using Lemma \ref{lm:Witt_of_continuous_maps}, we obtain the composed map (notation is as in \S \ref{sec:schemes_for_profinite_sets}):
\begin{align}\label{eq:morphism_to_spec_of_discretes}
\pi \colon \Spec\bW(A)[1/\varpi] \rar \Spec {\rm Cont}_{\varpi}(T,k) \rar \Spec {\rm Cont}_{\rm disc}(T,k) = \underline{T}_k,
\end{align}
where ${\rm Cont}_{\rm disc}(T,k)$ denotes the continuous functions with respect to the discrete topology on $k = \bW(\kappa)[1/\varpi]$, and the last map results from the natural inclusion ${\rm Cont}_{\rm disc}(T,k) \subseteq {\rm Cont}_{\varpi}(T,k)$. 
% We also have a sequence similar to \eqref{eq:morphism_to_spec_of_discretes} without inverting $\varpi$. 
The main result of this section is the following theorem.

\begin{thm}\label{lm:no_torsors}
Let $A \in \Perf_\kappa$ be a ring such that each connected component of $\Spec A$ is the spectrum of a valuation ring, $T = \pi_0(A)$ and $\pi$ as in \eqref{eq:morphism_to_spec_of_discretes}. For any finite locally free $\bW(A)[1/\varpi]$-module $M$, there is a finite disjoint clopen decomposition $T = \coprod_{i=1}^r T_i$, such that $M|_{\pi^{-1}(\underline{T_i}_k)}$ is free. In particular, if $M$ has constant rank, then it is free.
\end{thm}

This theorem is proven in \S \ref{sec:proof_thm_no_torsors}. First we reduce to the case of a valuation ring, by using Noetherian approximation and a result of Gabber--Ramero, saying that the category of modules does not change when we pass to completion of a Henselian pair. In the case of a valuation ring, the Beauville--Laszlo lemma along with Noetherian approximation allow to extend our vector bundle around the generic point of the $(\varpi = 0)$-locus, which -- along with arc-descent for vector bundles -- reduces the theorem to the case of a (microbial) valuation ring, where it is an (immediate consequence of a) result of Kedlaya \cite{Kedlaya_16_ringAinf}.

\begin{rem}\label{rem:Bouthier_Cesnavicius}
In the equal characteristic case, Theorem \ref{lm:no_torsors} for line bundles is shown for all seminormal Henselian local rings $A$ in \cite[Cor.~3.1.5]{BouthierC_19}. In particular, in the equal characteristic case the first claim of Corollary \ref{thm:arc_exactness_and_splitting} below holds already for the \'etale topology. 
\end{rem}

\begin{cor}\label{cor:no_torsors_GLn_coh_zero}
Let $A$ be as in Theorem \ref{lm:no_torsors}. For any $n\geq 1$, $H^1_{\rm et}(\Spec \bW(A)[1/\varpi], \GL_n) = 0$. 
\end{cor}

We now can deduce the corollary needed for our application. 

\begin{cor}\label{thm:arc_exactness_and_splitting}
Let $G$ be a unramified reductive group over $k$ and let $B$ be a $k$-rational Borel subgroup. The sequence 
\[ 
1 \rar LB \rar LG \rar L(G/B) \rar 1
\] 
is exact for the $v$-topology on $\Perf_\kappa$. The same holds for any $k$-rational parabolic subgroup $B$ of a reductive $k$-group $G$, which splits over $\breve k$, if all Levi factors of $B$ are of the form $\GL_m$.
\end{cor}

Note that the statement of the corollary makes sense, as $G/B$ is projective over $k$, so that $L(G/B)$ is an arc-sheaf (and hence also $v$-sheaf) by Theorem \ref{thm:LX_is_vsheaf}. 

\begin{proof}[Proof of Corollary \ref{thm:arc_exactness_and_splitting}]
We must show that the right map is surjective in the $v$-topology. By Lemma \ref{lm:standard_vcover} it suffices to show that $LG(A) \rar L(G/B)(A)$ is surjective for all $A \in \Perf_{\kappa}$ as in Theorem \ref{lm:no_torsors}. As $G$ is split over $\breve k$, we may by enlarging $A$ assume that $G$ is split over $\bW(A)[1/\varpi]$. The sequence
\[1 \rar B \rar G \rar G/B \rar 1 \]
of sheaves of pointed sets on the \'etale site of $\Spec W(A)[1/\varpi]$ is exact by \cite[XXII, 5.8.3]{SGA3}. Taking non-abelian cohomology we get the exact sequence,
\[
1 \rar B(\bW(A)[1/\varpi]) \rar G(\bW(A)[1/\varpi]) \rar (G/B)(\bW(A)[1/\varpi]) \rar H^1_{\rm et}(\Spec \bW(A)[1/\varpi], B).
\]
where the first three terms are equal to $LB(A)$, $LG(A)$ and $L(G/B)(A)$. It remains to show that $H^1_{\rm et}(\Spec \bW(A)[1/\varpi], B) = 0$. We have $B = TU$ with $U$ the unipotent radical of $B$, and $T$ a split torus. Now, $U$ is split (cf. the proof of Lemma \ref{lm:coh_vanishes} below), so has a composition series with subquotients isomorphic to $\bG_a$, and $H^1_{\rm et}(S,\bG_a) = 0$ on any affine base $S$.
We deduce $H^1_{\rm et}(\Spec \bW(A)[1/\varpi],U) = 0$ and it suffices to show that $H^1_{\rm et}(\Spec \bW(A)[1/\varpi],T) = 0$, which is Corollary \ref{cor:no_torsors_GLn_coh_zero}. This shows the first claim of the corollary, and the second has the same proof. \qedhere
\end{proof}

Let us record the following consequence of Theorem \ref{lm:no_torsors} and Lemma \ref{lm:standard_vcover}.

\begin{cor}\label{cor:projective_space_as_quotient}
For $n\geq 1$, the natural map $L(\bA^{n+1} \sm \{0\}) \rar L\bP^n$ of $v$-sheaves on $\Perf_\kappa$ is surjective. Thus, $L\bP^n$ is equal to the $v$-quotient of $L(\bA^{n+1} \sm \{0\})$ by the scalar action of $L\bG_m$.
\end{cor}

\subsection{Proof of Theorem \ref{lm:no_torsors}}\label{sec:proof_thm_no_torsors}

We write $X = \Spec \bW(A)$, $U = \Spec \bW(A)[1/\varpi]$ and let $M$ be a finite locally free $\bW(A)[1/\varpi]$-module.

\subsubsection{Reduction to the case of a valuation ring} \label{sec:notorsors_step4} 
% The set $T = \pi_0(\Spec A)$ of connected components is profinite, and the natural morphism $\Spec A \rar \underline{T}_{\kappa}$ corresponds to an inclusion ${\rm Cont}(T,\kappa) \har A$. Taking Witt-vectors, inverting $\varpi$ and applying Lemma \ref{lm:Witt_of_continuous_maps} give maps (notation is as in Section \ref{sec:schemes_for_profinite_sets}):
% \begin{align}\label{eq:morphism_to_spec_of_discretes}
% U \rar \Spec \bW({\rm Cont}(T,\kappa))[1/\varpi] = \Spec {\rm Cont}_{\varpi}(T,k) \rar \Spec {\rm Cont}_{\rm disc}(T,k) = \underline{T}_k,
% \end{align}
% where ${\rm Cont}_{\rm disc}(T,k)$ denotes the continuous functions with respect to the discrete topology on $k = \bW(\kappa)[1/\varpi]$, and the last map results from the natural inclusion ${\rm Cont}_{\rm disc}(T,k) \subseteq {\rm Cont}_{\varpi}(T,k)$. 
For  each $\tau \in T$, we have the corresponding evaluation map ${\rm Cont}(T,\kappa) \rar \kappa$, and global sections of the fiber of $\Spec A \rar \underline T_\kappa$ over $\tau$  are
\begin{equation}\label{eq:Atau_colimit}
A_\tau=A \otimes_{{\rm Cont}(T,\kappa)} \kappa = \dirlim_V (A \otimes_{{\rm Cont}(T,\kappa)} {\rm Cont}(V,\kappa)), 
\end{equation}
the filtered colimit taken over all open neighborhoods $V$ of $\tau$ in $T$. By assumption, $A_\tau$ is a valuation ring. The map $A \tar A_\tau$ induces a map $\bW(A) \tar \bW(A_\tau)$, which factors through the global sections of the fiber of the map as in \eqref{eq:morphism_to_spec_of_discretes} (but without inverting $\varpi$) corresponding to $\tau$:
\begin{equation}\label{eq:WAtau_factors_through_fiber}
\bW(A) \stackrel{\alpha}{\tar} \bW(A) \otimes_{{\rm Cont}(T,\caO_k), \tau} \caO_k \stackrel{\beta}{\tar} \bW(A_\tau).   
\end{equation}

\begin{lm}\label{lm:WAtau_is_padic_completion_of_fiber}
$\bW(A_\tau)$ coincides with the $\varpi$-adic completion of $\bW(A) \otimes_{{\rm Cont}(T,\caO_k), \tau} \caO_k$ via the map in \eqref{eq:WAtau_factors_through_fiber}.
\end{lm}
\begin{proof}
We have to show that the $\varpi$-adic closure of 
\[
\ker(\alpha) = \ker(\ev_\tau \colon {\rm Cont}_{\rm disc}(T, \caO_k)\rar \caO_k)\cdot \bW(A)
\] 
in $\bW(A)$ is equal to $\ker(\beta \alpha) = \left\{\sum_{i=0}^\infty [a_i]\varpi^i \colon a_i \in \ker(A \tar A_{\tau}) \right\}$. But if $a \in \ker(A \rar A_\tau)$, then \eqref{eq:Atau_colimit} shows that there exists some $V_a \subseteq T$ open such that $a \in \ker(A \rar A_{\tau'})$ for all $\tau' \in V_a$. As $T$ is profinite, we may (and do) assume that $V_a$ is also closed (by shrinking it, if necessary). Given $x = \sum_{i=0}^\infty [a_i]\varpi^i \in \ker(\beta\alpha)$, and $N > 0$, let $V_N = \bigcap_{0\leq i < N} V_{a_i}$, and let $\chi_N \colon T \rar \caO_k$ be the characteristic function of $V_N$. It is in ${\rm Cont}_{\rm disc}(T,\caO_k)$ as $V_N$ open and closed. For each $N > 0$, regarding $\chi_N$ as an element of $\bW(A)$, we have $\chi_N x \in \ker(\alpha)$ and $\chi_N x \equiv x \mod \varpi^N \bW(A)$, so that $\chi_Nx \rar x$ for $\varpi$-adic topology.
\end{proof}

Suppose now Theorem \ref{lm:no_torsors} is proven for all valuation rings. Then $M \otimes_{\bW(A)[1/\varpi]} \bW(A_\tau)[1/\varpi]$ is free. By \cite[5.4.42]{GabberR_03} and Lemma \ref{lm:WAtau_is_padic_completion_of_fiber}, its ``decompletion'' $(M \otimes_{{\rm Cont}_{\rm disc}(T,\caO_k), \tau} \caO_k)[1/\varpi]$ is a free $(\bW(A) \otimes_{{\rm Cont}(T,\caO_k), \tau} \caO_k)[1/\varpi]$-module. This latter ring is the filtered colimit of the global sections 
% rings $\bW(A)[1/\varpi] \otimes_{{\rm Cont}(T,k), \tau} {\rm Cont}(U,k)$ of global sections 
of $\Spec \bW(A)[1/\varpi] \times_{\underline T_k} \underline V_k$, where $V$ goes through open neighbourhoods of $\tau$ in $T$. By Noetherian approximation \cite[Tag 01ZR]{StacksProject}, a fixed trivialization of $(M \otimes_{{\rm Cont}_{\rm disc}(T,\caO_k), \tau} \caO_k)[1/\varpi]$ comes from trivialization of the restriction of $M$ to the open subscheme $U \times_{\underline T_k} \underline V_k$ of $U$ for a sufficiently small open $V$. For varying $\tau \in T$, the opens $V$ form a covering of $T$. Now, using that $T$ is profinite, we are done by \cite[08ZZ]{StacksProject}.

\subsubsection{Extension to the generic point of the $X \sm U$}\label{sec:notorsors_step1}
It remains to prove the theorem for a valuation ring $A$. Let $\fm_A$ be the maximal ideal of $A$ and let $K = {\rm Frac}(A)$. We regard $M$ as a locally free $\caO_U$-module. Let $R = \dirlim_t \bW(A)[\frac{1}{[t]}]$ and write $X_0 = \Spec R$ and $U_0 = \Spec R[1/\varpi]$. The $\varpi$-adic completion of $R$ is $\bW(K)$. We have the finite locally free $\caO_{U_0}$-module $M_2 = M \otimes_{\caO_U} \caO_{U_0}$, and by the Beauville--Laszlo lemma \cite{BeauvilleL_95}, to give a finite locally free $\caO_{X_0}$-module $\cM$ with $\cM \otimes_{\caO_{X_0}} \caO_{U_0} = M_2$ is the same as to give a finite (locally) free $\bW(K)$-submodule $M_1$ in the finite dimensional $W(K)[1/\varpi]$-vector space $M_2 \otimes_R \bW(K) = M \otimes_{\bW(V)} \bW(K)$. As $\bW(K)$ is a discrete valuation ring, this is always possible. Fix such an $M_1$ and let $\cM$ be the corresponding finite locally free $\caO_{X_0}$-module.
%another ref for BL-lemma: (see \cite[5.2.9]{ScholzeW_20})

Then $\cM$ glues with $\tilde M$ to a vector bundle on $X_0 \cup U$, and Theorem \ref{lm:no_torsors} now follows from \cite[Thm.~2.7]{Kedlaya_16_ringAinf}. Nevertheless, below we explain how to reduce Theorem \ref{lm:no_torsors} to the special case of \cite[Thm.~2.7]{Kedlaya_16_ringAinf} for microbial valuation rings\footnote{Recall that a valuation ring is called microbial, if it possesses a prime ideal of height $1$.}. Our argument differs from that in \cite{Kedlaya_16_ringAinf}, and we believe that it is interesting in its own right.

\subsubsection{Noetherian approximation}\label{sec:notorsors_step2}
For $t \in \fm_A \sm \{0\}$, let $X_t = \Spec \bW(A)[\frac{1}{[t]}]$ and let $U_t = X_t \sm \{\varpi = 0\} = \Spec \bW(A)[\frac{1}{[t]}, \frac{1}{\varpi}]$. We have $\prolim_t X_t = X_0$ and $\prolim_t U_t = U_0$ with all appearing schemes affine. Let $p_t \colon X_0 \rar X_t$ denote the natural map. By \cite[Tag 01ZR]{StacksProject} there is some $t$ and some finitely presented $\caO_{X_t}$-module $\cM_t$  such that $\cM \cong p_t^\ast(\cM_t)$. By \cite[Tag 02JO]{StacksProject} we may, after shrinking $t$ (by this we mean shrinking $X_t$), assume that $\cM_t$ is $\caO_{X_t}$-flat -- and, consequently, locally free -- $\caO_{X_t}$-module. On the other side, for each $t'$, we have the locally free $\caO_{U_{t'}}$-module $M_{t'} = M|_{U_t'}$ satisfying $(p_{t'}|_{U_{t'}})^\ast(M_{t'}) \cong M_2$. Again, by \cite[Tag 01ZR]{StacksProject}, the isomorphism 
\[ 
(p_t|_{U_t})^\ast(\cM_t[1/\varpi]) = p_t^\ast(\cM_t)[1/\varpi] = \cM[1/\varpi] \cong M_2 \cong (p_t|_{U_t})^\ast(M_t) 
\]
on $\Spec(R[1/\varpi])$ must come from some finite level, i.e., after shrinking $t$ further, this comes from an isomorphism of $\caO_{U_t}$-modules $\cM_t[1/\varpi] \cong M_t$. Glueing $\cM_t$ (on $X_t$) with $M$ (on $U$) along this isomorphism over $U_t$, we have extended $M$ to a vector bundle $\cM_t^{(1)}$ on $X_t \cup U$ for some $t \in \fm_A \sm \{0\}$.

Now, if $A$ is microbial, we are done by the microbial case of \cite[Thm.~2.7]{Kedlaya_16_ringAinf} (see also \cite[Prop.~14.2.6]{ScholzeW_20}): indeed, it shows that $\cM_t^{(1)}$ extends to all of $X$, and hence is a trivial. We may thus assume in the following that $A$ is not microbial.

\subsubsection{Glueing along an arc-cover}\label{sec:notorsors_step3} Find some $\fp \subset \fq \in \Spec A$ with $t \not\in \fp,\fq$, such that $\fp \rightsquigarrow \fq$ is an immediate specialization (this is always possible, see \cite[Rem.~2.2]{BhattM_18}). Then $A \rar A_\fp \times A/\fp$ is an arc-cover \cite[Cor.~2.9]{BhattM_18}, and $A/\fp$ is microbial. Let $\cN$ denote the restriction of $\cM_t^{(1)}$ to $\Spec \bW(A_\fp)$, and consider the restriction of $\cM_t^{(1)}$ to $\Spec \bW(A/\fp) \cap (X_t \cup U)$. By the microbial case of \cite[Thm.~2.7]{Kedlaya_16_ringAinf} , the latter extends to a vector bundle $\cN'$ on $\Spec\bW(A/\fp)$ (which is necessarily trivial). We may now glue $\cN$ and $\cN'$ along $\Spec \bW({\rm Frac}(A/\fp))$, on which both are defined and agree.\footnote{Indeed, $S \mapsto \{\text{vector bundles on } \bW(S)\}$ is a stack for arc-topology: for $v$-topology this is \cite[Thm.~4.1]{BhattS_17}, and for the arc-topology the same proof applies (cf. \S\ref{sec:arc_descent}).} This gives a vector bundle $\cN''$ on $X = \Spec \bW(A)$, which is necessarily trivial, and whose restriction to $U$ is isomorphic to $M$. As $\cN''$ is trivial, also $M$ is trivial, and we are done.

\section{Fixed points on loop spaces of partial flag manifolds}

In this section we work in the setup of \S \ref{sec:setup_over_finite_field}. Moreover, we fix a reductive group $G$ over the local field $k = \bW(\bF_q)[1/\varpi]$.

\subsection{$\sigma$-conjugacy classes}\label{sec:Kottwitz_review} We review some results from \cite{Kottwitz_85}, which we need below. Let $\ff$ be any algebraically closed extension of $\obF$. Then $L = \bW(\ff)[1/\varpi]$ is an extension of $\breve k = \bW(\obF)[1/\varpi]$, and the Frobenius automorphism of $\ff$ over $\bF_q$ induces an automorphism $\sigma$ of $L$ over $k$,  so that $L^\sigma = k$ \cite[Lm.~1.2]{Kottwitz_85}. Attached to the reductive group $G$ over $k$ Kottwitz defines\footnote{The definition is only given in the case $\charac k = 0$, but it works similarly in the case $\charac k > 0$.} the set $B(G) = H^1(\langle \sigma \rangle, G(L))$. Concretely, $B(G)$ is the quotient of $G(L)$ modulo $\sigma$-conjugacy: $x$ is $\sigma$-conjugate to $y$ if there exists $g \in G(L)$ such that $g^{-1}x\sigma(g) = y$. We denote the $\sigma$-conjugacy class of $b\in G(L)$ by $[b]$ resp. by $[b]_G$, if we want to specify the ambient group $G$. The set $B(G)$ is independent of the choice of $\ff$. 

Assume now that $G$ is unramified, and fix a $k$-rational maximal torus $T$ of $G$, which is contained in a $k$-rational Borel subgroup. The set $B(G)$ can be parametrized as follows.
Let $\pi_1(G)$ denote the Borovoi fundamental group of $G$, which is isomorphic to the quotient of $X_\ast(T)$ by the coroot lattice. Then one can attach to $[b] \in B(G)$ two invariants, the \emph{Kottwitz point} $\kappa_G(b) \in \pi_1(G)_{\Gal(k^{\rm sep}/k)}$ and the \emph{Newton point} $\nu_b \in (W \backslash X_\ast(T)_\bQ)^{\Gal(k^{\rm sep}/k)}$. Then the map 
\[
(\nu, \kappa_G) \colon B(G) \har (W \backslash X_\ast(T)_\bQ)^{\Gal(k^{\rm sep}/k)} \times \pi_1(G)_{\Gal(k^{\rm sep}/k)}
\]
is injective. Moreover, the image of $\nu_b$ and $\kappa_G(b)$ in $\pi_1(G)_{\Gal(k^{\rm sep}/k)} \otimes_{\bZ} \bQ$ coincide (thus, if $\pi_1(G)_{\Gal(k^{\rm sep}/k)}$ is torsion free, a $\sigma$-conjugacy class is determined by its Newton point). 

\begin{lm}\label{lm:fibers_BPBG}
Let $P \subseteq G$ be a $k$-rational parabolic subgroup of $G$. The fibers of the natural map $B(P) \rightarrow B(G)$ are finite. 
\end{lm}
\begin{proof}
This follows from the above description and its functoriality.
\end{proof}

\subsection{Sheaf of $b\sigma$-fixed points}\label{sec:bsigma_fixed_points}

First we recall the following definition from \cite[1.12]{RapoportZ_96} (see also \cite[3.3, Appendix A]{Kottwitz_97}). Let $H$ be any  linear algebraic group over $k$ and let $b \in H(\breve k)$. Let $\ff \in \Perff$ be an algebraically closed field, $L = \bW(\ff)[1/\varpi]$ and $\sigma$ be as in \S \ref{sec:Kottwitz_review}. For $b \in H(\breve k)$, let $H_b$ denote the functor on $k$-algebras, 
\begin{equation}\label{eq:form_of_Levi}
R \mapsto H_b(R) = \{g \in H(R \otimes_k L) \colon g(b\sigma) = (b\sigma)g \}.
\end{equation}
This functor is representable by an affine smooth group over $k$, and moreover, the definition is independent of $\ff$, in the sense that if $H_b'$ denotes the group $H_b$ defined with respect to $\ff = \obF$, then $H_b' \cong H_b$. (In \cite[1.12]{RapoportZ_96} only the mixed characteristic case is considered, but the equal characteristic case works similarly).

We come back to our unramified reductive group $G$. Until the end of this section fix a $k$-rational parabolic subgroup $P$ of $G$. We have the projective $k$-scheme $G/P$. We denote its base change to $\breve k$ again by $G/P$, so that $L(G/P)$ is an arc-sheaf on $\Perf_{\obF}$ (by Theorem \ref{thm:LX_is_vsheaf}). The $k$-rational structure on $G/P$ gives the geometric Frobenius $\sigma$ on $L(G/P)$ (as at the end of \S\ref{sec:loop_spaces}). For $b \in G(\breve k)$, we can consider the arc-sheaf 
\begin{equation}\label{eq:LGP_bsigma_sheaf}
L(G/P)^{b\sigma} := {\rm Eq}\left(L(G/P) \doublerightarrow{\id}{b\sigma} L(G/P) \right),
\end{equation}
where $b\sigma$ is the automorphism of $L(G/P)$ induced by $gP \mapsto b\sigma(g)P$. We will show below that its is represented by the constant scheme attached to a profinite set. First we study the geometric points of $L(G/P)^{b\sigma}$.

\begin{prop}\label{prop:geometric_profiniteness_of_bsigma_fixed_points}
Let $\ff \in \Perff$ be any algebraically closed field and let $L = \bW(\ff)[1/\varpi]$. Let $b \in G(\breve k)$. Then the following hold:
\begin{itemize}
\item[(i)] If $[b]_G \cap P(\breve k) = \varnothing$, then $(G/P)(L)^{b\sigma} = \varnothing$.
\item[(ii)] If $b \in P(\breve k)$, then $(G/P)(L)^{b\sigma} = (G/P)(\breve k)^{b\sigma}$, and this set can naturally be identified with the set of $k$-rational points of a projective scheme over $k$. In particular, it is a profinite set with respect to the $\varpi$-adic topology.
\end{itemize}
\end{prop}

Let $k^{\nr}$ denote the maximal unramified extension of $k$.

\begin{lm}\label{lm:coh_vanishes}
With notation as in Proposition \ref{prop:geometric_profiniteness_of_bsigma_fixed_points}, let $F$ be $L$ or $k^{\nr}$. Then $H^1(F^{\rm sep}/F, P_b) = 1$.
\end{lm}
\begin{proof}
If $\charac k = 0$, then $F$ is perfect and the result follows directly from Steinberg's theorem, as $\cd(F) \leq 1$ \cite[II.3.3 c)]{Serre_Galois_Cohomology}. If $\charac k > 0$, we need a small argument. The group $P_b$ is a $k$-rational parabolic subgroup of the (connected) reductive group $G_b$. Hence the unipotent radical $U$ of $P_b$ is defined over $k$ and split \cite[3.14,3.18]{BorelT_65}, i.e., has a composition series over $k$ with all subquotients isomorphic to $\bG_a$. As $H^1(F^{\rm sep}/F,\bG_a) = 1$ (see e.g. \cite[II.1.2 Prop.~1]{Serre_Galois_Cohomology}), we deduce  $H^1(F^{\rm sep}/F, U) = 1$. Now $P_b/U$ is a connected reductive $k$-group, and as $\cd(F) \leq 1$ (see \cite[II.3.3 c)]{Serre_Galois_Cohomology}), the extension \cite[8.6]{BorelS_68} due to Borel--Springer of Steinberg's theorem shows that $H^1(F^{\rm sep}/F, P_b/U) = 1$. Combining these two vanishing results, the lemma follows.
\end{proof}

\begin{proof}[Proof of Proposition \ref{prop:geometric_profiniteness_of_bsigma_fixed_points}]
(i): By Lemma \ref{lm:coh_vanishes} (applied to $b = 1$), the natural map $G(L) \rar (G/P)(L)$ is surjective. Suppose $(G/P)(L)^{b\sigma} \neq \varnothing$. Thus there exists $g \in G(L)$ such that $b' := g^{-1}b\sigma(g) \in P(L)$. Now $[b']_P \in B(P) = H^1(\langle \sigma \rangle, P(L)) = H^1(\langle \sigma \rangle, P(\breve k))$ (the equality follows from \cite{Kottwitz_85}). With other words, there is a representative $b'' \in P(\breve k)$ of $[b']_P$. We deduce $[b]_G = [b'']_G \in B(G)$, so that $[b]_G \cap P(\breve k) \neq \varnothing$. This shows (i).

\noindent (ii): We have the $\Gal(k^{\rm sep}/k)$-equivariant short exact sequence of discrete (with respect to $\Gal(k^{\rm sep}/k)$-action) pointed sets,
\[
1 \rightarrow P_b(k^{\rm sep}) \rightarrow G_b(k^{\rm sep}) \rightarrow (G_b/P_b)(k^{\rm sep}) \rightarrow 1.
\]
Taking cohomology with respect to the action of the subgroup $\Gal(k^{\rm sep}/k^{\nr})$, and using Lemma \ref{lm:coh_vanishes}, we deduce the exact  sequence of discrete pointed $\Gal(k^{\nr}/k)$-sets,
\begin{equation}\label{eq:first_seq}
1 \rightarrow P_b(k^{\nr}) \rightarrow G_b(k^{\nr}) \rightarrow (G_b/P_b)(k^{\nr}) \rightarrow 1.
\end{equation}

As in \cite[1.3]{Kottwitz_85} let $\caW(L^{\rm sep}/k)$ (resp. $\mathcal{W}(L/k)$) be the group of continuous automorphisms of $L^{\rm sep}$ (resp. $L$) fixing $k$ pointwise, which induce on the residue field an integral power of the Frobenius automorphism. Consider the $1$-cocycle $\tau \mapsto c_{\tau} \colon \caW(L^{\rm sep}/k) \rightarrow P(L^{\rm sep})$, which is trivial on the subgroup of elements of $\caW(L^{\rm sep}/k)$  fixing $\breve k$, and which is then determined by $c_\sigma = b$. Composing with the embedding $P(L^{\rm sep}) \har G(L^{\rm sep})$, this also gives a $1$-cocycle with values in $G(L^{\rm sep})$. Consider the actions of $\caW(L^{\rm sep}/k)$ on $P(L^{\rm sep})$ and $G(L^{\rm sep})$ twisted by these 1-cocycles, that is $\tau \in \caW(L/k)$ acts by $\tau^\ast(g) = c_\tau \tau(g)c_\tau^{-1}$. We have the short exact sequence of the pointed $\caW(L^{\rm sep}/k)$-sets with respect to this twisted action,
\[
1 \rightarrow  P(L^{\rm sep}) \rightarrow G(L^{\rm sep}) \rightarrow (G/P)(L^{\rm sep}) \rightarrow 1
\]
Taking the cohomology with respect to the subgroup $\Gal(L^{\sep}/L)$ of $\caW(L^{\rm sep}/k)$ and applying Lemma \ref{lm:coh_vanishes} again, we deduce the short exact sequence of $\caW(L/k)$-pointed sets,
\begin{equation}\label{eq:second_seq}
1 \rightarrow  P(L) \rightarrow G(L) \rightarrow (G/P)(L) \rightarrow 1,
\end{equation}
Let $B(G)' := H^1(\caW(L^{\rm sep}/k),G(L^{\rm sep})) = H^1(\caW(L/k),G(L))$ with respect to the twisted action. Then 
\[
B(G) \rar B(G)', \quad [g]_G \mapsto \text{cocycle determined by $\sigma \mapsto gb^{-1}$}
\]
is a bijection of sets (cf. \cite[I.5.3, Prop.~35]{Serre_Galois_Cohomology}), i.e., the pointed set $B(G)'$ can be identified with the set $B(G)$, but with distinguished element $[b]_G$. The same works for $B(P)'$ and $B(P)$.

By construction (cf. \cite[3.3]{Kottwitz_97}), we have $P_b(k^{\rm nr}) \subseteq P(L)$ and $G_b(k^{\nr}) \subseteq G(L)$, and these inclusions are equivariant with respect to the natural restriction map $\caW(L/k) \rightarrow \Gal(k^\nr/k)$, where we consider the twisted $\caW(L/k)$-action on $P(L)$, $G(L)$. Thus we deduce a map from the exact sequence \eqref{eq:first_seq} to \eqref{eq:second_seq}, which is equivariant with respect to $\caW(L/k) \rightarrow \Gal(k^\nr/k)$.
Using the functoriality of the long exact cohomology sequence we deduce the commutative diagram of pointed sets,
\[
\xymatrix{
1 \ar[r] & P_b(k) \ar@{=}[d] \ar[r] & G_b(k) \ar@{=}[d] \ar[r] & (G_b/P_b)(k) \ar[d] \ar[r] & H^1(k,P_b) \ar@{^(->}[d] \ar[r] & H^1(k,G_b) \ar@{^(->}[d] \\
1 \ar[r] & P(L)^{b\sigma(\cdot)b^{-1}} \ar[r] & G(L)^{b\sigma(\cdot)b^{-1}} \ar[r] & (G/P)(L)^{b\sigma} \ar[r] & B(P)' \ar[r] & B(G)'
}
\]
where the two left vertical arrows are bijections by \cite[1.12]{RapoportZ_96}, and the two right vertical arrows are the injective maps as in \cite[(3.5.1)]{Kottwitz_97}. By Lemma \ref{lm:fibers_BPBG} the fiber in $B(P)'$ over the distinguished point of $B(G)'$ is finite. Repeating the same arguments for all (finitely many) $\sigma$-conjugacy classes in $P(\breve k)$, which are contained in $[b]_G \cap P(\breve k)$, we thus deduce that $(G/P)(L)^{b\sigma}$ is a finite union of copies of $(G_b/P_b)(k)$. In particular, it is independent of the choice of $\ff$. The last claim follows from Lemma \ref{lm:projective_variety_gives_profinite_set}. 
\end{proof}

\begin{lm}\label{lm:projective_variety_gives_profinite_set}
If $X$ is a proper $k$-scheme, then $X(k)$ with the $\varpi$-adic topology is a profinite set.  
\end{lm}
\begin{proof}
As $X/k$ is separated, $X(k)$ is Hausdorff and totally disconnected by \cite[Prop.~5.4]{Conrad_12}. Moreover, by \cite[Cor.~5.6]{Conrad_12}, $X(k)$ is compact, hence profinite.\qedhere

% Old, valid proof:
% First suppose $X$ is projective over $k$. If $X \har Y$ is a closed immersion of $k$-schemes, then $X(k) \har Y(k)$ is a closed immersion in the $\varpi$-adic topology. Thus it is enough to prove the lemma in the case $X = \bP_k^n$, where it follows from $\bP^n(k) = \bP^n(\caO_k) = \prolim_r \bP^n(\caO_k/(\varpi^r))$. 

% Let $X$ be arbitrary proper $k$-scheme. As $X/k$ is separated, it follows from Then $X(k)$ is Hausdorff  From separatedness of $X/k$, and using that $X(k) \times X(k)$ is homeomorphic to $(X\times_k X)(k)$ (this holds $k$-schemes which are affine of finite type by \cite[Lm.~5.4.16(ii)]{GabberR_03}, and extends then to arbitrary $X$ satisfying \eqref{eq:cond_on_X_topologize}, cf. \cite[(2.2.7.1)]{BouthierC_19}), it follows that $X(k)$ is Hausdorff. and totally disconnected, being the union of finitely many closed subsets of $k^n$. By the projective case along with Lemma \ref{lm:application_Chow_lemma}, there is a surjection $P \tar X(k)$ from a profinite set $P$ onto $X(k)$. In particular, $X(k)$ is quasi-compact. It follows that $X(k)$ is profinite.
\end{proof}

% \begin{lm}\label{lm:application_Chow_lemma}
% Let $L$ be a field, and let $Y$ be a proper $L$-scheme. Then there exists a projective $L$-scheme $\widetilde Y$ and a proper morphism $\widetilde Y \rar Y$ which is surjective on $L$-points.
% \end{lm}
% \begin{proof} By Chow's lemma, we may find $\pi \colon Y_1 \rar Y$ proper with $Y_1$ projective $L$-scheme, with the property that that there is an open dense subset $U \subseteq Y$ such that $\pi^{-1}(U) \rar U$ is an isomorphism, and in particular surjective on $L$-points. Now, $Y \sm U$ (equipped with reduced scheme structure) is proper over $L$ and $Y(L) = U(L) \cup (Y \sm U)(L)$. Then we apply the above argument to $Y \sm U$, etc. As $Y$ is of finite type over $L$, the lemma follows by Noetherian induction.
% We can now 
% \end{proof}

\begin{ex}
Let $G = \GL_2$, $T$ the diagonal torus and $P$ a Borel subgroup containing $T$. For $b = \varpi^{(1,-1)}$ we obtain $G_b = P_b \cong T$, so that $(G_b/P_b)(k) = \{ \ast \}$ reduces to a point. On the other side, the fiber of $B(P) \rightarrow B(G)$ over $[b]_G$ consists of the two elements $[\varpi^{(1,-1)}]_P$ and $[\varpi^{(-1,1)}]_P$. 
\end{ex}

As a corollary to the proof of Proposition \ref{prop:geometric_profiniteness_of_bsigma_fixed_points}, we can describe the structure of $(G/P)(\breve k)^{b\sigma}$ more closely. As the fibers of $B(P) \rar B(G)$ are finite, we can write $[b]_G \cap P(\breve k) = \coprod_{i=1}^r [b_i]_P$ with $b_i = g_i^{-1} b \sigma(g_i)$ for some elements $g_i \in G(\breve k)$. Conjugation by $g_i$ defines an isomorphism ${\rm Int}(g_i) \colon G_b(k) \stackrel{\sim}{\rar} G_{b_i}(k)$.

\begin{cor}\label{cor:decomposition_of_set_GPbsigma}
With above notation, $(G/P)(\breve k)^{b\sigma} = \coprod_{i=1}^r G_{b_i}(k)/P_{b_i}(k)$ is a disjoint decomposition into clopen subsets. This decomposition is $G_b(k)$-equivariant, where $G_b(k)$ acts by left multiplication on $(G/P)(\breve k)^{b\sigma}$, and via ${\rm Int}(g_i)$ and left multiplication on $G_{b_i}(k)/P_{b_i}(k)$.
\end{cor}
\begin{proof}
This follows from the long exact sequence of pointed sets in the proof of Proposition \ref{prop:geometric_profiniteness_of_bsigma_fixed_points} applied to each $b_i$.
\end{proof}

In the rest of this section, for a topological space $T$ we will write $\underline{T}$ instead of $\underline{T}_{\obF}$ (see \eqref{eq:functor_cont_functions_top_space}).

\begin{prop}\label{prop:representability_bsigma_fixed_points}
Let $P \subseteq G$ be a $k$-rational parabolic subgroup and $b \in G(\breve k)$. There is a natural isomorphism of arc-sheaves $f \colon \underline{(G/P)(\breve k)^{b\sigma}} \stackrel{\sim}{\longrar} L(G/P)^{b\sigma} $. In particular, if $[b]_G \cap P(\breve k) = \varnothing$, then $L(G/P)^{b\sigma} = \varnothing$.
\end{prop}

\begin{proof}
First, by Corollary \ref{cor:morphisms_from_const_sheaf_into_loop_group}, there is a natural map $\underline{(G/P)(\breve k)^{b\sigma}} \rar L(G/P)$. As $(G/P)(\breve k)^{b\sigma}$ is the set of $b\sigma$-fixed points, one checks that this map factors through a map $f \colon \underline{(G/P)(\breve k)^{b\sigma}} \rar L(G/P)^{b\sigma}$. We have to show that this is an isomorphism. Let $R\in \Perff$ be a valuation ring with algebraically closed fraction field. Write $U = \Spec R$ and let $\eta \in U$ be the generic point. As $(G/P)(\breve k)^{b\sigma}$ is a profinite set and $|U|$ is a chain of specializations, we have
\begin{equation}\label{eq:val_ring_valued_points_constant_sheaf} \underline{(G/P)(\breve k)^{b\sigma}}(R) = {\rm Cont}(|U|, (G/P)(\breve k)^{b\sigma}) = {\rm Cont}(\{\eta\}, (G/P)(\breve k)^{b\sigma}) = \underline{(G/P)(\breve k)^{b\sigma}}(\{\eta\}).
\end{equation}
On the other hand, from Lemma \ref{lm:representability_coincidence_set} it follows that the natural map $L(G/P)(U) \rar L(G/P)(\{\eta\})$ is injective. Hence the same holds for the subsheaf $L(G/P)^{b\sigma}$. 
This observation combined with \eqref{eq:val_ring_valued_points_constant_sheaf} and Proposition \ref{prop:geometric_profiniteness_of_bsigma_fixed_points} implies that $f(U)$ is bijective.

Now, $L(G/P)^{b\sigma}$ is a subsheaf of the quasi-separated $v$-sheaf $L(G/P)$ (Lemma \ref{lm:representability_graph}). Therefore $L(G/P)^{b\sigma}$ is itself quasi-separated. 
Now, $(G/P)(\breve k)^{b\sigma}$ is a profinite set by Proposition \ref{prop:geometric_profiniteness_of_bsigma_fixed_points}. Thus $\underline{(G/P)(\breve k)^{b\sigma}}$ is qcqs $v$-sheaf. Finally, Lemma \ref{lm:check_surj_vsheaves} shows that $f$ is an isomorphism.
\end{proof}

Next, let $H$ be a linear algebraic group over $k$, let $b \in H(\breve k)$ and let $H_b$ be the $k$-group as in \eqref{eq:form_of_Levi}. Then $H_b(k)$ is a locally profinite (and second countable) group, and we have the corresponding arc-sheaf $\underline{H_b(k)}$ on $\Perff$. We also have the automorphism ${\rm Int}(b) \circ \sigma \colon g \mapsto b\sigma(g)b^{-1}$ of $LH$. This gives the equalizer
\[
LH^{b\sigma} := {\rm Eq}\left(LH \doublerightarrow{{\rm Int}(b) \circ \sigma}{\id} LH\right),
\]
which is also an arc-sheaf on $\Perff$. For $R \in \Perf_{\obF}$ and $g \in LH(R)$, regarded as automorphisms of the restriction $LH \times_{\obF} R$ of $LH$ to $\Perf_R$, we have $\sigma g = \sigma(g) \sigma$. Therefore, explicitly
\begin{equation}\label{eq:LGbsigma_R_points}
LH^{b\sigma}(R) = \{g \in LH(R) \colon g b\sigma = b\sigma g\},
\end{equation}
Similar to the above we have a morphism $\underline{H_b(k)} \rar LH^{b\sigma}$. We would like to show that it is an isomorphism. Lemma \ref{lm:check_surj_vsheaves} does not apply directly, as $\underline{H_b(k)}$ is in general not quasi-compact, so we need some more work.

\begin{prop}\label{lm:inclusion_into_LGbsigma}
The natural morphism $f \colon \underline{H_b(k)} \rar LH^{b\sigma}$ is an isomorphism.
\end{prop}

\begin{proof}

The group $H_b \otimes_k \breve k$ is (isomorphic to) a closed subgroup of $H$ \cite[\S3.3]{Kottwitz_97} (our cocycle is trivial on the inertia). If $\sigma_b$ denotes the geometric Frobenius on $LH_b$, we have $(LH_b)^{\sigma_b} = LH^{b\sigma}$, and we have to show that $f \colon \underline{H_b(k)} \rar (LH_b)^{\sigma_b}$ is an isomorphism. Lemma \ref{lm:qc_loc_profin} applied to $X = H_b$, shows that $f \colon \underline{H_b(k)} \rar LH_b$ is quasi-compact. As $(LH_b)^{\sigma_b} \har LH_b$ is a closed immersion, also $f$ is quasi-compact. 

It is easy to check that $f$ induces a bijection on field valued points. Moreover, the same argument as in the proof of Proposition \ref{prop:representability_bsigma_fixed_points} shows now that the bijectivity part of the condition (ii) of Lemma \ref{lm:check_surj_vsheaves} holds for $f$. It remains true after pullback along any $Y \rar (LH_b)^{\sigma_b}$ with $Y$ quasi-compact, hence $f$ is an isomorphism after each such pullback, and the result follows. \qedhere
\end{proof}

\section{Loop Deligne--Lusztig spaces}\label{sec:lDLS}
Now we define loop Deligne--Lusztig spaces. We work in the setup of \S\ref{sec:setup_over_finite_field}. 
We fix an unramified reductive group $G_0$ over $k$ and let $G = G_0 \times_k \breve{k}$ be the (split) base change to $\breve k$. In $G_0$ we fix a $k$-rational maximally split maximal torus $T_0$, which splits after an unramified extension, and a $k$-rational Borel subgroup $B_0 = T_0U_0$ with unipotent radical $U_0$ containing it. Denote by $T,B,U$ the base changes of $T_0,B_0,U_0$ to $\breve k$. Let $W = W(T,G)$ denote the Weyl group of $T$ in $G$. Let $S \subseteq W$ be the set of simple reflections attached to simple roots in $B$. The Frobenius $\sigma$ of $\breve{k}/k$ acts on $S$ and on $W$. For an element $w \in W$ we denote by ${\rm supp}(w)$ the set of simple reflections appearing in a (any) reduced expression of $w$. By $\overline{\supp}(w)$ we denote the  smallest $\sigma$-stable subset of $S$ containing $\supp(w)$.

\subsection{Relative position}\label{sec:rel_pos}

The group $G$ acts diagonally on $G/B \times G/B$ and on $G/U \times G/U$, and by the geometric Bruhat decomposition (for the split group $G$), we have the decomposition into $G$-orbits, 
\[
G/B \times G/B = \coprod_{w \in W} \caO(w) \quad  \text{ and } \quad   G/U \times G/U = \coprod_{\dot w \in N_G(T)(\breve k)} \dot \caO(\dot w),
\]
where the first is a locally closed decomposition into finitely many locally closed subvarieties. The field of definition of $\caO(w)$ is the unramified extension $k_d/k$ of degree $d$, where $d$ is the lowest number such that $\sigma^d(w) = w$, and all $\dot\caO(\dot w)$ are defined over $\breve k$.

\begin{Def}
\begin{itemize}
\item[(i)] For a $\breve k$-scheme $Y$, we say that two $Y$-valued points $g,h \in (G/B)(Y)$ of $G/B$ are \emph{in relative position $w \in W$}, in which case we write $g \stackrel{w}{\rightarrow} h$, if $(g,h) \colon Y \to G/B \times G/B$ factors through $\caO(w)$.
\item[(ii)] For $R \in \Perff$, we say that two $R$-valued points $g,h \in L(G/B)(R) = (G/B)(\bW(R)[1/\varpi])$ are \emph{in relative position $w \in W$}, in which case we write $g \stackrel{w}{\rightarrow} h$, if the corresponding $\bW(R)[1/\varpi]$-valued points of $G/B$ are in relative position $w$.
\end{itemize}
\end{Def}

It follows from the definitions that $g,h \in L(G/B)(R)$ are in relative position  $w$ if and only if $(g,h) \colon \Spec R \rightarrow L(G/B) \times L(G/B)$ factors through $L\caO(w)$.

\begin{lm}\label{lm:aux_factorization_rel_pos}
Let $R \in \Perf_{\obF}$ and $S = \Spec R$. Let $x,y \colon S \rar LG^2$. Then the composition $S \stackrel{(x,y)}{\rar} LG^2 \rar L(G/B)^2$ factors through $L\caO(w) \rar L(G/B)^2$ if and only if $x^{-1}y \colon S \rar LG$ factors through $L(BwB) \rar LG$. A similar claim holds for $G/U$ instead of $G/B$.
\end{lm}
\begin{proof}
Let $\widetilde{S} = \Spec \bW(R)[1/\varpi]$ and let $\tilde x,\tilde y \colon \widetilde S \rar G$ be the maps corresponding to $x,y$. The first (resp. second) statement in the lemma holds if and only if $\widetilde S \stackrel{\tilde x,\tilde y}{\rar} G^2 \rar (G/B)^2$ factors through the locally closed subset $\caO(w) \subseteq (G/B)^2$ (resp. $\widetilde{x^{-1}y} = \tilde x^{-1}\tilde y \colon \widetilde S \rar G$ factors through the locally closed subset $BwB \subseteq G$). By Lemma \ref{lm:Witt_reduced}, $\widetilde S$ is reduced, hence each of these two factorization claims holds if and only if it holds topologically. Thus to show their equivalence it is sufficient to show that they are equivalent when $\widetilde S$ is replaced by $\Spec K$ for an algebraically closed extension $K/\breve k$. This equivalence follows from the Bruhat decomposition. The proof for $G/U$ is the same.
\end{proof}

\subsection{Definition of $X_w(b)$ and $\dot X_{\dot w}(b)$}\label{sec:def_DLV}

The following definition was suggested to the author by P. Scholze.

\begin{Def}\label{def:Xwb_and_covers}
Let $b \in G(\breve k)$, $w \in W$ and $\dot{w} \in N_G(T)(\breve k)$. We define $X_w(b)$ and $\dot X_{\dot w}(b)$ by the following Cartesian diagrams of presheaves on ${\rm Perf}_{\overline{\bF}_q}$:

\centerline{\begin{tabular}{cc}
\begin{minipage}{2in}
\begin{displaymath}
\leftline{
\xymatrix{
X_w(b) \ar[r] \ar[d] & L\caO(w) \ar[d]\\
L(G/B) \ar[r]^-{(\id,b\sigma)} & L(G/B) \times L(G/B) 
}
}
\end{displaymath}
\end{minipage}
& \qquad\qquad  and \qquad\qquad
\begin{minipage}{2in}
\begin{displaymath}
\leftline{
\xymatrix{
\dot X_{\dot w}(b) \ar[r] \ar[d] & L\dot \caO(\dot w) \ar[d]\\
L(G/U) \ar[r]^-{(\id,b\sigma)} & L(G/U) \times L(G/U) 
}
}
\end{displaymath}
\end{minipage}
\end{tabular}
}
\noindent where $b\sigma \colon L(G/B) \stackrel{\sim}{\rar} L(G/B)$, and $b$ acts by left multiplication. If we want to emphasize the group $G$, we write $X_w^G(b)$ resp. $\dot X_{\dot w}^G(b)$. 
\end{Def}

If $\dot w$ lies over $w$, then there is a natural map $\dot X_{\dot w}(b) \rar X_w(b)$. All maps in the Cartesian diagrams in Definition \ref{def:Xwb_and_covers} are injective.
Directly from Theorem \ref{thm:LX_is_vsheaf} (and the fact that the formation of limits commutes with    the inclusion of sheaves into presheaves) we deduce:
\begin{cor}\label{cor:Xwb_arc_sheaf}
$X_w(b)$ and $\dot X_{\dot w}(b)$ are arc-sheaves on $\Perff$. 
\end{cor}

\begin{lm}\label{lm:mult_by_bsigma} Let $w\in W$, $b \in G(\breve k)$. The automorphism $b\sigma \colon L(G/B) \stackrel{\sim}{\rar} L(G/B)$ restricts to an isomorphism $X_w(b) \stackrel{\sim}{\rightarrow} X_{\sigma(w)}(b)$. The same statement holds for $\dot X_{\dot w}(b)$.
\end{lm}
\begin{proof}
Let $R \in \Perff$ and let $Y = \Spec R$. Let $s \in X_w(b)(R) \subseteq L(G/B)(R)$. We have to show that $b\sigma(s)$ and $(b\sigma)^2(s)$  are in relative position $\sigma(w)$, i.e., that $(b\sigma(s), (b\sigma)^2(s)) \colon Y \rar L(G/B)^2$ factors through $L\caO(\sigma(w)) \har L(G/B)^2$. As $s,b\sigma(s)$ are in relative position $w$, we get a commutative diagram
\[
\xymatrixcolsep{5pc}\xymatrix{
Y \ar[r]^-{(s,b\sigma(s))} \ar[rd] & L(G/B) \times L(G/B) \ar[r]^{b\sigma \times b\sigma}& L(G/B) \times L(G/B) \\
& L\caO(w) \ar[r]^{\sim} \ar@{^{(}->}[u] & L\caO(\sigma(w)) \ar@{^{(}->}[u] 
}
\]
(where the lower horizontal arrow is the restriction of the upper horizontal arrow), and the composition of the two upper arrows is $(b\sigma(s),(b\sigma)^2(s))$. This shows that $b\sigma(s),(b\sigma)^2(s)$ are in relative position $\sigma(w)$. Thus we have a morphism $b\sigma \colon X_w(b) \rightarrow X_{\sigma(w)}(b)$  and it has an obvious inverse. The same proof applies to $\dot X_{\dot w}(b)$. \qedhere
\end{proof}

In the rest of this section, for a topological space $X$ we will write $\underline{X}$ instead of $\underline{X}_{\obF}$.

\subsubsection{Action of $\underline{G_b(k)}$}\label{sec:action_of_Gbk} Let $G_b$ be the inner form of a Levi subgroup of $G_0$ over $k$ as in \S\ref{sec:bsigma_fixed_points}. Recall that we have $\underline{G_b(k)} = LG^{b\sigma} \subseteq LG$ by Proposition \ref{lm:inclusion_into_LGbsigma}. Now $LG$ acts on $L(G/B)$ and this action restricts to an action of $\underline{G_b(k)}$ on $L(G/B)$. Similarly, we have an action of $\underline{G_b(k)}$ on $L(G/U)$.

\begin{lm}\label{lm:action_on_Xwb} The action of $\underline{G_b(k)}$ on $L(G/B)$ resp. $L(G/U)$ restricts to an action on $X_w(b)$ resp. $\dot X_{\dot w}(b)$. 
\end{lm}

\begin{proof}
Let $R \in \Perff$ and let $Y = \Spec R$. Let $s \in X_w(b)(R) \subseteq L(G/B)(R)$ and let $g \in \underline{G_b(k)}(R) = LG^{b\sigma}(R)$. Similar as in Lemma \ref{lm:mult_by_bsigma} we have the commutative diagram of sheaves on $\Perf_R$,
\[
\xymatrixcolsep{5pc}\xymatrix{
Y \ar[r]^-{(s,b\sigma(s))} \ar[rd] & L(G/B) \times L(G/B) \ar[r]^{g \times g}& L(G/B) \times L(G/B) \\
& L\caO(w) \ar[r]^{\sim} \ar@{^{(}->}[u] & L\caO(w) \ar@{^{(}->}[u] 
}
\]
Now the composite of the two upper horizontal maps is $(gs,gb\sigma(s)) = (gs,b\sigma(gs)) \colon Y \rar L(G/B)^2$, as $g$ commutes with $b\sigma$ (see \eqref{eq:LGbsigma_R_points}). Thus we see that $(gs,b\sigma(gs))$ factors through $L\caO(w) \har L(G/B)^2$, i.~e., $gs \in X_w(b)(R)$. The proof for $\dot X_{\dot w}(b)$ is similar.
\end{proof}

\begin{rem}\label{rem:change_b}
If $b' = g^{-1}b\sigma(b)$ for $b,b',g \in G(\breve k)$, then $x \mapsto gx$ defines isomorphisms $X_w(b') \stackrel{\sim}{\rar} X_w(b)$ and $\dot X_{\dot w}(b') \stackrel{\sim}{\rar} \dot X_{\dot w}(b)$. In particular, $X_w(b), \dot X_{\dot w}(b)$ depend up to isomorphism only on the $\sigma$-conjugacy class $[b]$ of $b$. Also, conjugation by $g$ defines an isomorphism $\underline{G_{b'}(k)} \stackrel{\sim}{\rar} \underline{G_b(k)}$, with respect to which the above isomorphisms are equivariant. 
\end{rem}

\subsubsection{Action of $\underline{T_w(k)}$}\label{sec:action_torus}
We may consider the $1$-cocycle of the Weil group of $k$ with values in $W$, which is determined by being trivial on the inertia subgroup and sending $\sigma$ to $w$. This determines a form $T_w$ of $T$, which is (isomorphic to) an unramified $k$-rational maximal torus of $G$. We have 
\[
T_w(k) = \{t \in T(\breve k) \colon t^{-1}\dot w \sigma(t) = \dot w\}.
\]
As abelian groups we have $X_\ast(T_w) = X_\ast(T)$, and the action of $\sigma$ on $X_\ast(T_w)$ is given by $\sigma_w := {\rm Ad}(w) \circ \sigma$, where $\sigma$ stands for the Frobenius action on $X_\ast(T)$, induced from the action of the absolute Galois group of $k$ on $X_\ast(T)$ (recall that $T$ is unramified).

We have the sheaf $\underline{T_w(k)}$ attached to the locally profinite set $T_w(k)$, and it is equal to $LT^{{\rm Ad}(w)\circ \sigma}$ (by the same argument as in Proposition \ref{lm:inclusion_into_LGbsigma}). 

\begin{lm}
The right multiplication action of $LT$ on $L(G/U)$ restricts to an action of $\underline{T_w(k)}$ on $\dot X_{\dot w}(b)$. Moreover, the morphism $\dot X_{\dot w}(b) \rar X_w(b)$ is $\underline{T_w(k)}$-equivariant if $X_w(b)$ is equipped with the trivial $\underline{T_w(k)}$-action.
\end{lm}
\begin{proof}
The first claim is proven similarly to Lemma \ref{lm:action_on_Xwb}: if $R \in \Perff$, $Y = \Spec R$ and $s \in \dot X_{\dot w}(R)$, $t \in LT^{{\rm Ad}(w)\circ \sigma}(R)$, then $(st,b\sigma(st)) = (st,b\sigma(s)\sigma(t)) \in L(G/U)^2(R)$ is the composition of $(s,b\sigma(s)) \colon Y \rar L(G/U)^2$ and $t \times \sigma(t) \colon L(G/U)^2 \rar L(G/U)^2$. Now $(s,b\sigma(s))$ factors through $L\dot\caO(\dot w)$ and $t \times \sigma(t)$ preserves the subsheaf $L\dot\caO(\dot w)$. The second claim follows from a similar claim for the projection $L(G/U) \rar L(G/B)$ and the $LT$-action.
\end{proof}

We study the maps $\dot X_{\dot w}(b) \rar X_w(b)$ in detail in \S\ref{sec:Torsors}.

\subsection{Disjoint decomposition} Arguments from \cite[3]{Lusztig_76_Fin} generalize to the loop context. Let $I \subseteq S$ be a $\sigma$-stable subset, $W_I \subseteq W$ the subgroup generated by $I$, and $P_I = B W_I B$ the corresponding standard $k$-rational parabolic subgroup of $G$. 
Recall the arc-sheaf of $b\sigma$-fixed points $L(G/P_I)^{b\sigma}$ from \eqref{eq:LGP_bsigma_sheaf}.

\begin{lm}\label{lm:disj_union}
Let $I \subseteq S$ is $\sigma$-stable subset, and let $w \in W$ with $\supp(w) \subseteq I$. The composition $X_w(b) \har L(G/B) \rightarrow L(G/P_I)$ factors through $L(G/P_I)^{b\sigma} \har L(G/P_I)$.
\end{lm}
\begin{proof}
We have the geometric Bruhat decomposition for $G/P_I$, 
\[
G/P_I \times G/P_I = \coprod_{w \in W_I \backslash W / W_I} \caO_I(w).
\]
As ${\rm supp}(w) \subseteq I$, we have $w \in P_I$, and hence $\caO(w)$ maps into $\caO_I(1)$ under the natural projection $(G/B)^2 \rar (G/P_I)^2$. For an $s \in X_w(b)(R)$ we thus have the commutative diagram
\[
\xymatrixcolsep{3pc}\xymatrix{
\Spec R \ar[r]^-{(s,b\sigma(s))} \ar[rd] & L(G/B) \times L(G/B) \ar[r] & L(G/P_I) \times L(G/P_I) \\
& L\caO(w) \ar[r] \ar@{^{(}->}[u] & L\caO_I(1) \ar@{^{(}->}[u] 
}
\]
With other words, $\Spec R \stackrel{s}{\rightarrow} L(G/B) \rightarrow L(G/P_I)$ factors through $L(G/P_I)^{b\sigma}$.\qedhere
\end{proof}

The following is an immediate consequence of Lemma \ref{lm:disj_union} and Proposition \ref{prop:representability_bsigma_fixed_points}.

\begin{cor}\label{cor:emptyness_Xwb_support}
Let $w \in W$ and $b \in G(\breve k)$. If $[b]_G \cap P_{\overline{\supp}(w)}(\breve k) = \varnothing$, then $X_w(b) = \varnothing$. 
\end{cor}

\begin{rem}\label{rem:wlog_assume_b_in_parabolic}
It follows from Corollary \ref{cor:emptyness_Xwb_support} and  Remark \ref{rem:change_b} that when studying the sheaves $X_w(b)$ and $\dot X_{\dot w}(b)$, we may without loss of generality assume that $b \in P_{\overline{\supp}(w)}(\breve k)$.
\end{rem}

Note that by exactness of \eqref{eq:second_seq} in the case $L = \breve k$, any element in $(G/P_I)(\breve k)^{b\sigma}$ is represented by some $gP_I$ with $g \in G(\breve k)$.

\begin{prop}\label{prop:disj_decomposition}
Let $w \in W$ and write $I = \overline{\supp}(w)$. Let $M_I$ be the unique Levi subgroup of $P_I$ containing $T$. Let $b \in G(\breve k)$. There is a natural morphism $X_w^G(b) \rar \underline{(G/P_I)(\breve k)^{b\sigma}}$. Let $gP_I \in (G/P_I)(\breve k)^{b\sigma}$. The fiber $X_w^G(b)_g$ over the corresponding $\obF$-point of $\underline{(G/P_I)(\breve k)^{b\sigma}}$ is isomorphic to $X_w^{M_I}(g^{-1}b\sigma(g))$. 

Moreover, if $\dot w \in G(\breve k)$ is a lift of $w$, then $\dot w \in M_I(\breve k)$, and the fiber $\dot X_{\dot w}^G(b)_g$ over $gP_I \in (G/P_I)(\breve k)^{b\sigma}$ of the composition $\dot X_{\dot w}^G(b) \rar X_w^G(b) \rar \underline{(G/P_I)(\breve k)^{b\sigma}}$ is isomorphic to $\dot X_{\dot w}^{M_I}(g^{-1}b\sigma(g))$.
\end{prop}

\begin{proof}
The first statement follows from Lemma \ref{lm:disj_union} and Proposition \ref{prop:representability_bsigma_fixed_points}: Let $g \in G(\breve k)$ be some lift of $gP_I$. From $b\sigma(gP_I) = gP_I$ we deduce $b' := g^{-1}b\sigma(g) \in P_I(\breve k)$, and via the natural projection $P_I \tar P_I/{\text{unipotent radical}} \cong M_I$, we regard $b'$ as an element of $M_I(\breve k)$. Let $\pi \colon L(G/B) \rar L(G/P_I)$ denote the natural map. Let $R \in \Perff$ and let $s \colon Y := \Spec R \rar L(G/B)$ be in $X_w(b)_g(R)$, that is $(s,b\sigma(s)) \colon Y \rar L(G/B)^2$ factors through $L\caO(w)$ and $\pi\circ s \colon Y \rar L(G/P_I)$ factors through $Y \rar \Spec \obF \stackrel{gP_I}{\rar} L(G/P_I)$. Then $g^{-1}s \colon Y \rar L(G/B) \rar L(G/B)$ satisfies $\pi \circ g^{-1}s = g^{-1}\pi s = 1\cdot P_I \colon Y \rar L(G/P_I)$, or equivalently, $g^{-1}s$ factors through a section $(g^{-1}s)' \colon Y \rar L(P_I/B) = L(M_I/B \cap M_I)$.

Further one checks that the relative position of the sections $g^{-1}s$ and $b'\sigma(g^{-1}s)$ is $w$, i.e., $(g^{-1}s, b'\sigma(g^{-1}s)) \colon Y \rar L(G/B)^2$ factors through $L\caO^G(w) \har L(G/B)^2$ (here the upper index $G$ only indicates that we refer to $\caO(w)$ for the group $G$). Combined with the above we see that $(g^{-1}s,b'\sigma(g^{-1}s))$ factors through $L\caO^G(w) \times_{L(G/B)^2} L(M_I/B \cap M_I)^2 = L\caO^{M_I}(w)$ (the right hand side makes sense as $w \in W_I \subseteq W$). This means that $(g^{-1}s)' \in X_w^{M_I}(b')(Y)$. Conversely, a section $s' \in X_w^{M_I}(b')(Y)$ determines a section $gs' \colon Y \rar L(G/B)$, which lies in $X_w(b)(Y)$. These two maps are mutually bijective and functorial in $Y$. The same proof applies to $\dot X_{\dot w}^G(b)$.
\end{proof}

Finally, we determine the structure of $X_w^G(b)$ in terms of arc-sheaves $X_w^{M_I}(b_i)$ attached to the Levi subgroup $M_I$ of $G$. Therefore we first show a general result.

\begin{prop}\label{prop:vsheaves_product_over_profinite_base}
Let $H$ be a locally profinite, second-countable group, $H' \subseteq H$ a closed subgroup such that $H \tar Q := H/H'$ has a continuous section. Let $\pi \colon X \rar \underline{Q}$ be a map of $v$-sheaves on $\Perff$, and assume that $\underline{H}$ acts on $X$ such that the action commutes with $\pi$ and the $\underline H$-action on $\underline Q$ by left multiplication. Let $t \colon \Spec \obF \rar \underline Q$ be a geometric point, and let $X_t := X \times_{\underline Q} \Spec \obF$ be the fiber over $t$. Then $X \cong \underline Q \times X_t$ as $v$-sheaves.
\end{prop}
\begin{proof}
As the fibers $X_t$ for varying $t$ are all isomorphic by the $\underline H$-action, we may assume that $t$ corresponds to the coset $1\cdot H' \in Q$. The continuous section to $H \tar Q$ induces a map $s \colon \underline{Q} \rar \underline{H}$ of $v$-sheaves. Let $\iota_t \colon X_t \rar X$ be the natural map. Let ${\rm act}_X$ denote the action of $\underline H$ on $X$. Put $\alpha := {\rm act}_X \circ (s \times \iota_t) \colon \underline Q \times X_t \rar \underline H \times X \rar X_w(b)$. Now we define a map in the other direction: let $R \in \Perff$, let $\beta \in X(R)$. The actions of the element $s\pi\beta \in \underline H(R)$ on $X$ and $\underline Q$ determine a commutative diagram with bijective horizontal arrows:
\[
\xymatrix{
X(R) \ar[r]^{s\pi\beta} \ar[d]^{\pi} & X(R) \ar[d]^{\pi} \\
\underline Q(R) \ar[r]^{s\pi\beta} & \underline Q(R)
}
\]
Let $\gamma \in X(R)$ be the unique element such that $(s\pi\beta)(\gamma) = \beta$. For a $v$-sheaf $Y$, let $f_Y \colon Y \rar \Spec \obF$ denote the unique morphism to the final object. We claim that $\pi\gamma = tf_{\Spec R} \in \underline Q(R)$. As $s\pi\beta \colon X(R) \rar X(R)$ is an isomorphism, it is enough to show that $(s\pi\beta)(\pi\gamma) = (s\pi\beta)(tf_{\Spec R})$. By the commutativity of the diagram above, we have $(s\pi\beta)(\pi\gamma) = \pi\beta$. On the other side, consider the composed map 
\[
\underline Q(R) \stackrel{s}{\longrar} H(R) \longrar \underline{Q}(R), \qquad \delta \mapsto (s\delta)(tf_{\Spec R})
\]
where the second map is the ``orbit map'' for the action of $H(R)$ on the element $tf_{\Spec R} \in \underline Q(R)$. Let ${\rm pr} \colon \underline H \rar \underline Q$ be the natural map. As $t \in \underline Q(\Spec \obF)$ corresponds to the coset $1\cdot H'$, we deduce that the image of $\pi \beta$ under the composed map above is ${\rm pr} (s\pi\beta) = ({\rm pr} \circ s)(\pi\beta) = \pi\beta$, i.~e., with other words we have $(s\pi\beta)(tf_{\Spec R}) = \pi\beta$, proving the claim. 

The association $X(R) \rar X(R)$, $\beta \mapsto \gamma$ defined above is functorial in $R$, so it defines a map $\varepsilon_0 \colon X \rar X$ of $v$-sheaves. The claim shows that $\pi\varepsilon_0 = tf_X$. This gives a map $\varepsilon_1 \colon X \rar X_t$, such that $\iota_t \varepsilon_1 = \varepsilon_0$. Finally we get the map $\varepsilon := (\pi,\varepsilon_1) \colon X \rar \underline Q \times X_t$. One now shows that $\alpha$ and $\varepsilon$ are mutually inverse. 
\end{proof}

\begin{thm}\label{cor:disjoint_dec_of_Xwb_complete} Let $b \in G(\breve k)$, $w \in W$, let $\dot w \in G(\breve k)$ be any lift of $w$. Write $I = \overline{\supp}(w)$. As in the paragraph preceding Corollary \ref{cor:decomposition_of_set_GPbsigma}, write $[b]_G \cap P_I(\breve k) = \coprod_{i=1}^r [b_i]_{P_I}$ with $b_i = g_i^{-1} b \sigma(g_i)$ for finitely many $b_i \in P_I(\breve k)$. We have the equivariant isomorphisms
\[
X_w^G(b) \cong \coprod_{i=1}^r \underline{G_{b_i}(k)/P_{I,b_i}(k)} \times X_w^{M_I}(b_i), \quad \text{ and } \quad \dot X_{\dot w}^G(b) \cong \coprod_{i=1}^r \underline{G_{b_i}(k)/P_{I,b_i}(k)} \times \dot X_{\dot w}^{M_I}(b_i).
\]
\end{thm}
\begin{proof}
This follows from Proposition \ref{prop:vsheaves_product_over_profinite_base}, Proposition \ref{prop:disj_decomposition} and Corollary \ref{cor:decomposition_of_set_GPbsigma}. The only thing to check (to be able to apply Proposition \ref{prop:vsheaves_product_over_profinite_base}) is that for each $i$, $G_{b_i}(k) \tar G_{b_i}(k)/P_{I,b_i}(k)$ has a continuous section. More generally, let $H$ be any reductive group over $k$ and $B$ a Borel subgroup. It follows from the Iwasawa decomposition 
% (see e.~g. \cite[3.3.2]{Tits_79}) 
and the existence of special points in the Bruhat--Tits building, that there is a compact open subgroup $H(k)^0 \subseteq H(k)$ such that the composition $H(k)^0 \har H(k) \tar H(k)/B(k)$ is surjective. In particular, $H(k)^0 \rar H(k)/P(k)$ is surjective for any $k$-rational parabolic subgroup $B \subseteq P \subseteq H$. Now $H(k)^0$ is profinite, and hence a continuous section exists by \cite[I\,\S1 Prop.~1]{Serre_Galois_Cohomology}.
\end{proof}

Concerning representability of $X_w(b)$, $\dot X_{\dot w}(b)$ we deduce the following result, allowing to reduce to the case that $\overline{\supp}(w) = S$.

\begin{cor}\label{cor:representability_reduction_to_full_support}
In the situation of Theorem \ref{cor:disjoint_dec_of_Xwb_complete} suppose that for all $i$, $X_w^{M_I}(b_i)$ resp. $\dot X_{\dot w}^{M_I}(b_i)$ is representable by an ind-scheme or a scheme. Then the same holds for $X_w^G(b)$ resp. $\dot X_{\dot w}^G(b)$.
\end{cor}

\subsection{Frobenius-cyclic shift} 
The same arguments as in \cite[Proof of Thm.~1.6]{DeligneL_76} give the following lemma.

\begin{lm}\label{lm:Frobenius_cyclic_shift}
Assume $w = w_1w_2$, $w' = w_2 \sigma(w_1) \in W$, such that $\ell(w) = \ell(w_1) +\ell(w_2) = \ell(w')$. Then there is an isomorphism $X_w(b) \cong X_{w'}(b)$. If $\dot w$, $\dot w'$, $\dot w_1$, $\dot w_2 \in G(\breve k)$ are lifts of $w,w',w_1,w_2$, satisfying $\dot w = \dot w_1 \dot w_2$, $\dot w' = \dot w_2 \sigma(\dot w_1)$, then $\dot X_{\dot w}(b) \cong \dot X_{\dot w'}(b)$.
\end{lm}
\begin{proof}
Let $R \in \Perff$. Let $g \in X_w(b)(R)$, so that $g \in L(G/B)(R)$ and $g \stackrel{w}{\rar} b\sigma(g)$. By assumption, $\caO(w) \cong \caO(w_1) \times_{G/B} \caO(w_2)$ and $\caO(w') \cong \caO(w_2) \times_{G/B} \caO(\sigma(w_1))$, so there exists a unique $\tau(g) \in L(G/B)(R)$ which fits into the commutative diagram of relative positions,
\begin{center}
\begin{tikzcd}
g \arrow[rr, "w"] \arrow[rd, "w_1"] &[-20pt] &[-20pt] b\sigma(g) \arrow[rd, "\sigma(w_1)"] &[-30pt] \\
&[-20pt] \tau(g) \arrow[rr, "w'"] \arrow[ru, ,"w_2"] &[-20pt] &[-20pt] b\sigma(\tau(g)) 
\end{tikzcd}
\end{center}
This defines a map $X_w(b) \rightarrow X_{w'}(b)$, $g \mapsto \tau(g)$. The same argument gives maps $\tau' \colon X_{w'}(b) \rightarrow X_{\sigma(w)}(b)$ and $\tau^{(\sigma)} \colon X_{\sigma(w)}(b) \rightarrow X_{\sigma(w')}(b)$. One checks that $\tau'(\tau(g)) = b\sigma(g)$, hence these maps fit into the commutative diagram
\begin{center}
\begin{tikzcd}
X_w(b) \arrow[r, "\tau"] \arrow[d, "\sigma"] & X_{w'}(b) \arrow[d, "\sigma"] \arrow[dl, "\tau'"'] \\
X_{\sigma(w)}(b) \arrow[r, "\tau^{(\sigma)}"]  & X_{\sigma(w')}(b)
\end{tikzcd}
\end{center}
As the vertical arrows are isomorphisms (Lemma \ref{lm:mult_by_bsigma}), also all others are. This proves the first claim. To prove the second claim we notice first that we have $U \dot w_1 U \dot w_2 U = U \dot w_1 \dot w_2 U$ (as subvarieties of $G$). 
Hence $\dot \caO(\dot w) \cong \dot \caO(\dot w_1) \times_{G/U} \dot \caO(\dot w_2)$, and similarly for $\dot w', \dot w_2, \sigma(\dot w_1)$. Now the same proof as for $X_w(b)$ also applies to $\dot X_{\dot w}(b)$. \qedhere
\end{proof}

Two elements $w,w' \in W$ are said to be \emph{$\sigma$-conjugate by a cyclic shift} (notation: $w \stackrel{\sigma}{\longleftrightarrow} w'$), if there are three sequences $(w_i)_{i=1}^{n+1}$, $(x_i)_{i=1}^n$, $(y_i)_{i=1}^n$ of elements of $W$ such that $w_1 = w$, $w_{n+1} = w'$ and for each $1\leq i \leq n$: $w_i = x_iy_i$, $w_{i+1} = y_i\sigma(x_i)$ and $\ell(w_i) = \ell(x_i) + \ell(y_i) = \ell(w_{i+1})$. 
% In \cite[\S3]{He_07}, the equivalence relation $$

A $\sigma$-conjugacy class $C$ in $W$ is called \emph{cuspidal} if $C \cap W_J = \varnothing$ for any proper subset $J \subsetneq S$. For a $\sigma$-conjugacy class $C \subseteq W$, let $C_{\rm min}$ denote the set of all elements of minimal length in $C$. One important property of the cyclic shift is the following result.

\begin{thm}[Theorem 3.2.7 of \cite{GeckP_00}, \S6 of \cite{GeckKP_00} and Theorem 7.5 of \cite{He_07}]\label{thm:all_min_length_elements_cyclic_shifts}
Let $C \subseteq W$ be a cuspidal $\sigma$-conjugacy class. Assume that $\overline{\supp}(w) = S$ for an (equivalently any) $w \in C_{\rm min}$. Then for all $w,w' \in C_{\rm min}$ we have $w \stackrel{\sigma}{\longleftrightarrow} w'$. 
\end{thm}

We have the following corollary to Theorems \ref{thm:all_min_length_elements_cyclic_shifts} and \ref{cor:disjoint_dec_of_Xwb_complete}.

\begin{cor}\label{cor:isomorphism_class_depends_only_on_Cmin}
Let $b \in G(\breve k)$ and let $C$ be a conjugacy class in $W$. All $X_w(b)$ for $w$ varying through $C_{\rm min}$ are mutually $G_b(k)$-equivariantly isomorphic.
\end{cor}

% Note that in the classical counterpart of this corollary (and of Lemma \ref{lm:Frobenius_cyclic_shift}), \cite[Proof of Thm.~1.6]{DeligneL_76}, one only obtains universal homeomorphisms of classical Deligne--Lusztig varieties, but as we work with sheaves on $\Perf$, all Frobenii are invertible, and we automatically get isomorphisms.

\begin{proof}
If the class $C$ is cuspidal, this follows from Lemma \ref{lm:Frobenius_cyclic_shift} and Theorem \ref{thm:all_min_length_elements_cyclic_shifts}. If $C$ is arbitrary, we need the following result:

\begin{thm}[Theorem 3.2.12 of \cite{GeckP_00} and Theorem 8.2 of \cite{AdamsHN_20}] \label{thm:non_cuspidal_classes_parametrization}
The map 
\[
C \mapsto \{(\overline{\supp}(w),C \cap W_{\overline{\supp}(w)}) \colon w \in C_{\rm min}\}
\]
defines a bijection between $\sigma$-conjugacy classes of $W$ and the set of pairs $(J,D)$ consisting of a $\sigma$-stable subset $J \subseteq S$ and cuspidal $\sigma$-conjugacy class $D$ in $W_J$, modulo conjugation by (a certain subgroup of) $W$.
\end{thm}

For $j=1,2$, let $w_j \in C_{\rm min}$. Put $J_j = \overline{\supp}(w_j)$. By Theorem \ref{thm:non_cuspidal_classes_parametrization}, $C_{\rm min} \cap W_{J_j}$ is a single cuspidal conjugacy class of $W_{J_j}$, of which $w_j$ is an element of minimal length; moreover, there exists some element $x \in W$ such that $J_2 = J_1^x$ and $(C\cap W_{J_1})^x = C \cap W_{J_2}$. By Theorem \ref{cor:disjoint_dec_of_Xwb_complete}, 
\[X_{w_j}(b) \cong \coprod_{i=1}^r \underline{G_{b_{j,i}}(k)/P_{J_j,b_{j,i}}(k)} \times X_{w_j}^{M_{J_j}}(b_{j,i}) \]
where $[b]_G = \coprod_{i=1}^r [b_{j,i}]_{P_{J_j}}$.
Now, conjugation by $x$ defines an isomorphism $M_{J_1} \cong M_{J_2}$, which maps the cuspidal class $C\cap W_{J_1}$ to $C \cap W_{J_2}$,
%the $\sigma$-conjugacy classes $[b_{j,i}]$ are permuted, 
thus defining an isomorphism between the sheaves $X_{w_1}^{M_{J_1}}(b_{1,i})$ and $X_{w_2}^{M_{J_2}}(b_{2,i})$ (permuting the indices $i$, if necessary), which by the cuspidal case discussed first, only depend on $C_ {\rm min} \cap W_{J_j}$, not on $w_j$. 
\end{proof}

\section{Ind-representability}\label{sec:bonn_roq}

We keep the setup of \S\ref{sec:lDLS} and study representability properties of the sheaves $X_w(b)$, $\dot X_{\dot w}(b)$. We closely follow the strategy of \cite{BonnafeR_08}, where in the setup of classical Deligne--Lusztig theory affineness of certain Deligne--Lusztig varieties is shown. Here the claim ``$X_w$ affine'' is replaced by ``$X_w(b)$ ind-representable''. The idea is that in \cite{BonnafeR_08} the result follows from affineness of certain subvarieties $\caO \subseteq (G/B)^d$, whereas in our setup the affineness of $\caO$ gives that $L\caO$ is ind-representable by Proposition \ref{prop:loop_of_affine_is_indproaffpfp}, which implies our result.

The \emph{Braid monoid} $B^+$ attached to $(W,S)$ is the monoid with the presentation
\[
B^+ = \left\langle (\underline{x})_{x \in W} \colon \forall x,x' \in W, \ell(xx') = \ell(x) + \ell(x) \Rar \underline{xx'} = \underline{x}  \,\underline{x}' \right\rangle.
\]
The automorphism $\sigma$ of $W$ extends to an automorphism of $B^+$. For $I \subseteq S$, we denote by $w_I$ the longest element in $W_I$, and by $\underline{w}_I$ the corresponding element of the $B^+$.

\begin{thm}\label{thm:ind_rep_via_Braid}
Let $I$ be an $\sigma$-stable subset of $S$ and let $w \in W_I$ be such that there exists an integer $d>0$ and $a \in B^+$ with $\underline w \sigma(\underline w) \dots \sigma^{d-1}(\underline w) = \underline{w}_I a$. Then for all $b \in G(\breve k)$ and all $\dot w$ lifting $w$, the arc-sheaves $X_w(b)$, $\dot X_{\dot w}(b)$ are representable by ind-schemes.
\end{thm}

\begin{cor}\label{thm:ind_rep_min_length}
Let $w \in W$ be of minimal length in its $\sigma$-conjugacy class. Then for all $b \in G(\breve k)$, and all lifts $\dot w$ of $w$, the arc-sheaves $X_w(b)$, $\dot X_{\dot w}(b)$ are representable by ind-schemes.
\end{cor}

Prior to the proofs of Theorem \ref{thm:ind_rep_via_Braid} and Corollary \ref{thm:ind_rep_min_length} we note that as $G_0$ is unramified, it has a hyperspecial model $\cG$ over $\caO_k$ whose special fiber $\cG \otimes_{\caO_k} \bF_q$ is a reductive group over $\bF_q$, such that the Weyl group of its base change to $\obF$ is equal to the Weyl group of $G$. In particular, all combinatorial arguments from \cite{BonnafeR_08} carry over to our situation. 

\begin{rem}\label{rem:difference_to_Bonnafe_Rouquier}
A difference to \cite{BonnafeR_08} is that we need to give a separate proof for $\dot X_{\dot w}(b)$, whereas in the classical Deligne--Lusztig theory the equivalence ``$X_w$ affine $\LRar$ $\dot X_{\dot w}$ affine'' is immediate.
\end{rem}

\subsection{Proof of Theorem \ref{thm:ind_rep_via_Braid}}
The proof goes along the lines of the proof of \cite[Thm.~B]{BonnafeR_08}. For a sequence $(x_1,\dots,x_r)$ of elements of $W$ define
\begin{align*}
\caO(x_1,\dots,x_r) := \caO(x_1) \times_{G/B} \dots \times_{G/B} \caO(x_r) 
\end{align*}
If $y_1, \dots, y_s \in W$ such that $\underline{x}_1 \underline{x}_2 \dots \underline{x}_r = \underline{y}_1 \underline{y}_2 \dots \underline{y}_s$ in $B^+$, then $\caO(x_1, \dots, x_r) \cong \caO(y_1, \dots, y_s)$ (even canonically, see \cite[Application 2]{Deligne_97}). For lifts $\dot x_1, \dots, \dot x_r \in N_G(T)(\breve k)$ of $x_1, \dots, x_r$ put
\begin{align*}
\dot\caO(\dot x_1,\dots,\dot x_r) &:= \dot \caO(\dot x_1) \times_{G/U} \dots \times_{G/U} \dot \caO(\dot x_r) \\ 
&= \{(g_0 U, g_1 U, \dots, g_r U) \in (G/U)^{r+1} \colon \forall 1\leq i\leq r \colon g_i^{-1}g_{i+1} \in U\dot x_i U\}.
\end{align*}
Then $T$ acts on $\dot\caO(\dot x_1,\dots,\dot x_r)$ by
\[
T \ni t \colon (g_0U, g_1U, \dots, g_r U) \mapsto (g_0tU, g_1 {\rm Ad}(x_1)^{-1}(t)U, \dots, g_1 {\rm Ad}(x_r\dots x_1)^{-1}(t)U)
\]
As in \cite[Proof of Prop.~3]{BonnafeR_08}, the natural map 
\begin{align*}
\dot\caO(\dot{x}_1, \dots,\dot{x}_r) &\rar \caO(x_1,\dots,x_r) \\
(g_0U, \dots, g_r U) &\mapsto \caO(g_0B,\dots, g_rB)
\end{align*}
identifies $\caO(x_1,\dots,x_r)$ with the quotient of $\dot\caO(\dot{x}_1, \dots,\dot{x}_r)$ by the action of $T$. We then have: 
\begin{equation}\label{eq:affine_equiv_affine}
\dot\caO(\dot{x}_1, \dots,\dot{x}_r) \text{ affine } \LRar \caO(x_1,\dots,x_r) \text{ affine}
\end{equation}
($\Leftarrow$ as a (geometric) quotient maps are affine; $\Rar$ by \cite[Cor.~8.21]{Borel_91}). Let $w_0 \in W$ denote the longest element. 
\begin{prop}\label{prop:BonRouq_affine_Ow} (see \cite[Prop.~3]{BonnafeR_08})
If there exists $v \in B^+$, such that $\underline{x}_1\cdots\underline{x}_r = \underline{w}_0 v$, then $\caO(x_1,\dots,x_r)$ is affine. If $\dot{x}_i$ are lifts of the $x_i$ ($1\leq i \leq r$), then also $\dot\caO(\dot x_1,\dots,\dot x_r)$ is affine.
\end{prop}
\begin{proof} The proof from \cite[Prop.~3]{BonnafeR_08} that $\caO(x_1,\dots,x_r)$ is affine applies \emph{mutatis mutandis} (in fact, the setup there is over a finite field $\bF_q$ instead of the local field $k$, but this does not affect anything). The claim for $\dot \caO(\dot x_1, \dots,\dot x_r)$ follows from the equivalence \eqref{eq:affine_equiv_affine}.
\end{proof}

Now we can prove Theorem \ref{thm:ind_rep_via_Braid}. By Corollary \ref{cor:representability_reduction_to_full_support}, we may assume that $I = S$. Now, the point (same as in \cite{BonnafeR_08}) is that $X_w(b)$ possesses also a slightly different presentation, which is more convenient for our purposes. Consider the morphism $\Delta_d \colon L(G/B) \rar L(G/B)^d$ given by $s \mapsto (s,b\sigma(s),(b\sigma)^2(s), \dots, (b\sigma)^{d-1}(s))$. Then the diagram
\[
\xymatrix{
X_w(b) \ar[r] \ar[d] & L\caO(w,\sigma(w),\dots, \sigma^{d-1}(w)) \ar[d]\\
L(G/B) \ar[r]^-{\Delta_d} & L(G/B)^d
}
\]
where the upper horizontal map is given by $s \mapsto (s,b\sigma(s), \dots, (b\sigma)^{d-1}(s))$ is Cartesian. Now $\caO(w,\sigma(w),\dots, \sigma^{d-1}(w))$ is of finite type over $k$, and by Proposition \ref{prop:BonRouq_affine_Ow} also affine, so \\ $L\caO(w,\sigma(w),\dots, \sigma^{d-1}(w))$ is an ind-scheme by Proposition \ref{prop:loop_of_affine_is_indproaffpfp}. Now $\Delta_d$ is by Lemma \ref{lm:representability_graph} representable by closed immersions. Hence $X_w(b)$ is closed sub-ind-scheme of \\ $L\caO(w,\sigma(w),\dots, \sigma^{d-1}(w))$. The same proof (with $G/B$ replaced by $G/U$ and \\ $\caO(w,\sigma(w),\dots,\sigma^{d-1}(w))$ by $\dot\caO(\dot w,\sigma(\dot w),\dots,\sigma^{d-1}(\dot w))$ applies to $\dot X_{\dot w}(b)$.

\subsection{Proof of Corollary \ref{thm:ind_rep_min_length}}

Corollary \ref{thm:ind_rep_min_length} follows now from Theorem \ref{thm:ind_rep_via_Braid} in a similar way as \cite[Thm.~A]{BonnafeR_08} follows from \cite[Thm.~B]{BonnafeR_08}. Namely, let $C$ be a $\sigma$-conjugacy class in $W$, and let $C_{\rm min}$ denote the set of elements of minimal length in $C$. Let $d$ be the smallest positive integer $k$ such that $w\sigma(w) \sigma^2(w) \dots \sigma(w)^{k-1} = 1$ and $\sigma^k$ acts trivially on $W$. 

First, we prove the theorem for good elements in $C_{\rm min}$ and their lifts. An element $w \in C_{\rm min}$ is called \emph{good}, if there exists a sequence of subsets $I_r \subseteq I_{r-1} \subseteq \dots \subseteq I_1 \subseteq S$, such that $\underline{w} \sigma(\underline{w}) \dots \sigma^{d-1}(\underline w) = \underline{w}_{I_1}^2 \underline{w}_{I_2}^2 \dots \underline{w}_{I_r}^2$. Now Theorem \ref{thm:ind_rep_via_Braid} applies to good $w$ and shows the ind-representability of $X_w(b)$, once it is proven that $I_1$ is $\sigma$-stable. But this is the case (see \cite[Prop.~4]{BonnafeR_08}). Also, if $\dot w\in G(\breve k)$ is any lift of a good element $w \in C_{\rm min}$, then Theorem \ref{thm:ind_rep_via_Braid} also shows ind-representability of $\dot X_{\dot w}(b)$.

Now we show Corollary \ref{thm:ind_rep_min_length} for all $w\in C_{\rm min}$. By Corollary \ref{cor:representability_reduction_to_full_support} we may assume that $\overline{\supp}(w) = S$. By the above paragraph, Theorem \ref{thm:ind_rep_min_length} holds for all good $w \in C_{\rm min}$. Thus by Theorem \ref{thm:all_min_length_elements_cyclic_shifts} and Lemma \ref{lm:Frobenius_cyclic_shift} it remains to show that there always exists a good element in $C_{\rm min}$. But this is a result of Geck--Michel, Geck--Kim--Pfeiffer and He (see \cite[Thm.~6]{BonnafeR_08} and the references there). This finishes the proof of Corollary \ref{thm:ind_rep_min_length} for $X_w(b)$.

Finally, let $w \in C_{\rm min}$ and let $\dot w$ be any lift of $w$. It remains to show that $\dot X_{\dot w}(b)$ is ind-representable. Again, by Corollary \ref{cor:representability_reduction_to_full_support} we may assume that $\overline{\supp}(w) = S$. As in the preceding paragraph, the result follows from the good case, the existence of a good $w'' \in C_{\rm min}$ with $w \stackrel{\sigma}{\longleftrightarrow} w''$, Lemma \ref{lm:Frobenius_cyclic_shift} and the following (obvious) observation: For any $w,w',w_1,w_2 \in W$ as in Lemma \ref{lm:Frobenius_cyclic_shift} and any lift $\dot w$ of $w$, there are lifts $\dot{w}_j$ of $w_j$ ($j = 1,2$), such that $\dot w = \dot w_1 \dot w_2$.

\section{A criterion for non-representability by schemes}\label{sec:nonrep}

We keep the setup of \S\ref{sec:lDLS}. It turns out that the sheaves $X_w(b)$ are rarely representable by schemes. More precisely, we have the following result. Recall that $C_{\rm min}$ denote the set of elements of minimal length in a $\sigma$-conjugacy class $C \subseteq W$.

\begin{thm}\label{thm:non_rep}
Let $b \in G(\breve k)$, and let $C$ be a $\sigma$-conjugacy class in $W$, such that $X_w(b) \neq \varnothing$ for a (equivalently, any) $w \in C_{\rm min}$. Then for any $w \in C \sm C_{\rm min}$, $X_w(b)$ is not representable by a scheme.
\end{thm}

The proof is given in \S\ref{sec:proof_thm_nonemptyness}. In particular, in all cases ``between Theorem \ref{thm:ind_rep_via_Braid} and Corollary \ref{thm:ind_rep_min_length}'' we obtain ind-schemes, which are not schemes:

\begin{cor}\label{cor:ind_rep_non_rep}
Let $b, C$ be as in Theorem \ref{thm:non_rep}. Let $w \in C \sm C_{\rm min}$ be such that it satisfies the assumptions of Theorem \ref{thm:ind_rep_via_Braid}. Then $X_w(b)$ is an ind-scheme, which is not a scheme. 
\end{cor}

An example of an element $w$ satisfying the assumptions of the corollary is the longest element of $W$, unless the Dynkin diagram of $G$ is disjoint union of diagrams of type $A_1$.

\subsection{A closed subfunctor of $X_w(b)$}

We adapt an idea from \cite[Proof of Thm.~1.6]{DeligneL_76} to our situation. Let $b\in G(\breve k)$ and let $w, w' \in W$, $s \in S$ such that $w = sw'\sigma(s)$ and $\ell(w) = \ell(w') + 2$. By the properties of the Bruhat decomposition we have an isomorphism, 
\begin{equation}\label{eq:bruhat_iso_additive_length}
\caO(s) \times_{G/B} \caO(w') \times_{G/B} \caO(\sigma(s)) \stackrel{\sim}{\rar} \caO(w),
\end{equation}
given by sending $g_1 \stackrel{s}{\rar} g_2 \stackrel{w'}{\rar} g_3 \stackrel{\sigma(s)}{\rar} g_4$ to $g_1 \stackrel{w}{\rar} g_4$. It remains an isomorphism after applying $L(\cdot)$. Let now $R \in \Perf_{\obF}$ and $g \in X_w(b)(R)$, so $g \stackrel{w}{\rar} b\sigma(g)$. Applying $L(\cdot)$ to \eqref{eq:bruhat_iso_additive_length}, we see that there exist unique $\gamma g, \delta g \in L(G/B)(R)$, such that $g \stackrel{s}{\rar} \gamma g \stackrel{w'}{\rar} \delta g \stackrel{\sigma(s)}{\rar} b\sigma(g)$. Let $X_1$ be the subfunctor of $X_w(b)$ defined by
\begin{equation}\label{eq:subfunctor_X1}
X_1(R) = \{g \in X_w(b)(R) \colon \delta g = b\sigma(\gamma g) \}.
\end{equation}
In \cite{DeligneL_76} it is immediate that the so defined $X_1$ is a closed subscheme of a classical Deligne--Lusztig variety. Similarly, we have the following proposition. 
\begin{prop}\label{prop:closed_subfunctor_X1}
In the above situation, $X_1 \rar X_w(b)$ is representable by closed immersions, and moreover, $X_1$ fits into the Cartesian diagram
\[
\xymatrix{
X_1 \ar[r]^\gamma \ar[d] & X_{w'}(b) \ar[d] \\ L\caO(s) \ar[r]^{{\rm pr}_2} & L(G/B),
}
\]
where the lower map is $(g \stackrel{s}{\rar} h) \mapsto h$ and the left map is $g \mapsto (g \stackrel{s}{\rar} \gamma g)$.
\end{prop}
\begin{proof}
We have a diagram with Cartesian squares
\[
\xymatrix@C+2pc{
X_1 \ar[r]^\gamma \ar[d] & X_{w'}(b) \ar[r] \ar[d] & L(G/B) \ar[d]^{(\id,b\sigma)} \\
X_w(b) \ar[r]^{g\mapsto \gamma g,\delta g} & L\caO(w') \ar[r] & L(G/B)\times L(G/B),
} 
\]
where the right square is as in the definition of $X_{w'}(b)$. Then also the outer square is Cartesian, and the first claim follows, as the right vertical map is representable by closed immersions by Lemma \ref{lm:representability_graph}. The second claim formally follows from \eqref{eq:bruhat_iso_additive_length} after applying $L(\cdot)$.
\end{proof}

\begin{cor}\label{cor:non_emptyness}
Let $b \in G(\breve k)$ and let $C$ be a $\sigma$-conjugacy class in $W$. If for some $w \in C_{\rm min}$, $X_w(b)$ is non-empty, then $X_{w'}(b)$ is non-empty for all $w' \in C$.
\end{cor}
\begin{proof}
By Corollary \ref{cor:isomorphism_class_depends_only_on_Cmin}, $X_w(b)$ is non-empty for all $w \in C_{\rm min}$. Let $w' \in C$. By \cite[Thm.~3.2.9]{GeckP_00} resp. \cite[Thm.~7.6]{He_07}, there exists a sequence of elements $w' = w_0$, $w_1,\dots,w_n$ of $W$ such that $w_n \in C_{\rm min}$ and for each $0\leq i\leq n-1$ there is some $s_i \in S$ such that $w_{i+1} = s_i w_i \sigma(s_i)$, $\ell(w_{i+1}) \leq \ell(w_i)$. It suffices to show that $X_{w_{i+1}}(b)\neq \varnothing$ implies $X_{w_i}(b) \neq \varnothing$. There are two cases: either $\ell(w_i) = \ell(w_{i+1}) + 2$ or $\ell(w_i) = \ell(w_{i+1})$. In the first case we are in the situation of Proposition \ref{prop:closed_subfunctor_X1}, so that $X_{w_i}(b)$ has a closed subfunctor, which lies over the non-empty $X_{w_{i+1}}(b)$, and is easily seen to be non-empty itself, e.g., by looking at its fiber at a geometric point (as in \S\ref{sec:proof_thm_nonemptyness}). In the second case, $\ell(w_i) = \ell(w_{i+1})$, and $w_i$, $w_{i+1}$ are related by a cyclic shift (using \cite[Lm.~1.6.4]{DeligneL_76}), so that $X_{w_i}(b) \cong X_{w_{i+1}}(b)$.
\end{proof}

\subsection{Proof of Theorem \ref{thm:non_rep}}\label{sec:proof_thm_nonemptyness}

As $w$ is not minimal in its $\sigma$-conjugacy class, there exists some $w' \in W$ and a simple reflection $s \in S$ with $w = s w' \sigma(s)$ and $\ell(w) = \ell(w') + 2$. By Corollary \ref{cor:non_emptyness}, $X_{w'}(b) \neq \varnothing$. Let $\ff \in \Perf_{\obF}$ be a field such that $X_{w'}(b)(\ff) 
\neq \varnothing$ and let $x \in X_{w'}(b)(\ff)$. The claim of the theorem may be checked after restriction of $X_w(b)$ to $\Perf_{\ff}$, so we may replace $\obF$ by $\ff$. Using that $L(G/B)$ is separated (Lemma \ref{lm:representability_graph}) and that $X_{w'}(b) \rar L(G/B)$ is a monomorphism, one shows by a standard argument that $x \colon \Spec \ff \rar X_{w'}(b)$ is representable by closed immersions.

The map $x$ induces a map $\Spec\bW(\ff)[1/\varpi] \rar G/B$ and we have 
\[
\caO(s) \times_{G/B} \Spec\bW(\ff)[1/\varpi] \cong \bA^1_{\bW(\ff)[1/\varpi]}.
\]
As $\Spec\ff = L(\Spec \bW(\ff)[1/\varpi])$, and as $L(\cdot)$ commutes with fiber products, we deduce that 
\begin{equation}\label{eq:fiber_product_Os}
L\caO(s) \times_{L(G/B)} \Spec\ff \cong L\bA^1_{\ff}. 
\end{equation}
Thus we have a commutative diagram, 

\[
\xymatrix{
L\bA^1_{\ff} \ar[r] \ar[d] & X_1 \ar[r] \ar[d] & L\caO(s) \ar[d] \\ 
\Spec \ff \ar[r]^x & X_{w'}(b) \ar[r] & L(G/B)  
}
\]
in which the outer square is Cartesian by \eqref{eq:fiber_product_Os}, and the right square is Cartesian by Proposition \ref{prop:closed_subfunctor_X1}. It follows that also the left square is Cartesian. By the above, the left lower map is representable by closed immersions, so the upper left map also is. Thus the first map in
\[
L\bA^1_\ff \rar X_1 \rar X_w(b)
\]
is representable by closed immersions. But by Proposition \ref{prop:closed_subfunctor_X1}, the second also is, so the composition is too. If $X_w(b)$ would be representable by a scheme, $L\bA^1$ would then be representable by a closed subscheme, which is false. This proves Theorem \ref{thm:non_rep}.

\section{The morphisms $\dot X_{\dot w}(b) \rar X_w(b)$}\label{sec:Torsors}

Here we study in detail the maps $\dot X_{\dot w}(b) \rar X_w(b)$. The first goal is to define certain discrete sheaf $LF_w/\ker \kappa_w$, a natural map $\alpha_{w,b} \colon X_w(b) \rar LF_w/\ker \kappa_w$, and a map $\dot w \mapsto \bar{\dot w}$ from the set of all lifts $F_w(\breve k)$ of $w$ to $(LF_w/\ker \kappa_w)(\obF)$, such that $\dot X_{\dot w}(b) \rar X_w(b)$ factors over the clopen subsheaf $X_w(b)_{\bar{\dot w}} := \alpha_{w,b}^{-1}(\underline{\{\bar{\dot w}\}})$ (Proposition \ref{prop:natural_classifying_map}). The second goal will be to prove that $\dot X_{\dot w}(b) \rar X_w(b)_{\bar{\dot w}}$ is a $\underline{T_w(k)}$-torsor for the $v$-topology (Proposition \ref{prop:arc_surj_of_covers}), and actually, almost a pro-\'etale map. Moreover, all $\dot X_{\dot w}(b)$ for $\dot w$ lying over the same $\bar{\dot w}$ are isomorphic.

\subsection{Units in Witt vectors}\label{sec:units_Witt}

Let $R \in \Perf$. Attached to any $x \in \bW(R)[1/\varpi]$ we have the function
\[
{\rm ord}_\varpi(x) \colon \Spec R \rar \bZ \cup \{\infty\}, \quad s \mapsto \text{$\varpi$-adic valuation of $x(s) \in \bW(k(s))[1/\varpi]$,}
\]
where $k(s)$ is the residue field of $s$. For $R \in \Perf_{\bF_q}$, let $\phi \colon \bW(R)[1/\varpi] \rar \bW(R)[1/\varpi]$ be the Frobenius automorphism (cf. e.g. \cite[\S1.2.1]{FarguesFontaine_book}). It is given by the formula $\sum_{i\geq N}^\infty [x_i]\varpi^i \mapsto \sum_{i\geq N}^\infty [x_i^q]\varpi^i$.

\begin{lm}\label{lm:properties_of_ord}
Let $R \in \Perf$, $x,y \in \bW(R)[1/\varpi]$. Then $\ord_\varpi(xy) = \ord_\varpi(x) + \ord_\varpi(y)$. Moreover, if $R \in \Perf_{\bF_q}$, then $\ord_\varpi(x) = \ord_\varpi(\phi(x))$.
\end{lm}

\begin{proof}
Immediate computation. 
\end{proof}

\begin{lm}\label{lm:units_Witt} Let $R \in \Perf$ and $x \in \bW(R)[1/\varpi]$. Then $\ord_\varpi(x)$ is upper semi-continuous. If $x$ is a unit, then it is continuous. Moreover, if $\ord_\varpi(x)<\infty$ and locally constant, then $x$ is a unit.
\end{lm}
\begin{proof}
The first two claims follow from \cite[Rem.~5.0.2]{KedlayaL_15} in $\charac k = 0$ case; in the other case the same argument works.\footnote{Also a direct computation using that $R$ is perfect (and, in particular, reduced) works.} Suppose $\ord_\varpi(x) < \infty$ and locally constant.  Then it takes only finitely many values. Using that $\bW(\cdot)[1/\varpi]$ commutes with finite products, we may suppose that $\ord_\varpi(x)$ is constant. Then the leading coefficient of $x$ is in no maximal ideal of $R$, hence in $R^\times$, and hence $x$ is a unit. \qedhere
% Let us give a direct argument. We may assume $N = 0$. Let $y = \sum_{i=0}^\infty [y_i]\varpi^i \in \bW(R)$ be such that $xy = \varpi^M$ for some $M \geq 0$. ...
\end{proof}

It follows from this lemma that there is a natural map $L\bG_m \rar \underline{\bZ}_{\obF}$ of arc-sheaves on $\Perf_{\obF}$, where $\bZ$ is equipped with the discrete topology.

\subsection{Loop groups of unramified tori}\label{sec:loop_groups_of_unram_tori}
Let $S_0$ be an unramified $k$-torus and $S \cong \bG_m^n$ its base change to $\breve k$. Let $\sigma = \sigma_S \colon LS \rar LS$ be the corresponding geometric Frobenius, and denote also by $\sigma$ the action of the Frobenius automorphism of $\breve k/k$ on $X_\ast(S)$. According to \S\ref{sec:units_Witt}, we have a map $LS \rar \underline{X_\ast(S)}$ of sheaves of abelian groups on $\Perf_{\obF}$. Moreover, it is equivariant with respect to the action of $\sigma$ of both sides. Passing to $\sigma$-coinvariants on the right, we deduce a map
\[ 
\kappa = \kappa_S \colon LS \rar \underline{X_\ast(S)_{\langle \sigma \rangle}}.
\]
and its kernel is a subgroup $\ker \kappa$ of $LS$. The map $\kappa$ is the sheaf version of the Kottwitz map \cite[2.4,2.9]{Kottwitz_85}.

\begin{prop}\label{prop:surj_Lang_LT_onto_ker_kappa} We have a short exact sequence of sheaves of abelian groups for the pro-\'etale topology on $\Perf_{\obF}$,
\[
0 \rar \underline{S(k)} \rar LS \rar \ker \kappa \rar 0,
\]
where the right map is induced by the Lang map $t \mapsto t^{-1}\sigma(t) \colon LS \rar \ker\kappa \subseteq LS$.
\end{prop}
\begin{proof}
Let ${\rm Lang} \colon LS \rar LS$ denote the Lang map. By Proposition \ref{lm:inclusion_into_LGbsigma}, $\ker {\rm Lang} = (LS)^\sigma = \underline{S(k)}$. Next, $\kappa \circ {\rm Lang} = 0_{\underline{X_\ast(S)_{\langle \sigma \rangle}}} \circ \kappa = 0$, i.e., ${\rm Lang}$ factors through $\ker\kappa \subseteq LS$, and it remains to show that it is surjective onto $\ker\kappa$ for the \'etale topology. Let $\caS \cong \bG_{m,\caO_{\breve k}}^n$ be the canonical model of $S$ over $\caO_{\breve k}$. Then $L^+\caS$ is a subsheaf of $LS$, stable under $\sigma$. The restriction ${\rm Lang} \colon L^+\caS \rar L^+\caS$ is surjective for the pro-\'etale topology (by Lang's theorem for all of its truncations $L^+_r\caS$ ($r\geq 1$), which are connected perfectly finitely presented group schemes over $\obF$).

Now $\kappa$ factors as $LS \rar \underline{X_\ast(S)} \rar \underline{X_\ast(S)_{\langle \sigma \rangle}}$, the kernel of the first of these two maps is $L^+\caS$, and the proposition follows from the commutative diagram with exact rows
\[
\xymatrix{
0 \ar[r] & L^+\caS \ar[r] \ar[d] & LS \ar[r] \ar[d]^{{\rm Lang}} & \underline{X_\ast(S)} \ar[r] \ar[d] & 0 \\
0 \ar[r] & L^+\caS \ar[r] & \ker \kappa \ar[r] & \underline{(\sigma-1)X_\ast(S)} \ar[r] & 0
}
\]
by applying the snake lemma, and using that the outer vertical arrows are surjective.
\end{proof}

\subsection{A discrete invariant} For the rest of \S\ref{sec:Torsors} we work in the setup of \S\ref{sec:lDLS}. Fix $b \in G(\breve k)$ and $w \in W$.  As in \S\ref{sec:action_torus} we have the unramified torus $T_w$. Applying the results of \S\ref{sec:loop_groups_of_unram_tori} to $T_w$, and writing $\sigma_w$ for $\sigma_{T_w}$ and $\kappa_w$ for $\kappa_{T_w}$, we have the map $\kappa_w \colon LT_w \rar \underline{X_\ast(T_w)} \rar \underline{X_\ast(T_w)_{\langle \sigma_w \rangle}}$.
Consider also the $\breve k$-scheme $F_w$, defined as the connected component of $N_G(T)$ corresponding to $w$. Clearly, $F_w \rar \Spec \breve k$ is a trivial $T_w$-torsor. Applying the loop functor, we deduce the trivial $LT_w$-torsor $LF_w \rar \Spec \obF$. Quotienting out the $\ker\kappa_w$-action, we obtain the trivial $\underline{X_\ast(T_w)_{\langle \sigma_w \rangle}}$-torsor 
\begin{equation}\label{eq:LFw_mod_ker}
LF_w/\ker\kappa_w \rar \Spec \obF.
\end{equation}

\begin{prop}\label{prop:natural_classifying_map}
There is a natural map of arc-sheaves $\alpha_{w,b} \colon X_w(b) \rar LF_w/\ker\kappa_w$, satisfying the following properties. Let $\dot w \in LF_w(\obF) = F_w(\breve k)$ and let $\bar{\dot w} \in (LF_w/\ker\kappa_w)(\obF)$ be 
its image, then:
\begin{itemize}
\item[(i)] $\dot X_{\dot w}(b) \rar X_w(b)$ factors through the clopen subset $X_w(b)_{\bar{\dot w}} := \alpha_{w,b}^{-1}(\underline{\{ \bar{\dot w} \}})$. \item[(ii)] $\dot X_{\dot w}(b)\neq \varnothing$ if and only if $\bar{\dot w} \in \im(\alpha_{w,b})$.
\end{itemize}
In particular, $X_w(b) = \coprod_{\bar{\dot w}} X_w(b)_{\bar{\dot w}}$ is a disjoint decomposition into clopen subsets.
\end{prop}

We prove Proposition \ref{prop:natural_classifying_map} in \S\ref{sec:proof_of_nat_class_map}. The map $\alpha_{w,b}$ can be thought of as the discrete part of a ``classifying map'' for the maps $\dot X_{\dot w}(b) \rar X_w(b)$. 
The map $\alpha_{w,b}$ can be non-trivial, even for $G = \GL_2$, see \S\ref{sec:expl_GL2}. Finally, if $\dot w_1,\dot w_2$  lie in the same fiber of $LF_w(\obF) \rar (LF_w/\ker\kappa_w)(\obF)$, then $\dot X_{\dot w_1}(b)$, $\dot X_{\dot w_2}(b)$ are isomorphic, as the next lemma shows.

\begin{lm}\label{lm:Xdot_changes_nicely}
If $\dot w_1, \dot w_2 \in LF_w(\obF)$ have the same image in $(LF_w/\ker\kappa_w)(\obF)$, then there exists some $t \in T(\breve k)$ with $t^{-1}\dot w_1 \sigma(t) = \dot w_2$, and $g \mapsto gt$ induces an isomorphism $\dot X_{\dot w_1}(b) \stackrel{\sim}{\rar} \dot X_{\dot w_2}(b)$. 
\end{lm}
\begin{proof}
By \cite[2.4, 2.9]{Kottwitz_85}, $X_\ast(T_w)_{\langle \sigma_w \rangle} \cong B(T_w)$, the set of $\sigma_w$-conjugacy classes of $T_w$, so that for each $\tau \in \ker(\kappa_w(\obF))$, there exists $t \in T_w(\breve k)$ such that $\tau = t^{-1}\sigma_w(t) = t^{-1}\dot w \sigma(t) \dot w^{-1}$. This implies the first claim. The second claim is proven similarly as Lemma \ref{lm:action_on_Xwb}.
\end{proof}

\begin{rem}\label{rem:alpha_wb_on_geometric_points}
Anticipating the proof, we explicate $\alpha_{w,b}$ on geometric points. Let $\ff \in \Perff$ be an algebraically closed field, put $L = \bW(\ff)[1/\varpi]$. Let $g \in X_w(b)(\ff)$, that is $g \in L(G/B)(\ff)$, such that $(\widetilde g, \widetilde{b\sigma(g)}) \colon \Spec L \rar (G/B)^2$ factors through $\caO(w) \rar (G/B)^2$. The natural map $(G/U)(L) \rar (G/B)(L)$ is surjective, so $g$ lifts to some $\dot g \in L(G/U)(\ff)$. As $L$ is a field, by the geometric Bruhat decomposition (``at the level of $U$''), the section $(\widetilde{\dot g}, \widetilde{b\sigma(\dot g)}) \in (G/U)^2(L)$ lies in $\caO(\dot w)$ for some $\dot w \in LF_w (\ff) = F_w(L)$, that is $\widetilde{\dot g}^{-1}\widetilde{b\sigma(\dot g)} \in U(L)\dot w U(L)$. Replacing $\dot g$ by another lift amounts to replacing $\dot w$ by $t^{-1}{\rm Ad}(w)(\sigma(t)) \dot w = t^{-1}\sigma_w(t) \dot w\in LF_w(\ff)$ for some $t \in T(L)$. But $t^{-1}\sigma_w(t) \in (\ker\kappa_w)(\ff)$, so that the class of $\dot w$ in the discrete set $LF_w(\ff)/\ker\kappa_w (\ff) = (LF_w/\ker\kappa_w)(\ff)$ does not depend on the choice of the lift $\dot g$. 
\end{rem}

\subsection{Schubert cells} We have the locally closed subvariety $BwB \subseteq G$, which gives the subfunctor $L(BwB) \subseteq LG$. Fix a lift $\dot w \in F_w(\breve k)$ of $w$. Let $U,U^-$ be the unipotent radicals of $B$ resp. the opposite Borel subgroup. By the Bruhat decomposition \cite[14.12 Thm.~]{Borel_91}, 
\begin{equation}\label{eq:Bruhat_dec_param}
\beta \colon (U \cap wU^-w^{-1}) \times T \times U \rar BwB, \quad u_1,t,u_2 \mapsto u_1 \dot w t u_2
\end{equation}
is an isomorphism of $\breve k$-varieties. (In particular, $BwB$ is affine, and hence $L(BwB)$ is ind-representable; we will not need this, and moreover, the above map $L(BwB) \rar LG$ is in general not a locally closed immersion).

\begin{lm}\label{lm:BwB_to_Fw_map}
There is a natural (independent of $\dot w$) map $BwB \rar F_w$, given in terms of \eqref{eq:Bruhat_dec_param} by $\beta(u_1,t,u_2) \mapsto \dot w t$. Further, it maps $u_1t_1 \dot w t_2 u_2 \in BwB$ with $u_i \in U$, $t_i \in T$ to $t_1\dot w t_2$.
\end{lm}
\begin{proof}
This is an elementary computation. 
\end{proof}

Applying the loop functor to the map from Lemma \ref{lm:BwB_to_Fw_map} and composing with the natural projection we deduce a map
\begin{equation}\label{eq:map_from_LBwB_to_discrete_torsor}
L(BwB) \rar LF_w \rar LF_w / \ker\kappa_w.
\end{equation}

\subsection{Proof of Proposition \ref{prop:natural_classifying_map}}\label{sec:proof_of_nat_class_map}

We will define the map $\alpha_{w,b}$ arc-locally. This is sufficient as source and target are arc-sheaves, and the definition will be natural, hence compatible with pull-backs. Let $p \colon LG \rar L(G/B)$ denote the natural projection. The map $(LG)^2 \rar LG$, $(x,y)\mapsto x^{-1}y$ induces a map $(p\times p)^{-1}(L\caO(w)) \rar L(BwB)$, as follows by applying $L(\cdot)$ to the corresponding maps over $\breve k$. 

Let $R \in \Perf_{\obF}$ be such that $p(R)$ is surjective. Fix some section $s \colon L(G/B)(R) \rar LG(R)$ to $p(R)$. In what follows we always look at $R$-points, but for brevity we write $F$ instead of $F(R)$ for any sheaf $F$. We have the commutative diagram 

\begin{equation}\label{eq:diag_XwbR_LBwB}
\begin{tikzcd}[column sep=3em,row sep=3em]
  X_w(b) \arrow[r] \arrow[rrrd] & L(G/B) \arrow[r,"s"] & LG \arrow[r, "g \mapsto (g{,}b\sigma(g))"] & LG^2 \ar[r, "(x{,}y) \mapsto x^{-1}y"] & LG 
  \\
  &&& (p\times p)^{-1}(L\caO(w)) \arrow[u,hook] \arrow[r] & L(BwB) \arrow[u,hook] \arrow[r, "(11.3)"] & LF_w/\ker\kappa_w
\end{tikzcd}
\end{equation}
Indeed, the only thing we have to justify is that the composed horizontal map $X_w(b) \rar LG^2$ factors through the left vertical arrow. But we have $(p\times p)^{-1}(L\caO(w)) = LG^2 \times_{L(G/B)^2} L\caO(w)$, the composition $X_w(b) \rar LG^2 \stackrel{p\times p}{\rar} L(G/B)^2$ is just the map $g \mapsto (g,b\sigma(g))$ (as $p$ commutes with $b\sigma$), and the latter map factors through $L\caO(w) \rar L(G/B)^2$.

Diagram \eqref{eq:diag_XwbR_LBwB} gives a map $\alpha_{w,b} \colon X_w(b) \rar LF_w/\ker\kappa_w$, and one checks (using the second claim of Lemma \ref{lm:BwB_to_Fw_map}) that it is independent of the choice of the section $s$. In particular, if $R \rar R'$ is a map in $\Perf_{\obF}$, such that $p(R)$, $p(R')$ are surjective, then the obvious diagram, into which $\alpha_{w,b}(R)$, $\alpha_{w,b}(R')$ fit, is commutative. Finally, by Corollary \ref{thm:arc_exactness_and_splitting} and Lemma \ref{lm:standard_vcover}, any $R_0 \in \Perf_{\obF}$ admits an arc-cover $R_0 \rar R$, such that $p(R) \colon LG(R) \rar L(G/B)(R)$ is surjective, and the construction of the map $\alpha_{w,b}$ is complete. It is clear from this construction that $\alpha_{w,b}$ satisfies property (i) claimed in the proposition. If $\dot X_{\dot w}(b)\neq \varnothing$, then $\bar{\dot w} \in \im(\alpha_{w,b})$ by (i). This is one direction of (ii). The other follows from Proposition \ref{prop:arc_surj_of_covers} below (we do not use (ii) in the proof of Proposition \ref{prop:arc_surj_of_covers}, so there is no circular reasoning).

\subsection{The $\underline{T_w(k)}$-torsor $\dot X_{\dot w}(b) \rar X_w(b)_{\bar{\dot w}}$}

Similar as in \cite[10.12]{Scholze_ECD} we define:

\begin{Def}\label{def:torsor_loc_prof}
Let $H$ be a locally profinite group and let $\ast \in \{v, {\rm proet}\}$. A (sheaf) \emph{$\underline{H}$-torsor for the $\ast$-topology} is a map $Y \rar X$ of $v$-sheaves on $\Perf_\kappa$, with an action of $\underline{H}$ on $Y$ over $X$ such that $\ast$-locally on $X$, we have $\underline H$-equivariant isomorphism $Y \cong \underline H \times X$.
\end{Def}

In particular, for $\ast \in \{v, {\rm proet}\}$, an $\ast$-torsor is a $\ast$-surjection. 

\begin{prop}\label{prop:arc_surj_of_covers}
Let $\dot w \in LF_w(\obF)$ with image $\bar{\dot w} \in (LF_w/\ker\kappa_w)(\obF)$. Then $\dot X_{\dot w}(b) \rar X_w(b)_{\bar{\dot w}}$ is $\underline{T_w(k)}$-torsor for the pro-\'etale topology.
% Then $\dot X_{\dot w}(b) \rar X_w(b)_{\bar{\dot w}}$ is $\underline{T_w(k)}$-torsor for the $v$-topology. If $\charac k > 0$, it is even a $\underline{T_w(k)}$-torsor for the pro-\'etale topology.
\end{prop}

\begin{proof}
One checks that $\dot X_{\dot w}(b) \times \underline{T_w(k)} \rar \dot X_{\dot w}(b) \times_{X_w(b)_{\bar{\dot w}}} \dot X_{\dot w}(b)$, $(\dot g,t) \mapsto (\dot g, \dot gt)$ is an isomorphism. We now show that $\dot X_{\dot w}(b) \rar X_w(b)_{\bar{\dot w}}$ is surjective for the $v$-topology (resp. pro-\'etale topology if $\charac k > 0$). Let $R \in \Perf_{\obF}$ and $g \in X_w(b)_{\bar{\dot w}}(R)$. Replacing $R$ by a $v$-cover if necessary, we may lift $g$ to some $\dot{g} \in LG(R)$ (Theorem \ref{thm:arc_exactness_and_splitting}). If $\charac k > 0$, the same works for the \'etale topology by Remark \ref{rem:Bouthier_Cesnavicius}. As in \S\ref{sec:proof_of_nat_class_map}, $\dot g^{-1}b\sigma(g) \in L(BwB)(R)$, hence determines by Lemma \ref{lm:BwB_to_Fw_map} an element $\underline{\dot w}_{\dot g} \in LF_w(R)$. By Lemma \ref{lm:aux_factorization_rel_pos}, the image of $\dot g$ in $L(G/U)(R)$ will lie in $\dot X_{\dot w}(b)(R)$ if and only if 
\begin{equation}\label{eq:dot_w_must_be_constant_on_SpecR}
\text{$\underline{\dot w}_{\dot g}$ equals the image of $\dot w \in LF_w(\obF)$ under $LF_w(\obF) \rar LF_w(R)$.}
\end{equation}
It now suffices to show that after replacing $R$ by a pro-\'etale cover, our fixed lift $\dot g$ can be replaced by $\dot g' = \dot g t$ for some $t \in LT(R)$, such that $\dot g'$ satisfies \eqref{eq:dot_w_must_be_constant_on_SpecR}.

Let $(LF_w)_{\bar{\dot w}} = LF_w \times_{(LF_w/\ker\kappa_w)} \underline{\{\bar{\dot w}\}}$. This is a trivial torsor under the group $\ker \kappa_w$ and $\ker \kappa_w \rar (LF_w)_{\bar{\dot w}}$, $\tau \mapsto \tau \dot w$ is an isomorphism (of sheaves on $\Perf_{\obF}$). As $g \in X_w(b)_{\bar{\dot w}}(R)$, we have $\underline{\dot w}_{\dot g} \in (LF_w)_{\bar{\dot w}}(R)$. Replacing $\dot g$ by $\dot g t$ with $t \in LT(R)$ has the effect of replacing $\underline{\dot w}_{\dot g}$ by $t^{-1}\sigma_w(t)\underline{\dot w}_{\dot g}$. By Proposition \ref{prop:surj_Lang_LT_onto_ker_kappa}, $t \mapsto t^{-1}\sigma_w(t) \colon LT_w \rar \ker\kappa_w$ is surjective for the \'etale topology on $\Perf_{\obF}$, hence replacing $R$ by an \'etale cover, we may find some $t\in LT(R)$ such that $t^{-1}\sigma_w(t) \underline{\dot w}_{\dot g} = \dot w$, so that replacing $\dot g$ by $\dot g' = \dot g t$, we achieve $\underline{\dot w}_{\dot g'} = t^{-1}\sigma_w(t) \underline{\dot w}_{\dot g} = \dot w$, i.e., \eqref{eq:dot_w_must_be_constant_on_SpecR} holds for $\dot g'$. 

We are done in the case $\charac k > 0$. Finally, an unpublished result of Gabber \cite{Gabber_unpub} states (in particular) that separated \'etale morphisms descend along universally submersive maps between schemes (and hence along $v$-covers). Using this and the fact that the $v$-torsor $X_{\dot w}(b) \rar X_w(b)_{\bar{\dot w}}$ trivializes itself, it follows that $\dot X_{\dot w}(b) \rar X_w(b)_{\bar{\dot w}}$ is a pro-\'etale torsor also when $\charac k = 0$.
\qedhere

\end{proof}

% \begin{rem}
% If $\charac k = 0$, we also expect $\dot X_{\dot w}(b) \rar X_w(b)_{\bar{\dot w}}$ to be surjective for the pro-\'etale topology. With some additional quasi-compactness assumption (e.g., if $T_w(k)$ would be profinite and not just locally so), this would follow from the first claim of Proposition \ref{prop:arc_surj_of_covers} by effective $v$-descent for quasi-compact \'etale morphisms \cite[Thm.~5.17]{Rydh_10}, but in our situation this is not sufficient to conclude.
% \end{rem}

\section{Variants}\label{sec:variants}

\subsection{Another presentation of $X_w(b)$} \label{sec:var1}
Let the notation be as in the beginning of \S\ref{sec:lDLS}. Let $b \in G(\breve k)$ and consider the automorphism $\sigma_b$ of $LG$ given by $\sigma_b(g) = b\sigma(g)b^{-1}$. In contrast to the classical theory, $\sigma_b$ needs not to be the geometric Frobenius corresponding to an $\mathbb{F}_q$-rational structure on $LG$.\footnote{Let $G = \GL_2$, $b = \diag(\varpi,1)$, and let $\sigma_b(g) = b\sigma(g)b^{-1}$ be an automorphism of $LG$. One can show that there is no ind-scheme $H$ over $\bF_q$, such that $LG \cong H \otimes_{\bF_q} \obF$ and $\sigma_b$ is the corresponding geometric Frobenius.} For $b \in G(\breve k)$, $w \in W$, resp. its lift $\dot w \in N_G(T)(\breve k)$, we may consider the functors $S_w(b)$, $\dot S_{\dot w}(b)$ on $\Perf_{\obF}$ defined by Cartesian diagrams

\centerline{\begin{tabular}{cc}
\begin{minipage}{2in} 
\begin{displaymath}
\leftline{
\xymatrix{
S_w(b) \ar[r] \ar[d] & \dot w b^{-1} \sigma_b(LB) \ar[d]\\
LG \ar[r] & LG 
}
}
\end{displaymath}
\end{minipage}
& \qquad\qquad  and \qquad\qquad
\begin{minipage}{2in}
\begin{displaymath}
\leftline{
\xymatrix{
\dot S_{\dot w}(b) \ar[r] \ar[d] & \dot w b^{-1} \sigma_b(LU) \ar[d]\\
LG \ar[r] & LG, 
}
}
\end{displaymath}
\end{minipage}
\end{tabular}
}

\smallskip
\noindent where the lower horizontal maps are $g \mapsto g^{-1}\sigma_b(g)$. Note that in the left diagram the upper left entry only depends on $w$, not on $\dot w$. As $LU \rar LB \rar LG$ are closed immersions, $\dot S_{\dot w}(b) \rar S_w(b) \rar LG$ are closed sub-ind-schemes. The closed subgroup $L(B \cap wBw^{-1})$ of $LG$ acts on $LG$ by right multiplication. By checking on $R$-points (as in \S\ref{sec:def_DLV}) one sees that this action restricts to an action of $L(B \cap wBw^{-1})$ on $S_w(b)$ and to an action of $L(U \cap w U w^{-1})$ on $\dot S_{\dot w}(b)$. The following proposition can be regarded as an analog of \cite[1.11]{DeligneL_76}.

\begin{prop}\label{prop:another_formulation}
There are natural isomorphisms
\begin{align*} 
X_w(b)' &:= S_w(b) / L(B \cap wB w^{-1}) \stackrel{\sim}{\longrar} X_w(b) \\
\dot X_{\dot w}(b)' &:= \dot S_{\dot w}(b) / L(U \cap wU w^{-1}) \stackrel{\sim}{\longrar} \dot X_{\dot w}(b), \\
\end{align*}
where the quotients are taken in the category of arc-sheaves on $\Perf_{\obF}$.
\end{prop}
\begin{proof}
We only prove the first isomorphism, the second has an analogous proof. We have the composed map $S_w(b) \rar LG \rar L(G/B)$, and we claim that it factors through the natural monomorphism $X_w(b) \rar L(G/B)$. By Lemma \ref{lm:aux_factorization_rel_pos}, this is a computation on $R$-points. We have constructed a map $S_w(b) \rar X_w(b)$. One checks that it factors through the presheaf quotient $Q := (S_w(b)/L(B \cap wB w^{-1}))^{\rm presh}$ and, as $X_w(b)$ is an arc-sheaf, also through the arc-sheafification $X_w(b)'$. This gives the desired map $X_w(b)' \rar X_w(b)$ of arc-sheaves. Its surjectivity follows from Theorem \ref{thm:arc_exactness_and_splitting}, along with Lemma \ref{lm:aux_factorization_rel_pos}. To show injectivity, first observe that it suffices to show that $Q \rar X_w(b)$ is injective. Let $\alpha \colon S_w(b) \rar X_w(b) \har L(G/B)$ be the natural map. It suffices to show that if $R \in \Perf_{\obF}$ and $s_1,s_2 \in S_w(b)(R)$ satisfy $\alpha(s_1) = \alpha(s_2)$, then $\exists \gamma \in L(B \cap wBw^{-1})(R)$ such that $s_1\gamma = s_2$. The map $\alpha$ is induced by the map $\alpha_0 \colon S_w(b) \har LG \rar (LG/LB)^{\rm presh}$ (presheaf quotient) by composition with $(LG/LB)^{\rm presh} \har LG/LB \har L(G/B)$. As the latter two maps are injective (this holds as the quotient presheaf is always separated and by Theorem \ref{thm:arc_exactness_and_splitting}), we already have $\alpha_0(s_1) = \alpha_0(s_2)$. From this we get some $\gamma \in LB(R)$ with $s_1\gamma = s_2$. But now $s_i \in S_w(b)(R)$, i.e., $s_i^{-1}b\sigma(s_i) \in wLB(R)$ for $i =1,2$, and hence ($\dot w$ is any lift of $w$)
\[
\dot w LB(R) \ni s_2^{-1}b\sigma(s_2) = \gamma^{-1} s_1^{-1}b\sigma(s_1)\sigma(\gamma) \in \gamma^{-1} \dot w LB(R) \sigma(\gamma) = \dot w (\dot w^{-1}\gamma^{-1}\dot w)LB(R). 
\]
I.e., $\gamma$ must lie in $LB(R) \cap wLB(R)w^{-1} = L(B\cap wBw^{-1})(R)$, and we are done.
\end{proof}

\begin{ex}\label{ex:regular_b_example_GL2}
In the classical setup \cite[1.19]{DeligneL_76}, replacing $b$ by a $\sigma$-conjugate element, one can always achieve $\dot w = b$. This is not the case in our setup, already for $G = \GL_2$: Let $T$ the diagonal torus, $b = \diag(\varpi^c,\varpi^d)$ with $c \neq d \in \bZ$, $w$ the nontrivial element of the Weyl group of $T$. Then there is no representative of the $\sigma$-conjugacy class $[b]_G$ lying over $w$, and $X_w(b)$ is not of the form as considered in \S\ref{sec:var2} below.
\end{ex}

\subsection{Variant of the construction} \label{sec:var2} Let $G'$ be a reductive $k$-group, which splits over $\breve k$. Let $B' = T'U'$ be a $\breve k$-rational Borel subgroup and a $\breve k$-split maximal torus contained in it. Let $F$ denote the geometric Frobenius of $LG'$ (regarded as a sheaf on $\Perf_{\obF}$). 
We then can consider the functor $S_{T',B'} = S_{T',B'}^{G'}$ (resp. $S_{T',U'} = S_{T',U'}^{G'}$) on $\Perf_{\obF}$, defined as the pull-back of $F(LB')$ (resp. $F(LU')$) under the Lang map $\Lang_{G'} \colon LG' \rar LG'$, $g \mapsto g^{-1}F(g)$.
As in \S\ref{sec:var1}, $S_{T',U'} \subseteq S_{T',B'} \subseteq LG'$ are closed sub-ind-schemes, and $LB' \cap F(LB')$ resp. $LU' \cap F(LU')$ acts on $S_{T',B'}$ resp. on $S_{T',U'}$ by right multiplication. We can therefore consider the arc-sheaf quotients
\begin{align*}
X_{T',B'} &= S_{T',B'} /  LB' \cap F(LB') \\
X_{T',U'} &= S_{T',U'} /  LU' \cap F(LU') \\
\end{align*}
(we write $X_{T',B'}^{G'}$ resp. $X_{T',U'}^{G'}$ to specify $G'$). 
Then $\underline{G'(k)}$ acts on $S_{T',B'}$ and on $X_{T',B'}$ by left multiplication, and $\underline{G'(k)} \times \underline{T'(k)}$ acts on $S_{T',U'}$ and $X_{T',U'}$ by $(g,t) \colon x \mapsto gxt$ (this is shown similar as in \S\ref{sec:def_DLV}). The natural map $S_{T',U'} \rar S_{T',B'}$ induces a map $X_{T',U'} \rar X_{T',B'}$, and both of these maps are $\underline{G(k)}$-equivariant.

The spaces $X_{T',B'}$, $X_{T',U'}$ are related to $X_w(b)$, $\dot X_{\dot w}(b)$. With notation as in the beginning of \S\ref{sec:lDLS}, for $b \in N_G(T)(\breve k)$ we have the group $G_b$ as in \eqref{eq:form_of_Levi}. Its base change to $\breve k$ (again denoted $G_b$) may be identified with the centralizer of the Newton point $\nu_b$ of $b$, a closed subgroup of $G$ (cf. \cite[3.3]{Kottwitz_97}), which in our situation is a Levi factor of $G$. Let $T_b$ be the torus $T$ regarded as a subgroup of $G_b$. It is defined over $k$, as $b \in N_G(T)(\breve k)$. We also have the Borel subgroup $B_b := B \cap G_b$ of $G_b$ containing $T_b$. The geometric Frobenius of $LG_b$ is $\sigma_b(g) := b\sigma(g)b^{-1}$, where $\sigma$ is the geometric Frobenius on $LG$. 

\begin{lm}
There are injections of arc-sheaves $X_{T_b,B_b}^{G_b} \rar X_w^G(b)$ and $X_{T_b,U_b}^{G_b} \rar \dot X_b^G(b)$, where $w$ is the image of $b$ in $W$. Moreover, they are isomorphisms if $b$ is basic.
\end{lm}

\begin{proof}
The proof of the first claim is analogous to the proof of Proposition \ref{prop:another_formulation}. If $b$ is basic, then $G_b=G$ (as $\breve k$-groups) and surjectivity follows from Proposition \ref{prop:another_formulation}.
\end{proof}

Not for every $[b]_G$ there exists such a map into $X_w(b')$ for some $b' \in [b]_G$, cf. Example \ref{ex:regular_b_example_GL2}.

\section{Examples}\label{sec:examples}
Let the notation be as in the beginning of \S\ref{sec:lDLS}. For $R \in \Perf_{\obF}$, we denote by $\phi \colon \bW(R) \rar \bW(R)$ the unique lift $\sum_n [x_n]\varpi^n \mapsto \sum_n [x_n^q]\varpi^n$ of the $q$-power Frobenius map of $R$. On the other side, the letter $\sigma$ is reserved for the geometric Frobenius automorphism of sheaves over $\Spec \obF$.

\subsection{Unramified tori}
Let $G_0$ be an unramified $k$-torus. Then $W = 1$, $L(G/B) = \Spec \obF$ and $L(G/U) = LG$. As in \S\ref{sec:loop_groups_of_unram_tori} we have the Kottwitz map,
\[
\kappa_{G_0} \colon LG(\obF) = G(\breve k) \tar B(G_0) \cong X_\ast(G_0)_{\langle \sigma \rangle}.
\]

\begin{prop}\label{prop:ex_tori}
Let $G_0$ be an unramified torus. For any $b \in G(\breve k)$, we have $X_1(b) = \Spec \obF$. The image of $\alpha_{1,b}$ is $[b]_{G_0}\in B(G_0) = (LG/\ker\kappa_{G_0})(\obF)$, i.e., for any $b, \dot w \in G(\breve k)$, we have 
\[
\dot X_{\dot w}(b) \neq \varnothing \LRar \kappa_{G_0}(\dot w) = \kappa_{G_0}(b).
\]
Assume this is the case. Then $\dot X_{\dot w}(b)$ is $\underline{G_0(k) \times G_0(k)}$-equivariantly isomorphic to 
\[
\dot X_{\dot w}(\dot w) = \underline{G_0(k)}
\]
with $\underline{G_0(k) \times G_0(k)}$-action by $(g,t) \colon x \mapsto gxt$.
\end{prop}
\begin{proof}
The first statement is clear as $L\caO(1) = L(G/B)^2 = \Spec \obF$. For the second and third statements assume $\dot w$, $b$ are such that $\dot X_{\dot w}(b) \neq \varnothing$. Then $\dot X_{\dot w}(b)$ has some geometric point, i.e., $\dot X_{\dot w}(b)(\ff) \neq \varnothing$ for some algebraically closed field $\ff \in \Perf_{\obF}$. But
\[
\dot X_{\dot w}(b)(\ff) = \{ g \in G(\bW(\ff)[1/\varpi]) \colon g^{-1}b\sigma(g) = \dot w \}.
\]
This can only be non-empty if $\kappa_{G_0}(b) = \kappa_{G_0}(\dot w)$. Suppose this holds. By Remark \ref{rem:change_b} we may replace $b$ by a $\sigma$-conjugate element, e.g. $\dot w$, without changing $\dot X_{\dot w}(b)$ (up to an equivariant isomorphism). So it remains to compute $\dot X_{\dot w}(\dot w)$. Towards this, note that $\caO(\dot w) \cong G$ via $(x,\dot w x) \mapsto x$, and hence from Definition \ref{def:Xwb_and_covers} we deduce $\dot X_{\dot w}(\dot w) = (LG \doublerightarrow{\id}{\sigma} LG)$. Corollary \ref{cor:morphisms_from_cont_locprofin_into_loop_group} gives now a map $\underline{G_0(k)} \rar LG$, which factor through this equalizer, i.e. we get a map $\underline{G_0(k)} \rar \dot X_{\dot w}(\dot w)$, which is an isomorphism by the same argument as in the proof of Proposition \ref{lm:inclusion_into_LGbsigma}.
\end{proof}

\subsection{Varieties $X_w(b)$ for $\GL_2$}\label{sec:expl_GL2} We list all essentially different possibilities for $X_w(b)$ in the case $G_0 = \GL_2$. Let $T_0 \subseteq B_0 \subseteq G_0$ be the diagonal torus and upper triangular Borel subgroup, and let $T,B,G$ be the base changes to $\breve k$. Let $w_0$ denote the non-trivial element of the Weyl group. The following elements form a minimal system of representatives of all $\sigma$-conjugacy classes in $G(\breve k)$: $b = \varpi^{(c,c)}$ with $c \in \bZ$, $b = \left(\begin{smallmatrix} 0 & \varpi^c \\ \varpi^{c+1} & 0 \end{smallmatrix}\right)$ with $c \in \bZ$, $b = \varpi^{(c,d)}$ with $c > d$. The first two types are basic, the second is superbasic; in the first (resp. second, resp. third) case $G_b = G_0$ (resp. $G_b(k) = $ units of the quaternion algebra over $k$, resp. $G_b = T_0$). Let $Z \subseteq G(k)$ be the center, and consider the $\caO_{\breve k}$-scheme 
\[ 
\Omega^1_{\caO_{\breve k}} = \Spec \caO_{\breve k}[T]_{T-T^q}. 
\]

\begin{thm}\label{thm:GL2varieties}
For $G = \GL_2$, all $X_w(b)$ are schemes. With the obvious modifications (replace $\bQ_p$, $\bZ_p$, $\breve \bZ_p$, $p$, $\Omega_{\breve \bZ_p}^1$ by $k$, $\caO_k$, $\caO_{\breve k}$, $\varpi$, $\Omega^1_{\caO_{\breve k}}$) they are listed in Table \ref{tab:1} in the introduction.
\end{thm}

\begin{proof}
For $w=1$, everything easily follows from Theorem \ref{cor:disjoint_dec_of_Xwb_complete}. For $w = w_0$, $X_w(b)$ is by definition a subsheaf of $L(G/B) \cong L\bP^1$. For any $R \in \Perf_{\obF}$ with $\widetilde R := \bW(R)[1/\varpi]$ such that $L(\bA^2 \sm \{0\})(R) \rar L\bP^1(R)$ is surjective, we have
\[
L\bP^1(R) = \{(x,y) \in \widetilde R^2 \colon x,y \text{ generate $\widetilde R$ as $\widetilde R$-module} \}/\widetilde R^\times
\]
Denote the class of $(x,y)$ by $[x:y]$. By definition, $[x:y] \in L\bP^1(R)$ lies in $X_w(b)(R)$ if and only if the induced map $([x:y], b\sigma[x:y]) \colon \Spec \widetilde R \rar \bP^1 \times \bP^1$ factors through the open subscheme $(\bP^1 \times \bP^1) \sm \Delta$, where $\Delta$ is the diagonal of $\bP^1_k$. One checks that this is the case if and only if
\begin{align*} 
% \phi(x)y - x\phi(y) \in \bW(R)^\times & \qquad \text{if $b=1$, resp.}  \\ 
\varpi^c\phi(x)y - \varpi^dx\phi(y) \in \widetilde R^\times & \qquad \text{if $b=\varpi^{(c,d)}$ with $c\geq d$, \qquad \qquad  resp.} \\
\varpi^{c+1} x \phi(x) - \varpi^c y\phi(y) \in \widetilde R^\times & \qquad \text{if $b = \left(\begin{smallmatrix} 0 & \varpi^c \\ \varpi^{c+1} & 0 \end{smallmatrix}\right)$}. 
\end{align*} 

In what follows we use Lemmas \ref{lm:properties_of_ord}, \ref{lm:units_Witt} without further reference. First suppose $b = \varpi^{(c,d)}$ with $c>d$. Then for all $(x,y) \in L(\bA^2 \sm \{0\})(R)$, 
\begin{equation} %\label{eq:reguar_case_upper_cont_fct_equality} 
{\rm ord}_\varpi(\varpi^c\phi(x)y - \varpi^dx\phi(y)) = d + {\rm ord}_\varpi(x\phi(y)) = d + {\rm ord}_\varpi(x) + {\rm ord}_\varpi(y).
\end{equation}
If $[x:y] \in L\bG_m(R)$, then $x,y \in \widetilde R$ are units, so the right hand side is a locally constant function on $\Spec R$, and it follows that the left hand side also is, i.e., $[x:y] \in X_w(b)(R)$. Conversely, if $[x:y] \in X_w(b)(R)$, then the left hand side is locally constant, hence the right side also is. Hence $\ord_\varpi(x) = f - \ord_\varpi(y)$ for a locally constant $f$, i.e., $\ord_\varpi(x)$ is both upper and lower semi-continuous, hence continuous, i.e., locally constant. This implies that $x$ is a unit, and similarly we see that $y$ is a unit, i.e., $[x:y] \in L\bG_m(R)$. All this shows $X_w(b)(R) = L\bG_m(R)$ for all $R$ such that $L(\bA^2 \sm \{0\})(R) \rar L\bP^1(R)$ is surjective. As both, $X_w(b)$ and $L\bG_m$ are arc-subsheaves of $L\bP^1$, Corollary \ref{cor:projective_space_as_quotient} suffices to conclude that $X_w(b) = L\bG_m$.

Next, suppose $b = \left(\begin{smallmatrix} 0 & \varpi^c \\ \varpi^{c+1} & 0 \end{smallmatrix}\right)$. One checks that $X_w(b)$ does not change if $b$ is multiplied by a central element of $G(\breve k)$, so we may assume $c=0$. For all $(x,y) \in L(\bA^2 \sm \{0\})(R)$, $\ord_\varpi(\varpi x\phi(x)) \not\equiv \ord_\varpi(y\phi(y)) \,{\rm mod}\, 2$ at any point of $\Spec R$. Thus, by triangle inequality, 
\[
{\rm ord}_\varpi(\varpi x \phi(x) - y\phi(y)) = {\rm min}\{2{\rm ord}_\varpi(x), 2{\rm ord}_\varpi(y) + 1 \}.
\]
For $[x:y] \in L\bP^1(R)$ we may consider the disjoint decomposition $\Spec R = U_0 \cup U_1$, where 
\begin{align*}
U_0 &= \{s \in \Spec R \colon {\rm ord}_\varpi(x)(s) \leq {\rm ord}_\varpi(y)(s) \} \\
U_1 &= \{s \in \Spec R \colon {\rm ord}_\varpi(y)(s) < {\rm ord}_\varpi(x)(s) \} 
\end{align*}
(note that this is well defined, as multiplication of $x$ and $y$ by the same unit does not change the function $\ord(y) - \ord(x)$). Now suppose that $[x:y] \in X_w(b)(R)$. Then $U_0$ and $U_1$ are both open and closed. Indeed, in this case the function ${\rm min}\{2{\rm ord}_\varpi(x), 2{\rm ord}_\varpi(y) + 1 \}$ is locally constant on $\Spec R$ and $U_0$ (resp. $U_1$) is the preimage of $2\bZ$ resp. $2\bZ + 1$ under it. Now, define subsheaves $X_w(b)_0$ (resp. $X_w(b)_1$) of $X_w(b)$ by taking $X_w(b)_0(R) = \{ [x:y] \in X_w(b)(R) \colon \ord_\varpi(x) \leq \ord_\varpi(y) \}$ (resp. same formula with $>$ instead of $\leq$) for all $R$ as above. For any point $[x:y] \in X_w(b)(R)$, the pull-back $X_w(b)_i \times_{X_w(b)} \Spec R$ is $U_i$ ($i=0,1$). Thus $X_w(b)_0$, $X_w(b)_1$ are open and closed subfunctors covering $X_w(b)$. It remains to determine $X_w(b)_i$. On $X_w(b)_0$ (resp. $X_w(b)_1$), $\ord_\varpi(x)$ (resp. $\ord_\varpi(y)$) is locally constant, i.e., $x$ (resp. $y$) is a unit, and 
\[ 
X_w(b)_0 \cong L^+\bA^1_{\caO_k}, \quad [x:y] = [1:y/x] \mapsto y/x
\]
and similarly for $X_w(b)_1$.

Finally, let $b = \varpi^{(c,c)}$. As before, we may assume that $b=1$. Let $\dot w = \left(\begin{smallmatrix} 0 & 1 \\ 1 & 0 \end{smallmatrix}\right)$ be a lift of $w$. As we will see below (Lemma \ref{lm:alphawb_for_GL2_w0}), $\alpha_{w,b}$ factors through one point, and hence $\dot X_{\dot w}(1) \rar X_w(1)$ is surjective for the $v$-topology by Proposition \ref{prop:arc_surj_of_covers}. Let $[x:y] \in X_w(b)(R)$. After replacing $R$ by a $v$-cover, it lifts to a section of $\dot X_{\dot w}(1)$, which may be represented by $(x,y) \in \widetilde R^2$ satisfying $x\varphi(y) - \varphi(x)y = 1$.  
It follows from the description in \cite[Prop.~2.6]{CI_loopGLn} that for any $(x,y) \in \dot X_{\dot w}(1)(R)$, the functions $\ord_\varpi(x)$, $\ord_\varpi(y)$ are locally constant, and hence the same holds for any $[x:y] \in X_w(1)(R)$. It follows that $\ord_\varpi(yx^{-1})$ and $\ord_\varpi(\phi(yx^{-1}) - yx^{-1})$ are locally constant. Hence the locus in $\Spec R$, defined by $\ord_\varpi(yx^{-1}) \geq 0$ and $\ord_\varpi(\phi(yx^{-1}) - yx^{-1}) = 0$ is open and closed. We get the corresponding open and closed subfunctor 
\[
X_w(b)_0 = \left\{[x:y] \in L\bP^1 \colon \text{$x$ invertible, ${\rm ord}_{\varpi}(yx^{-1}) \geq 0$ and ${\rm ord}_\varpi(\phi(yx^{-1}) - yx^{-1}) = 0$} \right\} 
\] of $X_w(b) \subseteq L\bP^1$. (Note that the inequality ${\rm ord}_{\varpi}(yx^{-1}) \geq 0$ here is automatically an equality.) The locus in $L\bP^1$ given by the first of these two conditions is represented by $L^+\bA^1_{\caO_{\breve k}}$, and $X_w(b)_0 \cong L^+\Omega_{\caO_{\breve k}}^1$ is an open subscheme of it. 

It is clear that $X_w(b)_0$ is stabilized by the action of the subgroup $ZG_0(\caO_k) \subseteq G_0(k)$ and its $G_0(k)$-translates cover $X_w(b)$ (as follows by surjectivity of $\dot X_{\dot w}(b) \rar X_w(b)$ and \cite[Prop.~2.6]{CI_loopGLn}).  \qedhere
\end{proof}

Now we list the maps $\alpha_{w,b}$ (\S\ref{sec:Torsors}) and the coverings $\dot X_{\dot w}(b)$. By Lemma \ref{lm:Xdot_changes_nicely}, for all $\dot w \in LF_w(\obF)$ in the preimage of some $\bar{\dot w} \in (LF_w/\ker\kappa_w)(\obF)$, all $\dot X_{\dot w}(b)$ are mutually isomorphic. So it is enough to describe just one such. For $w=1$ we have $LF_1/\ker \kappa_1 = \underline{X_\ast(T)} \cong \underline{\bZ^2}$ ($(c,d) \in \bZ^2$ corresponds to the image of $\varpi^{(c,d)} \in LF_1/\ker\kappa_1$). According to Theorem \ref{thm:GL2varieties} we have two cases:

\begin{itemize}
\item \underline{$b = \varpi^{(c,c)}$, $w = 1$}. Then $\alpha_{1,b} \colon \underline{\bP^1(k)} \rar \underline{\bZ^2}$ factors through the point $\underline{\{(c,c)\}} \rar \underline{\bZ^2}$, i.e., for $\dot w \in F_1(\breve k)$, $\dot X_{\dot w}(\varpi^{(c,c)}) \neq \varnothing \LRar \dot w \in \varpi^{(c,c)}T(\caO_{\breve k})$. Moreover, $X_1(1) \cong \underline{(G/U)(k)}$.
\item \underline{$b=\varpi^{(c,d)}$ ($c > d$), $w = 1$}. Then $\alpha_{1,b} \colon \{0,\infty\} \rar \underline{\bZ^2}$ maps $0$ to $(c,d)$ and $\infty$ to $(d,c)$. This corresponds to the decomposition $[b]_G \cap T(\breve k) = [b_1]_T \, \dot\cup \,[b_2]_T$ with $b_1 =b$, $b_2= \diag(\varpi^d,\varpi^c)$. For $\dot w = \varpi^{(c,d)}$, $\dot X_{\dot w}(b) \cong T_0(k)$ according to Theorem \ref{cor:disjoint_dec_of_Xwb_complete} and Proposition \ref{prop:ex_tori}, and similarly in the other case.
\end{itemize}

For $w = w_0$, $X_\ast(T_{w_0}) \cong \bZ^2$ with $\sigma_{w_0}$ acting by $(c,d) \mapsto (d,c)$, hence $X_\ast(T_{w_0})_{\langle \sigma_{w_0} \rangle} \cong \bZ$, and $LF_{w_0}/\ker\kappa_{w_0} \cong \underline{\bZ}$, induced by $\left(\begin{smallmatrix} 0 & x \\ y & 0 \end{smallmatrix}\right) \mapsto \ord_\varpi(xy)$. Concerning $\alpha_{w_0,b}$ we have: 
\begin{lm}\label{lm:alphawb_for_GL2_w0}
For each $b$, $\alpha_{w_0,b} \colon X_{w_0}(b) \rar \underline{\bZ}$ factors through the point $\underline{\{\ord_\varpi(\det(b))\}} \rar \underline\bZ$.
\end{lm}
\begin{proof} Using Lemma \ref{lm:aux_factorization_rel_pos} for $L(G/U)$, and that $\det(u) = 1$ for any $R \in \Perf_{\obF}$ and any $u\in LU(R)$, the lemma follows by comparing $\ord_\varpi(\det(g^{-1}b\sigma(g)))$ with $\ord_\varpi(\det(\dot w))$ for a lift $\dot w$ of $w_0$ and applying Proposition \ref{prop:natural_classifying_map}(ii).
\end{proof}

Thus, for $b$ fixed, and $\dot w$ varying through lifts of $w_0$, all non-empty $X_{\dot w}(b)$ are mutually isomorphic. For $b$ basic, see \S\ref{sec:expl_GLn_basic_coxeter}. For $b=\varpi^{(c,d)}$ with $c > d$ we describe $\dot X_{\dot w}(b)$ here: 

\begin{itemize}
\item \underline{$b=\varpi^{(c,d)}$ ($c>d$) , $w = w_0$}. Let $\dot w = \left(\begin{smallmatrix} 0 & -\varpi^{\alpha}\\ \varpi^\beta & 0 \end{smallmatrix}\right)$ with $\alpha + \beta = c+d$. Let $a = y/x$ be a fixed coordinate on $X_{w_0}(b) = L\bG_m \subseteq L\bP^1$ (as in the proof of Theorem \ref{thm:GL2varieties}). Let $\tau$ be a coordinate on a second $L\bG_m$. Then $X_{\dot w}(b)$ is isomorphic to the subscheme of $L\bG_m^2$ given by the equation 
\[\tau^{-1} \sigma^2(\tau) = (\varpi^{d-\beta}\sigma(a) - \varpi^{c-\beta}a)^{-1}\sigma(\varpi^{d-\beta}\sigma(a) - \varpi^{c-\beta}a). \]
The (left) action of $\underline{G_b(k)} = \underline{T_0(k)}$ is given by $\diag(t_1,t_2).(a,\tau) = (t_1^{-1}t_2a, t_1\tau)$. The (right) action of $\underline{T_w(k)} \cong \underline{k_2^\times}$ (here $k_2/k$ denotes the unramified extension of degree $2$) is given by $(a,\tau).\lambda = (a,\tau\lambda)$.
\end{itemize}

\subsection{Case $G_0 = \GL_n$, $w$ Coxeter, $b$ basic}\label{sec:expl_GLn_basic_coxeter}
% The spaces $\dot X_{\dot w}(b)$ were studied in detail in \cite{CI_ADLV, CI_loopGLn}. 
Let $G_0 = \GL_n$, $b$ basic and $w$ Coxeter. Then $X_w(b)$ is independent of the choice of the Coxeter element $w$ by Corollary \ref{cor:isomorphism_class_depends_only_on_Cmin}. The non-empty covers $\dot X_{\dot w}(b)$ are all isomorphic (for varying $\dot w$), by the same argument as in Lemma \ref{lm:alphawb_for_GL2_w0}. For a special Coxeter element $w$, $\dot X_{\dot w}(b)$ were studied in detail in \cite{CI_loopGLn}.

As an example, in the case $b=1$ we describe the scheme $\dot X_{\dot w}(1)$, with $\dot w$ arbitrary with $\ord_\varpi(\det(\dot w)) = 0$, up to a $\underline{G_b(k)\times T_w(k)}$-equivariant non-canonical isomorphism. We have $\dot X_{\dot w}(1) \cong \coprod_{G(k)/G(\caO_k)} X_\caO$, where $X_\caO$ is a locally closed subscheme of $L^+\bA^n_{\caO_{\breve k}}$, 
\[
X_\caO = \{x \in L^+\bA^n_{\caO_{\breve k}} \colon \det g_n(x) \in L^+\bG_m \text{ and } \sigma(\det g_n(x)) = (-1)^{n-1}\det g_n(x) \},
\]
where $g_n(x)$ is the $n\times n$-matrix, whose $i$-th column is $\sigma^{i-1}(x)$. For general basic $b$, the non-empty $\dot X_{\dot w}(b)$ admit similar descriptions and are, in particular, schemes. For details, see \cite[Prop.~2.6, \S5.1 and \S5.2]{CI_loopGLn}.

\subsection{Case $G = \GL_3$, $b=1$ and $w$ the longest element}\label{sec:GL3_longest_w}

Let $G_0 = \GL_3$, $T$ the diagonal torus and $B$ the upper triangular Borel subgroup of $G = G_0 \times_k \breve k$, and $U$ its unipotent radical. Then $\dot w = \left(\begin{smallmatrix} 0&0&1\\0&1&0 \\ 1&0&0 \end{smallmatrix}\right) \in F_w(\breve k)$ is a lift of the longest element $w \in W$. 

Let $\caG$ be the canonical $\caO_k$-model of $G_0$, and for $\lambda \in \bZ$, let $\caG_\lambda$ be the unique (connected) hyperspecial $\caO_{\breve k}$-model of $G_0$ whose $\breve k$-points are $\varpi^{(\lambda,-2\lambda,\lambda)} \caG(\caO_{\breve k}) \varpi^{-(\lambda,-2\lambda,\lambda)}$. Let $\caU_\lambda$ be the schematic closure of $U$ in $\caG_\lambda$. Let $U_+ = \left(\begin{smallmatrix} 1&L\bG_a&L^+\bG_a\\0&1&0 \\ 0&0&1 \end{smallmatrix}\right)$ and $U_- = \left(\begin{smallmatrix} 1& 0 & L^+\bG_a\\0&1& L\bG_a \\ 0&0&1 \end{smallmatrix}\right)$ be closed sub-ind-group schemes of $LU$. We have the inner $k$-form $G_{\dot w}$ of $G$ (as in \S\ref{sec:bsigma_fixed_points}), which is $k$-isomorphic to $G_0$ as $[\dot w]_{G_0} = [1]_{G_0}$. As ${\rm Ad}\,\dot w$ stabilizes the cocharacter $(\lambda, -2\lambda, \lambda)$, $\caG_\lambda$ descends to a $k$-subgroup $\caG_{\lambda,k}$ of $G_{\dot w}$. Finally, define $\sigma_{\dot w} \colon LG \rar LG$ by $\sigma_{\dot w}(g) = \dot w \sigma(g) \dot w^{-1}$.

\begin{prop}\label{prop:ex_GL3_longest} The ind-scheme $\dot X_{\dot w}(1)$ is covered by closed $\underline{G_0(k)}$-stable sub-ind-schemes $X_{\infty}$, $X_{-\infty}$ and sub-schemes $X_\lambda$ for $\lambda \in \bZ$, each two of them intersecting non-trivially, such that, with notation as above, 
\begin{itemize}
\item[(i)] $X_{\pm\infty}$ is isomorphic to the pullback along $g \mapsto g^{-1}\sigma(g) \colon LG \rar LG$ of $\dot w U_{\pm}$
\item[(ii)] For $\lambda \in \bZ$, 

\centerline{\begin{tabular}{cc}
\begin{minipage}{2in}
\begin{displaymath}
X_\lambda \cong \coprod\limits_{g \in G_{\dot w}(k)/\caG_{\lambda,k}(\caO_k)} g X_{\lambda,\caO}
\end{displaymath}
\end{minipage}
& \qquad  where \qquad
\begin{minipage}{2in}
\begin{displaymath}
\leftline{
\xymatrix{
X_{\lambda,\caO} \ar[r] \ar[d] & \sigma_{\dot w}(L^+\caU_{\lambda}) \ar[d] \\
L^+\caG_\lambda \ar[r]^{g \mapsto g^{-1}\sigma_{\dot w}(g)} & L^+\caG_\lambda
}
}
\end{displaymath} 
\end{minipage}
\end{tabular}
}
\noindent is Cartesian.
\end{itemize}
\end{prop}

Note that $X_{\pm \infty}$ are closed sub-ind-schemes of $\dot X_{\dot w}(1)$, not representable by schemes, in accordance with Theorem \ref{thm:non_rep}. Moreover, the $1$-truncation (i.e., when $L^+$ is replaced by $L^+_1$; here for an $\caO_k$-scheme $\fX$ and $r\geq 1$, $L^+_r\fX$ is the set-valued functor on $\Perf_{\obF}$, sending $R$ to $\fX(\bW(R)/\varpi^r\bW(R))$) of $X_{\lambda,\caO}$ is isomorphic to the natural torsor over a classical Deligne--Lusztig variety attached to $\GL_3$ over $\bF_q$ and the longest element in the Weyl group.

\begin{proof}
As $U \cap wUw^{-1} = 1$, we have by Proposition \ref{prop:another_formulation}, $\dot X_{\dot w}(1) \cong \dot X_{\dot w}(1)' = \dot S_{\dot w}(1)$, the pull-back of $\dot wLU$ along ${\rm Lang} \colon LG \rar LG$, $g\mapsto g^{-1}\sigma(g)$. Thus we may replace $\dot X_{\dot w}(1)$ by $\dot S_{\dot w}(1)$. We claim that the image of $\dot S_{\dot w}(1)$ under ${\rm Lang}$ is contained in the sub-ind-scheme $\dot w \left(U_+ \cup U_- \cup \bigcup_{\lambda \in \bZ} L^+\caU_\lambda\right)$ of $\dot w LU$. This may be checked on geometric points. Let $\ff\in \Perf_{\obF}$ be an algebraically closed field and let $L = \bW(\ff)[1/\varpi]$. Suppose ${\rm Lang}$ maps $g \in \dot S_{\dot w}(1)(\ff)$ to $u = \dot w \left(\begin{smallmatrix} 1& a & c\\0&1& b \\ 0&0&1 \end{smallmatrix}\right)$ with $a,b,c \in L$. Our claim follows once we show that $\ord_\varpi(a) + \ord_\varpi(b) \geq 0$ and $\ord_\varpi(c) \geq 0$. Towards this, write $g_i$ ($1\leq i\leq 3$) for the $i$th column of $g$. Then $\sigma(g) = g\dot w u$ says that $g_3 = \sigma(g_1)$ and (eliminating $g_3$) we are left with the two equations
\begin{align}
\label{eq:cyclic_equation_GL3_longest} \sigma(g_2) &= g_2 + a\sigma(g_1) \\ 
\nonumber \sigma^2(g_1) &= g_1 + c\sigma(g_1) + bg_2.
\end{align}
Suppose first $b\neq 0$. Using the second equation we can eliminate $g_2$, and the first equation gets
\begin{equation}\label{eq:char_pol_g1_GL3_expl}
\sigma^3(g_1) - (\sigma(c) - b^{-1}\sigma(b))\sigma^2(g_1) + (b^{-1}\sigma(b)c - 1 - \sigma(b)a)\sigma(g_1) + b^{-1}\sigma(b) g_1 = 0.
\end{equation}
Moreover, as $g \in \GL_3(L)$, the columns of $g$ generate the $L$-vector space $L^3$, and this implies (using \eqref{eq:cyclic_equation_GL3_longest} and $g_3 = \sigma(g_1)$) that $g_1,\sigma(g_1),\sigma^2(g_1)$ is a basis of $L^3$. Thus $g_1$ is a generator of the $\ff$-isocrystal $(L^3,\sigma)$, and hence the $\varpi$-adic valuations of the coefficients of the characteristic polynomial \eqref{eq:char_pol_g1_GL3_expl} of $g_1$ lie over the Newton polygon of this isocrystal (cf. e.g. \cite[\S1.1]{Beazley_09}). From this one deduces $\ord_\varpi(a) + \ord_\varpi(b) \geq 0$ and $\ord_\varpi(c) \geq 0$. If $a \neq 0$, one can proceed similarly (this time eliminating $g_1$ via the first equation) to show the same inequalities. Finally, if $a=b=0$, second equation gives $\sigma^2(g_1) = g_1 + c\sigma(g_1)$, and the same argument works with the two-dimensional sub-isocrystal of $L^3$ generated by $g_1$ (it is indeed two-dimensional, as $g\in \GL_3(L)$), and shows that $\ord_\varpi(c)\geq 0$. 
% So, in any case we have $\ord_\varpi(c)\geq 0$ and either $\ord_\varpi(a) = \infty$ or $\ord_\varpi(b) = \infty$ or $\ord_\varpi(a) + \ord_\varpi(b) \geq 0$. 
This proves our claim.

Put $X_\lambda := {\rm Lang}^{-1}(\dot wL^+\caU_\lambda)$ and $X_{\pm \infty} := {\rm Lang}^{-1}(\dot w U_{\pm})$. That $\dot X_{\dot w}(1)$ is covered by $X_{\pm\infty}$ and all $X_\lambda$ follows from the above claim (stability under $\underline{G_0(k)}$ is immediate). That each two of $X_{\pm\infty}$, $X_\lambda$ intersect non-trivially is easily checked on geometric points, and (i) follows by construction. Let $\lambda \in \bZ$. Fix some $h \in LG(\obF)$ with $h^{-1}\sigma(h) = \dot w$. Then 
\[ 
X_\lambda \stackrel{\sim}{\rar} X_\lambda' := \{g \in LG \colon g^{-1}\sigma_{\dot w}(g) \in \sigma_{\dot w}(\caU_{\lambda}) \}, \quad g \mapsto h^{-1}g,
\]
transforming the $G_0(k)$-action into $G_{\dot w}(k)$-action via ${\rm Ad}\, h \colon G_{\dot w}(k) \stackrel{\sim}{\rar} G_0(k)$. The smallest $\sigma_{\dot w}$-stable subgroup of $G(\breve k)$ containing $\caU_\lambda(\caO_{\breve k})$ is $\caG_\lambda(\caO_{\breve k})$. Identify $LG_{\dot w}$ with $LG$ (only the geometric Frobenius $\sigma_{\dot w}$ is different). Consider the projection $\pi\colon LG_{\dot w} \rar LG_{\dot w}/L^+\caG_\lambda$. For $R \in \Perf_{\obF}$ and $g \in X_\lambda'(R)$, $\sigma_{\dot w}(g) \in g\sigma_{\dot w}(L^+\caU_\lambda)(R) \subseteq g L^+\caG_\lambda(R)$. Thus $\pi$ maps to $X_\lambda'$ to the discrete set $(LG_{\dot w}/L^+\caG_\lambda)^\sigma = G_{\dot w}(k)/\caG_{\lambda, k}(\caO_k)$, and (ii) follows easily.
\end{proof}

\bibliography{bib_ADLV}{}
\bibliographystyle{alpha}

\end{document}